\documentclass{amsart}
\usepackage{amsmath,amssymb}
\usepackage{dsfont}
\usepackage[normalem]{ulem}
\usepackage{hyperref}
\usepackage{enumitem}
\usepackage{xcolor}
\hypersetup{
    colorlinks,
    linkcolor={red!50!black},
    citecolor={blue!50!black},
    urlcolor={blue!80!black},
}

\numberwithin{equation}{section}

\newtheorem{theoremcounter}{theoremcounter}[section]

\theoremstyle{plain}

\newtheorem{corollary}[theoremcounter]{Corollary}
\newtheorem{lemma}[theoremcounter]{Lemma}

\newtheorem{proposition}[theoremcounter]{Proposition}
\newtheorem{theorem}[theoremcounter]{Theorem}

\theoremstyle{plain}

\theoremstyle{definition}

\newtheorem{definition}[theoremcounter]{Definition}

\theoremstyle{remark}

\newtheorem{remark}[theoremcounter]{Remark}

\usepackage[backend=bibtex,
            style=alphabetic,
            isbn=false,
            doi=false,
            url=false,
            maxnames=5,   
            maxalphanames=5,
            giveninits=true
           ]{biblatex}
\bibliography{bibliography.bib}
\usepackage{nccmath}

\usepackage{tikz}
\usetikzlibrary{matrix,arrows,calc}
\usetikzlibrary{positioning}
\usetikzlibrary{decorations.pathreplacing,decorations.markings,decorations.pathmorphing}
\usetikzlibrary{positioning,arrows,patterns}
\usetikzlibrary{cd}
\usetikzlibrary{intersections}

\tikzset{
	every loop/.style={very thick},
	comp/.style={circle,fill,black,,inner sep=0pt,minimum size=5pt},
	order bottom left/.style={pos=.05,left,font=\tiny},
	order top left/.style={pos=.9,left,font=\tiny},
	order bottom right/.style={pos=.05,right,font=\tiny},
	order top right/.style={pos=.9,right,font=\tiny},
	order node dis/.style={text width=.75cm},
	circled number/.style={circle, draw, inner sep=0pt, minimum size=12pt},
	below left with distance/.style={below left,text height=10pt},
    below right with distance/.style={below right,text height=10pt}
	}

\newcommand{\tx}{\text}
\newcommand{\thdash}{\nbd th}
\newcommand{\nbd}{\nobreakdash-\hspace{0pt}}

\newcommand{\texpdf}[2]{\texorpdfstring{#1}{#2}}

\makeatletter
\newcommand{\writelabel}[1]{#1\def\@currentlabel{#1}}
\makeatother

\newcommand{\minwidthmathbox}[2]{%
  \mathmakebox[{\ifdim#1<\width\width\else#1\fi}]{#2}%
}

\newcommand{\bbH}{\ensuremath{\mathbb{H}}}

\newcommand{\bbR}{\ensuremath{\mathbb{R}}}

\newcommand{\bbZ}{\ensuremath{\mathbb{Z}}}

\newcommand{\bbone}{\ensuremath{\mathds{1}}}

\newcommand{\cC}{\ensuremath{\mathcal{C}}}
\newcommand{\cD}{\ensuremath{\mathcal{D}}}

\newcommand{\cH}{\ensuremath{\mathcal{H}}}

\newcommand{\cO}{\ensuremath{\mathcal{O}}}

\newcommand{\cS}{\ensuremath{\mathcal{S}}}

\newcommand{\cV}{\ensuremath{\mathcal{V}}}

\newcommand{\frakg}{\ensuremath{\mathfrak{g}}}

\newcommand{\frakl}{\ensuremath{\mathfrak{l}}}

\newcommand{\fraks}{\ensuremath{\mathfrak{s}}}

\newcommand{\frakz}{\ensuremath{\mathfrak{z}}}

\newcommand{\frakA}{\ensuremath{\mathfrak{A}}}

\newcommand{\rmc}{\ensuremath{\mathrm{c}}}
\newcommand{\rmd}{\ensuremath{\mathrm{d}}}

\newcommand{\rmt}{\ensuremath{\mathrm{t}}}

\newcommand{\rmA}{\ensuremath{\mathrm{A}}}
\newcommand{\rmAp}{\ensuremath{\mathrm{A}'}}

\newcommand{\rmC}{\ensuremath{\mathrm{C}}}
\newcommand{\rmD}{\ensuremath{\mathrm{D}}}
\newcommand{\rmE}{\ensuremath{\mathrm{E}}}

\newcommand{\rmG}{\ensuremath{\mathrm{G}}}
\newcommand{\rmGp}{\ensuremath{\mathrm{G}'}}
\newcommand{\rmH}{\ensuremath{\mathrm{H}}}
\newcommand{\rmI}{\ensuremath{\mathrm{I}}}

\newcommand{\rmK}{\ensuremath{\mathrm{K}}}
\newcommand{\rmL}{\ensuremath{\mathrm{L}}}
\newcommand{\rmM}{\ensuremath{\mathrm{M}}}
\newcommand{\rmN}{\ensuremath{\mathrm{N}}}
\newcommand{\rmNp}{\ensuremath{\mathrm{N}'}}

\newcommand{\rmP}{\ensuremath{\mathrm{P}}}
\newcommand{\rmPp}{\ensuremath{\mathrm{P}'}}

\newcommand{\rmR}{\ensuremath{\mathrm{R}}}
\newcommand{\rmS}{\ensuremath{\mathrm{S}}}
\newcommand{\rmT}{\ensuremath{\mathrm{T}}}
\newcommand{\rmU}{\ensuremath{\mathrm{U}}}

\newcommand{\rmW}{\ensuremath{\mathrm{W}}}

\newcommand{\rmZ}{\ensuremath{\mathrm{Z}}}

\newcommand{\td}{\tilde}
\newcommand{\wtd}{\widetilde}
\newcommand{\ov}{\overline}

\newcommand{\wht}{\widehat}
\newcommand{\wt}{\widetilde}

\newcommand{\defcol}{\mathrel{:}}
\newcommand{\defeq}{\mathrel{:=}}

\newcommand{\condcol}{\mathrel{:}}

\newcommand{\mrelspace}[1]{\mathrel{\mspace{#1}}}

\let\rightarroworig\rightarrow
\renewcommand{\rightarrow}
  {\protect\relbar\mrelspace{-9.7mu}\rightarroworig}
\renewcommand{\hookrightarrow}
  {\protect\lhook\mrelspace{-3.1mu}\relbar\mrelspace{-11.9mu}\rightarroworig}

\let\leftarroworig\leftarrow
\renewcommand{\leftarrow}
  {\protect\leftarroworig\mrelspace{-9.7mu}\relbar}

\makeatletter
\@ifundefined{xlongrightarrow}{%

}{%

}
\makeatother

\makeatletter
\@ifundefined{xlongleftarrow}{%

}{%

}
\makeatother

\newcommand{\ra}{\rightarrow}
\newcommand{\hra}{\hookrightarrow}

\newcommand{\mto}{\mapsto}
\newcommand{\lmto}{\longmapsto}

\renewcommand{\Re}{\mathrm{Re}}
\renewcommand{\Im}{\mathrm{Im}}

\newcommand{\sgn}{\mathrm{sgn}}

\renewcommand{\ker}{\operatorname{ker}}

\newenvironment{psmatrix}{\left(\begin{smallmatrix}}{\end{smallmatrix}\right)}

\newcommand{\rT}{{\,{}^\rmt\!}}

\newcommand{\linspan}{\operatorname{span}}

\newcommand{\NN}{\ensuremath{\mathbb{N}}}
\newcommand{\ZZ}{\ensuremath{\mathbb{Z}}}

\newcommand{\RR}{\ensuremath{\mathbb{R}}}
\newcommand{\CC}{\ensuremath{\mathbb{C}}}

\newcommand{\SL}[1]{\ensuremath{\mathrm{SL}_{#1}}}
\newcommand{\SO}[1]{\ensuremath{\mathrm{SO}_{#1}}}

\newcommand{\HS}{\ensuremath{\mathbb{H}}}

\newcommand{\Ga}{\Gamma}
\newcommand{\ga}{\gamma}

\newcommand{\Ind}{\operatorname{Ind}}
\newcommand{\Res}{\operatorname{Res}}
\newcommand{\Ker}{\operatorname{Ker}}

\newcommand{\clinspan}{\mathop{\ov{\mathrm{span}}}}

\newcommand{\SAff}[1]{\mathrm{SAff}_{#1}}

\newcommand{\cusp}{\mathrm{cusp}}

\newcommand{\gen}{\mathrm{gen}}
\newcommand{\hol}{\mathrm{hol}}
\newcommand{\abs}{\mathrm{abs}}
\newcommand{\av}{\mathrm{av}}
\newcommand{\rel}{\mathrm{rel}}
\newcommand{\gr}{\mathrm{gr}}
\newcommand{\prim}{\mathrm{prim}}

\newcommand{\SV}{\operatorname{SV}}
\newcommand{\MV}{\mathrm{MV}}
\newcommand{\SC}{\operatorname{SC}}
\newcommand{\SVabs}{\SV_{\mathrm{abs}}}
\newcommand{\SVrel}{\SV_{\mathrm{rel}}}

\newcommand{\HSp}{\HS'}
\newcommand{\G}{\rmG}
\newcommand{\A}{\rmA}
\newcommand{\N}{\rmN}

\newcommand{\Gp}{\rmG'}
\newcommand{\Ap}{\rmA'}

\newcommand{\Np}{\rmN'}
\newcommand{\Hp}{\rmH'}
\newcommand{\Zp}{\rmZ'}

\newcommand{\Gap}{\Ga'}

\newcommand{\fraksl}[1]{\fraks\frakl_{#1}}
\newcommand{\frakgp}{\frakg'}

\newcommand{\frakzp}{\frakz'}

\newcommand{\htG}{\wht{\rmG}}
\newcommand{\htGp}{\wht{\rmG}{\vphantom{\rmG}}'}

\newcommand{\piLsq}{\pi_{\rmL^2}}

\newcommand{\pip}{\pi^{\mathrm{SAff}}}
\newcommand{\piN}{\pi^\rmN}
\newcommand{\chiN}{\chi^\rmN}
\newcommand{\chiP}{\chi^\rmP}
\newcommand{\sgnP}{\sgn^\rmP}
\newcommand{\piP}{\pi^\rmP}

\newcommand{\pisaff}{\pi^{\mathrm{SAff}}}

\newcommand{\ISL}{I^{\mathrm{SL}}}
\newcommand{\DSL}{D^{\mathrm{SL}}}

\newcommand{\rmLH}{\rmL^\rmH}
\newcommand{\rmRH}{\rmR^\rmH}

\newcommand{\fol}{\mathrm{fol}}
\newcommand{\tot}{\mathrm{tot}}
\newcommand{\eucl}{\mathrm{eucl}}
\newcommand{\ver}{\mathrm{vert}}
\newcommand{\CasFol}{\cD^{\fol}}
\newcommand{\CasTot}{\cD^{\tot}}
\newcommand{\LapFol}{\Delta^{\fol}}
\newcommand{\LapTot}{\Delta^{\tot}}

\newcommand{\LapVert}{\Delta^{\ver}}
\newcommand{\LapCmp}[1]{\Delta^{\mathrm{cmp}(#1)}}
\newcommand{\CasFolEucl}{\cD^{\fol}_{\eucl}}
\newcommand{\CasTotEucl}{\cD^{\tot}_{\eucl}}

\def\be{\begin{equation}}   \def\ee{\end{equation}}     \def\bes{\begin{equation*}}    \def\ees{\end{equation*}}
\def\ba{\be\begin{aligned}} \def\ea{\end{aligned}\ee}   \def\bas{\bes\begin{aligned}}  \def\eas{\end{aligned}\ees}
\def\={\;=\;}  \def\+{\,+\,} 

\newcommand{\defeqwd}{\mathrel{\;\defeq\;}}
\newcommand{\congwd}{\mathrel{\;\cong\;}}
\newcommand{\newd}{\mathrel{\;\ne\;}}

\newcommand {\RL}{\mathrm L}
\newcommand{\rd}{\mathrm d}

\newcommand{\ovsgnP}{\ov{\sgn}\vphantom{\sgn}^\rmP}
\newcommand{\ovpiP}{\ov{\pi}\vphantom{\pi}^\rmP}
\newcommand{\ovISL}{\ov{I}\vphantom{I}^{\mathrm{SL}}}
\newcommand{\ovDSL}{\ov{D}\vphantom{D}^{\mathrm{SL}}}
\newcommand{\ovchiN}{\ov{\chi}\vphantom{\chi}^\rmN}

\newcommand{\de}{\, \mathrm d}
\newcommand{\SVrelM}{\SV_{\mathrm{rel},M}}
\newcommand{\CSVrelinf}{\cS\cV_{\mathrm{rel},\infty}}
\newcommand{\CSVrelM}{\cS\cV_{\mathrm{rel},M}}
\newcommand{\RC}{\mathrm C}
\newcommand{\del}{\partial}
\DeclareMathOperator{\symb}{symb}
\newcommand{\eps}{\epsilon}
\DeclareMathOperator{\spec}{spec}
\DeclareMathOperator{\dist}{dist}
\DeclareMathOperator{\supp}{supp}

\renewcommand{\epsilon}{\varepsilon}

 \makeatother

\begin{document}

\title[Spectral decomposition and Siegel-Veech transforms for strata]%
{Spectral decomposition and Siegel-Veech transforms for strata:
 The case of marked tori}
\date{\today}
\author{Jayadev~S.~Athreya}
\address{Department of Mathematics, University of Washington, Box 354350, Seattle, WA, 91895, USA}
\email{jathreya@uw.edu}
\thanks{J.S.~Athreya was supported by NSF grant DMS
  2003528; the Pacific Institute for the Mathematical Sciences;
  the Royalty Research Fund
and the Victor Klee fund at the University of Washington; } 
\author{Jean~Lagac\'e}
\address{Department of Mathematics, King's College London, Strand, London, WC2R
2LS, United Kingdom}
\email{jean.lagace@kcl.ac.uk}
\author{Martin~M\"oller}
\address{Institut f\"ur Mathematik, Goethe-Universit\"at Frankfurt,
Robert-Mayer-Str. 6-8,
60325 Frankfurt am Main, Germany}
\email{moeller@math.uni-frankfurt.de}
\thanks{M.~Möller was supported by the DFG-project MO 1884/2-1 and
  the Collaborative Research Centre
TRR 326 ``Geometry and Arithmetic of Uniformized Structures''.}
\author{Martin~Raum}
\address{Chalmers tekniska högskola och G\"oteborgs Universitet, Institutionen f\"or Matematiska vetenskaper, SE-412 96 G\"oteborg, Sweden}
\email{martin@raum-brothers.eu}
\thanks{M.~Raum was supported by Vetenskapsr\aa det Grants~2019-03551 and~2023-04217.}

\begin{abstract}
Generalizing the well-known construction of Eisenstein series on the modular
curves, Siegel-Veech transforms provide a natural construction of
square-integrable functions on strata of differentials on Riemannian surfaces.
This space carries actions of the foliated Laplacian derived from
the~$\SL{2}(\RR)$-action as well as various differential operators related
to relative period translations.
\par
In the paper we give spectral decompositions for the stratum of tori with two
marked points. This is a homogeneous space for a special affine
group, which is not reductive and thus does not fall into well-studied cases of
the Langlands program, but still allows to employ techniques from representation
theory and global analysis. Even for this simple stratum exhibiting all
Siegel-Veech transforms requires novel configurations of saddle
connections. We also show that the contiunuous spectrum of the foliated
Laplacian is much larger than the space of Siegel-Veech transforms, as opposed
to the case of the modular curve. This defect can be remedied by using
instead a compound Laplacian involving relative period translations.
\end{abstract}

\maketitle

\setcounter{tocdepth}{1}
\tableofcontents
\setcounter{tocdepth}{2}

 \section{Introduction}\label{sec:intro}

For the modular surface or more generally for quotients of the upper half plane
by a cofinite Fuchsian group~$\Gamma$ the space $\rmL^2(\Gamma \backslash \bbH)$
is well-known to decompose into the cuspidal part, the space of Eisenstein
transforms and the residual spectrum. The Laplace operator acts with discrete
spectrum on the cuspidal part, while Eisenstein series provide the continuous
spectrum. The fine structure of the cuspidal part, the size of the spectral
gap and the description of the residual spectrum is the context of various open
conjectures. There is a similar decomposition of $\rmL^2(\Gamma \backslash
\SL2(\bbR))$, after first decomposing the space into $\rmK$-types, where
$\rmK=\operatorname{SO}(2)$ is the standard maximal compact subgroup of $\SL2(\bbR)$.
\par
There are two natural generalizations of this decomposition problem. First,
we may replace~$\SL2(\bbR)$ by any Lie group~$G$ of higher rank or even $p$-adic
and study the decomposition of $\rmL^2(\Gamma \backslash G)$.
Second, we may replace~$\Gamma \backslash \SL2(\bbR)$ by a stratum
$\cH(\alpha)$ of area one flat surfaces with zeros of order $\alpha =
(m_1,\ldots,m_n)$ with the Masur--Veech measure~$\nu_\MV$. For instance the
stratum $\cH(0)$ of area one tori with one marked points can be identified
with the unit tangent bundle $\SL2(\bbZ) \backslash \SL2(\bbR)$ to the
modular surface $\SL2(\bbZ)\backslash \bbH$. The first generalization
has been studied intensively for semi-simple Lie groups, in particular
in connection with the Langlands program, for example \cite{langlands-1970,
langlands-1989,arthur-2013}. For the second generalization, the spaces
$\rmL^2(\cH(\alpha)) := \rmL^2(\cH(\alpha); \nu_\MV)$ and even more generally
for linear submanifolds of~$\cH(\alpha)$, notably the existence of a spectral
gap for the foliated Laplacian corresponding to
the $\SL2(\bbR)$-action has been established in work of Avila--Gou\"ezel
\cite{AG13}. However their work explicitly avoids a decomposition of the
spectrum as above (``since the geometry at infinity is very complicated'').
Given recent progress towards understanding the boundary of strata \cite{BCGGM}
we aim to shed light on how the boundary relates with the continuous
spectrum for strata.
\par
In this paper we focus on the first non-classical case namely the
stratum~$\cH(0,0)$ of area one tori with two marked points. At the same time this
is an instance of a space $\rmL^2(\Gamma \backslash G)$ for a non-reductive
group~$G$, namely the quotient
of the special affine group $\SAff{2}(\RR) = \SL2(\RR) \ltimes \RR^2$
by its integral lattice $\SAff{2}(\ZZ) = \SL2(\ZZ) \ltimes \ZZ^2$ minus the zero
section, which is identified with $\SL2(\ZZ) \backslash \SL2(\RR)$. Since the
Masur--Veech measure~$\nu_\MV$ extends over this locus, we may and will use the
identification
$$\rmL^2(\SAff{2}(\ZZ) \backslash
\SAff{2}(\RR)) \= \rmL^2(\cH(0,0))$$
throughout. We will rely on tools from representation theory, explain
why simple-minded generalizations from the modular surface case might fail,
and how these failures can be bridged.
\par
\medskip
\paragraph{\textbf{The perspective of Siegel--Veech transforms}} 
The Siegel--Veech transform is a method to construct functions in
$\rmL^2(\cH(\alpha); \nu_\MV)$ based on the analogy between lattice vectors
for homogeneous spaces and saddle connections on strata.
It takes as input a function~$f$ on~$\RR^2$, often
supposed smooth and compactly supported, and a `configuration' and returns the
function $\SV(f)$ associating with the flat surface~$(X,\omega)$ the sum
over~$f(v)$ for all saddle connections vectors~$v$ that stem from the given
configuration (see Section~\ref{sec:siegel_veech_transforms} for the precise
definition). For the special case of the modular surface, i.e.,\@ the case
of~$\cH(0)$, there is a unique configuration, which yields all primitive
lattice vectors and the Siegel--Veech transform of the spherical function
$f(v) = |v|^{2s}$ is just the usual (non-holomorphic) Eisenstein series.
In general the range of Siegel--Veech transforms on the modular surface
yields the spectral projection on the continuous spectrum of the Laplace
operator. For general strata, examples of
configurations are given by all saddle connections joining a simple zero
to a triple zero or by all core curves of cylinders. The modular surface
model case triggers the following questions.
\begin{enumerate}[label=(Q\arabic*)]
\item
\label{it:main_question:configurations}
  What is a complete set of configurations in the sense that their
  Siegel--Veech transforms account for all possible Siegel--Veech transforms?
\item
\label{it:main_question:continuous_spectrum}
  Are Siegel--Veech transforms responsible for all of the continuous
  spectrum of the foliated Laplacian $-\LapFol$
  (as defined below)?
\item
\label{it:main_question:cusp_forms}
  Is there a notion of cusp forms so that Siegel--Veech transforms
  are precisely the orthogonal complement of cusp forms? Is this notion
  of cusp forms related to boundary divisors 
  in the multi-scale
  compactification from \cite{BCGGM}, as they do in the case of the modular
  surface?
\end{enumerate}
We will answer these questions for $\cH(0,0)$ at the end of the introduction.
For each of the questions the answer is not quite the one we expected initially.
For general~$\cH(\alpha)$ all three of them seem completely open.
\par
\medskip
\paragraph{\textbf{The perspective of differential operators}} The
action of $\SL2(\RR)$ on strata~$\cH(\alpha)$ gives rise to a Casimir element
acting as a operator $\CasFol$ on $\rmL^2(\cH(\alpha))$. It 
is this operator or the corresponding Laplace operator $-\LapFol_k$ acting
on weight-$k$ modular forms on the projectivized stratum $\cH(\alpha)/\SO2(\RR)$
that we are mainly interested in. See Section~\ref{sec:diff_op} for details.
\par
Since we work in a homogeneous space for the group $\SAff{2}(\RR)$ we
have more differential operators at our disposal, which will also be
the case for strata~$\cH(0^{k})$  of tori with more than just one zero.
Even though $\SAff{2}(\RR)$ is not reductive, we show in
Proposition~\ref{prop:universal_envoloping_center}
that the center of the universal enveloping algebra is a polynomial ring
generated by a degree \emph{three}  `Casimir' element, which acts
as an operator that we call the total Casimir~$\CasTot$. Again we define
the corresponding Laplace operators~$-\LapTot_k$ on the projectivized strata.
\par
Another option is to incorporate the translation along torus fibers, i.e.\@
the relative period foliation, into a degree two differential operator
is to use an operator $\LapVert$.
The compound operator
$\LapCmp{\epsilon}_k \defeq \LapFol_k + \epsilon \LapVert$
operator is elliptic if and only if~$\epsilon > 0$, invariant under
$\SAff{2}(\RR)$-translations but, contrary to $\LapTot_k$, does not
commute with most other covariant
differential operators. We return to this operator at
the end of the introduction in connection with~\ref{it:main_question:continuous_spectrum}.
\par
\medskip
\paragraph{\textbf{The perspective of representation theory}}
Pullback via the map $\cH(0,0) \to \cH(0)$ forgetting the last point
gives an inclusion $\rmL^2( \cH(0)) \hookrightarrow \rmL^2(\cH(0,0))$.
We call its orthogonal complement
the \emph{genuine part} $$\rmL^2(\cH(0,0))^\gen = \rmL^2(\cH(0))^\perp.$$ From
now on we focus on this genuine part and 
discard the pullbacks of $\SL{2}(\RR)$-representations. 
The irreducible repre\-sentations of $\SAff{2}(\RR)$ are classified by Mackey
theory. 
As we recall in Theorem~\ref{thm:Gp_classification_unitary_dual}
they are pullbacks of $\SL{2}(\RR)$-representations, which we discarded,
and representations $\pip_{n,m}$ induced from characters of a fixed
Heisenberg subgroup of $\SAff{2}(\RR)$, with $\pip_{n_1,m_1}$
and~$\pip_{n_2,m_2}$ isomorphic if and only if~$n_1 m_1^2 = n_2 m_2^2$.
As a first step towards answering our main questions, we exhibit the decomposition of~$\rmL^2(\cH(0,0))$.
\par
\begin{theorem} \label{intro:decompSAff}
The genuine part of the $\rmL^2$-space of the stratum~$\cH(0,0)$ admits
a decomposition
\be
\rmL^2\big(\cH(0,0)\big)^\gen \=
\rmL^2\big( \SAff{2}(\ZZ) \backslash \SAff{2}(\RR) \big)^\gen
\congwd
  \bigoplus_{m = 1}^\infty
  \bigoplus_{n \in \ZZ}
  \pisaff_{n,m}
\ee
Explicitly, the representation $\pisaff_{n,m}$ is the $\SAff{2}(\RR)$-invariant
subspace generated by the lifts of Eisenstein series $E_{k;m,\beta}$ for $n=0$
and Poincar\'e series~$P_{k;n,m,\beta}$ for $n \neq 0$ for any integrable
function~$\beta: \RR^+ \to \CC$, as defined
in~\eqref{eq:def:affine_eisenstein_series}
and~\eqref{eq:def:affine_poincare_series}.
\end{theorem}
\par
A main tool in the proof of Theorem~\ref{intro:decompSAff} are Fourier
expansions. The Fourier expansions along the translation subgroup~$\RR^2$
of $\SAff2(\RR)$ plays only a minor role. More important is the Fourier
expansion along a subgroup isomorphic to~$\RR^2$ inside a Heisenberg subgroup
but with non-trivial intersection with $\SL2(\RR)$. We name these
the Fourier--Heisenberg  coefficients $c^{\rmH}(\,\cdot\,,n,r; v, v/y)$, since we decompose the coefficient $r=0$ even further, along a Heisenberg group, see
Section~\ref{sec:fourier_expansions}. Here $(\tau,z) = (x+iy, u+iv)$ are the
standard coordinates on the \emph{Jacobi half-space} $\bbH \times \CC$.
\par
Next we aim for the decomposition of $\rmL^2(\cH(0,0))$ into
irreducible $\SL2(\RR)$-repre\-sen\-tations.
In general the problem of decomposing the restriction of representations
into irreducible ones is known as the \emph{branching} problem and discussed
in many instances (e.g.\@ \cite{kobayashi-kubo-pevzner-2016,gan-gross-prasad-2020} and the references therein). 
Our case might be known, but since
we were not able to locate a proof in the literature we give the details
of the following result, see
Proposition~\ref{prop:saff_representation_sl_branching} for the full
statement including the case~$n=0$.
\par
\begin{proposition}%
\label{intro:saff_representation_sl_branching}
For any~$m \in \ZZ^\times$ and~$n \in \ZZ \setminus \{0\}$, the restrictions
of the $\SAff{2}(\RR)$-representations decompose as a direct integral
\be \label{eq:abstractbranching}
  \Res_{\SL2(\RR)}^{\SAff2(\RR)}\, \pisaff_{n,m}
\congwd
  \bigoplus_{k = 2}^\infty
  \DSL_{\sgn(n) k}
  \,\oplus\,
  \int^\oplus_{\RR^+}
  \big( \ISL_{+,i t} \oplus \ISL_{-,i t} \big)
  \,\rmd t
\tx{,}
\ee
where the discrete series $\DSL_{\sgn(n) k}$ and the principal series
representation $\ISL_{\pm,i t}$ are defined along with
Theorem~\ref{thm:sl2_langlands_classification_unitary}.
\par
In particular the complementary series does not occur in the decompostion
of $\rmL^2\big(\cH(0,0)\big)^\gen$.
\end{proposition}
\par
The decomposition into irreducible $\SAff{2}(\RR)$-representations in
Theorem~\ref{intro:decompSAff} is fully discrete. This corresponds to the fact
that there are square-integrable Eisenstein and Poincar\'e series that
contribute to individual constituents~$\pisaff_{n,m}$. It contrasts the
classical situation for~$\SL{2}(\RR)$ in which Eisenstein series contribute to
the continuous spectrum and cannot be square-integrable and eigenfunctions for
the Laplacian simultaneously.
Proposition~\ref{intro:saff_representation_sl_branching} recovers the
classical situation in parts: There are some square-integrable Eisenstein
series for~$\SAff{2}(\RR)$ that are eigenfunctions of the foliated Laplacian,
but there are also others that behave like Eisenstein series for~$\SL{2}(\RR)$.
\par
In the next result we clarify which Eisenstein and Poincar\'e series
are generating
the discrete and continuous pieces in which the representation breaks up
according to Proposition~\ref{intro:saff_representation_sl_branching}.
In the sequel we thus consider~$\pisaff_{n,m}$ as a subrepresentation
of $\rmL^2\big( \SAff{2}(\ZZ) \backslash \SAff{2}(\RR) \big)^\gen$
via the isomorphism of Theorem~\ref{intro:decompSAff}.
The $\Gamma$-factor in the next result and the Whittaker
function~$\rmW_{\kappa,\mu}(y)$ are defined
along with the complete statement of this result in
Theorem~\ref{thm:genuine_decomposition_branching}. It also includes
the corresponding statement for the representations~$\pisaff_{0,m}$.
\par
\begin{theorem} \label{intro:gen_decomp_branching}
For $k \in \ZZ \setminus \{0,\pm 1\}$ and $n \in \ZZ$ with
$nk > 0$
the representation  $\DSL_{\sgn(n) k}$
in~\eqref{eq:abstractbranching} is generated by the Poincar\'e series
for $\beta =     e^{-2 \pi |n| y}$ if $k>1$ and $\beta = y^{-k} e^{-2 \pi |n| y}$
if $k<-1$.
\par
Associating to $n \in \ZZ \setminus \{0\}$
and $\psi \in \rmL^2(\RR^+,dt)$ the lifts of
the Poincar\'e series $P_{k;n,m,\beta^\rmW_{k,n,\psi}}$ of the `Whittaker
transform' 
\begin{gather*}
  \beta^\rmW_{k,n,\psi}(y)
  \defeqwd
  \frac{1}{4\pi\, |n|^{\frac{3}{2}}}
  \int_{t \in \RR^+}
 \frac{\psi(t)}{(\Ga^\rmW(t) \Ga^\rmW(-t))^{\frac{1}{2}}}\,
  y^{-\frac{k}{2}}
  \rmW_{\frac{\sgn(n) k}{2}, i t} \big( 4 \pi |n| \, y\big)\,
  \,\rmd t
\end{gather*}
gives rise to isometric embeddings
\bes
\rmP^\rmW_+:  \bigoplus_{k \in 2 \ZZ}   \rmL^2\big( \RR^+, \rmd t \big)
\to \pisaff_{n,m}, \qquad 
\rmP^\rmW_-:  \bigoplus_{k \in 1+ 2 \ZZ}   \rmL^2\big( \RR^+, \rmd t \big)
\to \pisaff_{n,m} 
\ees
whose images are $\int^\oplus_{\RR^+} \ISL_{+,i t} \,\rmd t$ and
$\int^\oplus_{\RR^+} \ISL_{-,i t} \,\rmd t$ respectively, in the
decomposition~\eqref{eq:abstractbranching}.
\end{theorem}
\par
The proof has of course similarities with the way the Eisenstein transform
identifies the continuous spectrum of the modular surface, see e.g.\
\cite[Section~4.2.5]{Bergeron} for a textbook version. Note however that the 
principal series appear with infinite multiplicity which we accomodate by
first restricting to individual~$\pisaff_{n,m}$. Further, as opposed to the
classical case, Poincar\'e series associated with Whittaker functions
contribute to the continuous spectrum, which requires a more delicate estimate.
\par
\medskip

\paragraph{\textbf{The main results}} In view of the next theorem we
define the space of \emph{cusp forms} to be the subspace of modular-invariant
functions on projectivized strata where the Fourier coefficient
$c^{\rmH}(\,\cdot\,,0,0; v, v/y)$ vanishes. We use the same terminology for the
lifts of these functions to $\rmL^2\big( \SAff{2}(\ZZ) \backslash
\SAff{2}(\RR)) \big)^\gen$. Similarly, we focus on this genuine subspace by
considering only Siegel--Veech transforms of mean-zero functions from now on.
\par
\begin{theorem} \label{intro:SVupperbound}
Siegel--Veech transforms of compactly supported mean-zero functions are
contained in the subspace of\/ $\rmL^2\big( \SAff{2}(\ZZ) \backslash
\SAff{2}(\RR)) \big)^\gen$ which is annihilated by $\CasTot$, which is
the subspace $\oplus_{m = 1}^\infty \pisaff_{0,m}$. This space is
the orthogonal complement of the space of cusp forms.
\end{theorem}
\par
In the case $\cH(0,0)$ there are two obvious configurations, using the
`absolute periods', i.e.\@ lattice vectors, and using `relative periods'
joining one zero to the other. We denote the corresponding Siegel--Veech
transforms by $\SVabs(\,\cdot\,)$ and $\SVrel(\,\cdot\,)$ respectively. The
absolute Siegel--Veech transforms only contribute to the well-studied
non-genuine part of the $\rmL^2$-space and will be disregarded in the sequel.
\par
However the above is not a complete list of configurations!
In fact, for a point $(\Lambda, z) \in \cH(0,0)$
and any $M \in \NN$ the set $z + \tfrac1M \Lambda$ of translates of
the relative period by a $1/M$-th lattice vector also satisfies all properties
of a `configuration', and Theorem~\ref{intro:SVupperbound} also includes
these. We denote the corresponding Siegel-Veech transform by
$\SVrelM$ and let
\be \CSVrelM \= \clinspan \Big\{\SVrelM(f) : f \in \RC_{c,0}^\infty(\RR^2)\Big\}
\ee
Together with Theorem~\ref{intro:SVupperbound} the following result
shows that we have found all configurations, thus answering~\ref{it:main_question:configurations}.
\par
\begin{theorem} \label{intro:SVperp}
There is an orthogonal decomposition
\bes
\rmL^2 \big( \cH(0,0) \big)^{\gen} \= 
\rmL^2 \big( \cH(0,0) \big)^{\gen}_{\cusp} \oplus
\clinspan \Big(\bigcup_{M=1}^\infty \CSVrelM \Big)\,.
\ees
\end{theorem}
It also implies together with
Proposition~\ref{intro:saff_representation_sl_branching}
and Theorem~\ref{intro:decompSAff} that Siegel-Veech transforms 
do not account for the full continuous spectrum of~$\CasFol$ on~$\cH(0,0)$,
since every $\pisaff_{n,m}$ regardless of whether~$n = 0$ or not contributes
to its continuous spectrum. This answers~\ref{it:main_question:continuous_spectrum} negatively for this stratum.
Finally we observe that Theorem~\ref{intro:SVperp} is a positive answer
to the first part of~\ref{it:main_question:cusp_forms}. Note, however, that vanishing of a single Fourier
coefficient of an $\RR^2$-action is a codimension two condition rather than
a divisorial condition.
\par
While~\ref{it:main_question:continuous_spectrum} was answered negatively it makes sense to modify it to
\begin{enumerate}[label=(Q\arabic*'),start=2]
\item
\label{it:main_question:continuous_spectrum:alternative}
Is there an operator for which the Siegel--Veech transforms are responsible for all 
    its continuous spectrum, and if so what is it?
\end{enumerate}
The answer to~\ref{it:main_question:continuous_spectrum:alternative} is that there is such an operator, and it is the compound
Laplacian introduced earlier. With the given definition of cusp
forms, the behaviour of this operator parallels the usual Laplacian
on the modular surface, and as $\eps \searrow 0$ its discrete spectrum converges
to the part of the continuous spectrum of the foliated Laplacian missed by the
Siegel--Veech transforms.
\par
\begin{theorem} \label{intro:compound}
The compound Laplacian $-\LapCmp{\epsilon}_k$  has discrete spectrum
on the space of genuine cusp forms of\/~$\rmK$-type~$k$. As $\eps \searrow 0$, the
spectrum of $-\LapFol_k$ is comprised of limit points from the spectra of
$-\LapCmp{\eps}_k$. This remains true of the restriction of these operators to
cusp forms or their orthogonal complement.
\end{theorem}
\par
All the differential operators considered  here, $\LapFol$, $\LapTot$,
and~$\LapCmp{\eps}$, also exist for strata and linear manifolds therein
provided they have a non-trivial relative period foliation. Among those,
linear manifolds of rank one are the natural scope to extend the main
results of this paper. We plan to explore this in a follow-up paper.
\par
\medskip
\paragraph{\textbf{Notes and references}} For a given hyperbolic surface $\Gamma
\backslash \SL2(\RR) / \rmK$ the interpretation of the Siegel-Veech transform
as Eisenstein series has been used in a number of papers, starting with
\cite{Veech89}. See in particular \cite{BNRW} and the references there,
for example for applications to counting problems of lattice vectors in
star-shaped regions.
\par
The compound differential operator and its spectral decomposition for
the special case of Maass forms of weight zero appear in an unpublished manuscript
of Balslev~\cite{balslev-2011} in the equivalent guise of Jacobi forms
of weight and index~$0$. After adjusting to his set of coordinates one
checks that his Laplacian equals our~$-\LapCmp{4}_0$. The Fourier expansions
(Section~\ref{sec:fourier_expansions}) are also discussed in~\cite{balslev-2011}
aiming to decompose the $\rmL^2$-space into eigenspaces of his Laplacian.
The first statement  of Thereom~\ref{intro:compound} is also claimed without
proof
in~\cite{balslev-2011}. Balslev moreover computes explicitly a Weyl's law
for the spectrum of his Laplacian (in weight~$0$) and briefly addresses
the same question for the covering space given by replacing~$\SAff2(\ZZ)$ with
a subgroup of small finite index.
\par
The reader might also view this paper as a complement to the book of
Berndt--Schmidt~\cite{berndt-schmidt-1998}, where representations of the Jacobi
group are discussed from a perspective inspired by automorphic representation
theory. They, however, restrict very early in their treatment to central
character zero, which rules out precisely the case that we consider in the
present work. The prominent role played by the Schr\"odinger--Weil representation
in their setting reduces them to representations of the metaplectic group,
which were intensively studied for instance by Waldspurger in prior work.
Plenty of representation theoretic subtleties in the present work can only
occur because of the lack of such a tight connection to any (covering of) a
classical group.
\par
The Siegel(--Veech) transform for affine lattices has been used for effective equidistribution results in \cite{GKY}, see also \cite{SV}.
\par
There is a long history using Ratner's theory on the space $\cH(0,0)$ to
study saddle connection, notably their gap distributions, see e.g.\
\cite{ElMc, MSLorentz, Sanchez}.
\par
The analog of Selberg's conjecture (the size of the spectral gap or the
non-existence of complementary series) for strata or its congruence covers
is a question of Yoccoz. See \cite{MageeSelberg} and \cite{MageeRuehr} for
progress in this direction.
\par
\medskip
\paragraph*{\bf Organization of the paper.} Section \ref{sec:diff_op} sets the
table by defining the various Casimir and Laplace differential operators
relevant to our analysis from all perspectives. Section \ref{sec:RepTh} is
dedicated to the representation theory, building a decomposition
of~$\rmL^2(\Gp(\ZZ)\backslash\Gp(\RR))$ via Mackey theory. In order to study the
spectral decomposition of $\rmL^2(\cH(0,0))$ or the role of the Siegel--Veech
transforms, we introduce the Fourier and Fourier--Heisenberg expansions in
Section \ref{sec:fourier_expansions}. We provide the spectral decomposition in
Section \ref{sec:decomp}. There, we note that Theorems
\ref{thm:genuine_decomposition} and \ref{thm:genuine_decomposition_branching}
are refinements rather than restatements of Theorems
\ref{intro:decompSAff} and
\ref{intro:gen_decomp_branching} respectively. This is also where cusp forms are
introduced. While they do not correspond to the discrete part of the spectrum
for the foliated Laplacian, we show that they do for the compound Laplacian.
Finally, in Section \ref{sec:siegel_veech_transforms} we introduce the
Siegel--Veech transform and show that they cover the complement of the cusp
forms, giving a final decomposition and interpretation of $\rmL^2(\cH(0,0))$. 
\par
\medskip
\paragraph*{\bf Acknowledgements.}  Much of this work was conducted
while J.S.A. held the Chaire Jean Morlet at the Centre International
de Recherches Mathematique-Luminy in Autumn 2023. We thank CIRM for the
hospitality and the inspiring working conditions. This project originated during a meeting at the Institut d'\'Etudes Scientifiques in Cargese, in Summer 2022. Moreover, we thank
Raphael Beuzart-Plessis, Jan Bruinier, Anish Ghosh, and Lior
Silberman for helpful conversations.

\section{Differential operators for the special affine group}
\label{sec:diff_op}

Each stratum~$\cH(\alpha)$ admits an action by $\G(\RR) =\SL2(\RR)$. For an introduction to strata and the dynamics of the $\SL2(\RR)$-action, see, for example~\cite{AthreyaMasur}.
Up  to measure zero, the stratum $\cH(0,0)$ agrees with $\SAff{2}(\ZZ)
\backslash \SAff{2}(\RR)$ and is,
contrary to other strata in higher genus, a homogeneous space.
There are several interesting differential operators
acting on this space. First,  the Casimir element~$C$ of $\SL2(\RR)$
induces a second order `foliated' differential operator~$\CasFol$
 that involves
only the derivatives along the leaves of the foliation by $\SL2(\RR)$-orbits.
Second, we show in Proposition~\ref{prop:universal_envoloping_center}
that the group $\Gp(\RR)=\SAff{2}(\RR)$, despite not being reductive, has a
universal enveloping algebra, whose center is a polynomial ring in one
variable. We call a generator of this polynomial ring a Casimir element~$C'$.
It induces an order \emph{three} differential operator~$\CasTot$.
\par
Just as in the classical case of the modular curve, we may pass between functions
on $\cH(0,0)$ of a given $\rmK$-type and modular-invariant functions on the
quotient $\cH(0,0)/\rmK$, which is the quotient of the Jacobi half plane~$\bbH'$
by~$\SAff{2}(\ZZ)$. We state this correspondence in Section~\ref{ssec:automorphic}.
Under this correspondence, the 'total' and 'foliated' differential
operators~$\CasTot$ and~$\CasFol$ correspond to Laplace operators
$-\LapTot$ and~$-\LapFol$.
\par
To complete the picture, we observe that besides these two operators there
is a vertical Laplace operator~$-\LapVert$ which is $\G'(\RR)$-invariant.
We call any linear combination $- \LapCmp{\eps} = - \LapFol - \eps \LapVert$
with $\eps >0$ a \emph{compound
Laplace operator}, whose basic properties we discuss
in Section~\ref{ssec:Balslev-Op}.
\par 
\medskip
We will write elements in $\Gp(\RR) = \SL{2}(\RR) \ltimes \RR^2$ as $(g,w)$
where $w = (w_1,w_2)$ is a row vector with composition law $(g,w)\cdot
(\td{g},\td{w}) = (g\td{g}, w\td{g} + \td{w})$. We need the compact subgroup
$\rmK \defeq \SO{2}(\RR) \subset \rmG(\RR)\subset \Gp(\RR)$.
The Poincaré upper half plane and its affine extension, called the Jacobi upper
half space in~\cite{eichler-zagier-1985}, are 
\bas
  \HS
&\=
 \big\{ \tau \in \CC \,:\, \Im(\tau) > 0 \big\} \,\cong\, \G(\RR) \slash \rmK
\tx{,}
\\
\HSp \=
\HS \times \CC &\= 
\big\{ (\tau,z) \in \CC^2 \,:\, \Im(\tau) > 0 \big\} \,\cong\, \Gp(\RR) \slash \rmK
\tx{.}
\eas
Following the conventions for Jacobi forms, we use the coordinates 
\be \label{eq:Jacobicoords}
  \tau \= x + iy
\quad\tx{and}\quad
  z \= u + iv \= p \tau + q
\tx{.}
\ee
Let~$\frakg$ and~$\frakgp$ be the complexified Lie algebras of~$\G(\RR)$
and~$\Gp(\RR)$ respectively. Given the elements of $\frakg$ 
\bes
 F
\=
  \begin{psmatrix} 0 & 1 \\ 0& 0 \end{psmatrix}
\tx{,}\quad
  H
\=
  \begin{psmatrix} 1 & 0\\0 & -1 \end{psmatrix}
  \tx{,}\quad \text{and} \quad
  G
\=
  \begin{psmatrix} 0 &0 \\ 1 & 0 \end{psmatrix}
\tx{,}
\ees
we use as a basis of $\frakg$
\begin{gather}
    \label{eq:frakg_basis}
Z \= -i (F - G)\tx{,} \quad \text{and} \quad X_\pm \= \tfrac{1}{2} \big(H \pm i
(F + G) \big).
\end{gather}
Considering additionally the elements of $\frakgp$
\bes
  P
\=
  \big( \begin{psmatrix} 0 & 0\\0 & 0 \end{psmatrix}, (1,0) \big)
  \quad \text{and} \quad
  Q
\=
\big( \begin{psmatrix} 0 & 0 \\0 & 0 \end{psmatrix}, (0,1) \big)
\ees
we use as a basis for $\frakgp$ the set
\begin{gather}%
\label{eq:frakgp_basis}
(Z,0,0), \quad (X_{\pm},0,0), \quad \text{and} \quad Y_{\pm} = \tfrac{1}{2} (P
\pm i Q)\tx{.}
\end{gather}
In the sequel, we abuse notation and denote $Z = (Z,0,0)$
and~$X_\pm=(X_\pm,0,0)$ when it is clear that we are considering them as elements of
$\frakgp$.
\subsection{A Casimir element for the special linear group}%
\label{ssec:casimir_element_sl2}
By the general theory of reductive groups, a Casimir element~$C$ for~$\frakg
= \fraksl{2}$ is any generator of the center~$\frakz$ of the universal enveloping
algebra~$\rmU(\frakg)$. Such an element is given by $C = \sum_X X X^\vee$,
where~$X$ runs through a basis of the Lie algebra and~$X^\vee$ is the dual
of~$X$ with respect to the Killing form. Explicitly
\begin{gather}
\label{eq:casimir_element_sl2}
  C
=
  \tfrac{1}{4} X_+ X_-
  +
  \tfrac{1}{8} Z^2
  +
  \tfrac{1}{4} X_- X_+
=
  \tfrac{1}{2} X_+ X_-
  +
  \tfrac{1}{8} Z^2
  -
  \tfrac{1}{4} Z
\tx{,}
\end{gather}
and we define a foliated differential operator as the left action
\be   \CasFol\, f \,:=\,   2 C\, f
\ee
of the Casimir element, matching normalization used e.g.\@ in the theory of elliptic modular forms.

\subsection{A Casimir element for the special affine group}%
\label{ssec:casimir_element_saff2}

Since~$G'(\RR) = \SAff2(\RR)$ is not reductive, we determine the center of
$\rmU(\frakgp)$ in an ad hoc way. We nevertheless refer to~$C'$ below as a
\emph{Casimir element}. Similar computations of Casimir elements (that
are also degree three) have appeared for the Jacobi group in \cite{BCR} and
\cite{ConleyRaum}. The following proposition complements these computations
(and also those in \cite{berndt-schmidt-1998}) which were always restricted
to representations of the Jacobi group with non-trivial central character.
\par
\begin{proposition}%
\label{prop:universal_envoloping_center}
The center~$\frakzp$ of the the universal enveloping algebra~$\rmU(\frakgp)$ is
a polynomial ring
\begin{gather}
\label{eq:casimir_element_saff2}
  \frakzp \=
  \CC[C']
\quad\tx{with\ generator} \quad 
  C'
\=
  Z Y_+ Y_-
  -
  X_+ Y_-^2 
  +
  X_- Y_+^2
\tx{.}
\end{gather}
\end{proposition}
\par
Before proving the proposition, we observe that it gives rise to a differential
operator via  the left action
\be   \CasTot\, f \,:=\,   2 C'\, f.
\ee
\par
\begin{proof} Let $\frakA = \gr\, \rmU(\frakgp)$ be the  associated graded algebra
and
\begin{gather*}
\sigma:   \frakA
\ra
  \rmU(\frakgp)
\tx{,}\quad
  m_1 \cdots m_n
\lmto
  \sum_{\pi \in \rmS_n}
  m_{\pi(1)} \cdots m_{\pi(n)}
\tx{,}\quad
  m_i \in \frakgp \tx{\ for\ } 1 \le i \le n
\tx{,}
\end{gather*}
be the linear symmetrization map (which is not an algebra homomorphism).
\par
The leading term of any element of~$\frakzp$ yields a central element of~$\frakA$.
Conversely, since commutators in the associative algebra~$\rmU(\frakgp)$ strictly
lower the degree filtration and by induction on the
degree {(cf.~\cite[Theorem~10]{helgason-1959} and the lemmas used in its proof)},
we see that the symmetrization map yields a bijection 
\begin{gather*}
 \ker\big( \frakgp \circlearrowright \frakA \big)
\ra
  \frakzp
\tx{.}
\end{gather*}
In particular, $\frakA$ is commutative, but carries a non-trivial, degree preserving
representation of~$\frakgp$. To determine the kernel of the~$\frakgp$\nbd{}action
on~$\frakA$, we record that
\begin{align*}
  \big[ Y_+,\, Z^m X_+^{n_+} X_-^{n_-} \big]
&{}\=
  - m Z^{m-1} X_+^{n_+} X_-^{n_-}
  + n_- Z^m X_+^{n_+} X_-^{n_- - 1}
\tx{,}\\
  \big[ Y_-,\, Z^m X_+^{n_+} X_-^{n_-} \big]
&{}\=
    m Z^{m-1} X_+^{n_+} X_-^{n_-}
  - n_+ Z^m X_+^{n_+-1} X_-^{n_-}
\tx{.}
\end{align*}
We conclude that any homogeneous element of~$\frakA$ that vanishes under~$[Y_+, \,\cdot\,]$ is of the form
\begin{gather*}
  \sum_m \big( Z Y_- + X_- Y_+ \big)^m\,
  p^+_m(X_+, Y_+, Y_-)
\tx{,}
\end{gather*}
and any homogeneous element of~$\frakA$ that vanishes under~$[Y_-, \,\cdot\,]$ is of the form
\begin{gather*}
  \sum_m \big( Z Y_+ - X_+ Y_- \big)^m\,
  p^-_m(X_-, Y_+, Y_-)
\tx{,}
\end{gather*}
for suitable polynomials~$p^\pm_m$. By induction on the degree in~$Z$, we conclude that an element that is annihilated by both~$[Y_+, \,\cdot\,]$ and~$[Y_-, \,\cdot\,]$ is of the form
\begin{gather*}
  \sum_m \big( Z Y_+ Y_- - X_+ Y_-^2 + X_- Y_+^2 \big)^m\,
  q_m(Y_+, Y_-)
\end{gather*}
for suitable polynomials~$q_m$.

We next argue that every~$q_m$ is constant. To this end, note that~$q_m$ must
vanish under the Lie action of~$\frakgp$, since~$Z Y_+ Y_- - X_+ Y_-^2 + X_-
Y_+^2$ does. We have
\begin{gather*}
  \big[ X_+,\, Y_+^{n_+} Y_-^{n_-} \big]
=
  - n_- Y_+^{n_+ + 1} Y_-^{n_- - 1}
\quad\tx{and}\quad
  \big[ X_-,\, Y_+^{n_+} Y_-^{n_-} \big]
=
  - n_+ Y_+^{n_+ - 1} Y_-^{n_- + 1}
\tx{.}
\end{gather*}

By induction on the degree in~$Y_+$, we find that
\begin{gather*}
  \ker\big( \frakgp \circlearrowright \frakA \big)
\=
  \CC\big[ Z Y_+ Y_- - X_+ Y_-^2 + X_- Y_+^2 \big]
\tx{.}
\end{gather*}
The image of the generator on the right hand side under the symmetrization map equals
\begin{multline*}
  \big( Z Y_+ Y_- + Z Y_- Y_+ + Y_+ Z Y_- + Y_- Z Y_+ + Y_+ Y_- Z + Y_- Y_+ Z \big)
\\
  -
  2 \big( X_+ Y_-^2 + Y_- X_+ Y_- + Y_-^2 X_+ \big)
  +
  2 \big( X_- Y_+^2 + Y_+ X_- Y_+ + Y_+^2 X_- \big)
\tx{.}
\end{multline*}
Using the commutator of~$Z$ and~$Y_\pm$, we calculate that the expression in the first pair of parentheses simplifies to~$6 Z Y_+ Y_-$. For the two other expressions, we obtain
\begin{alignat*}{2}
&
  2 \big( X_+ Y_-^2 + Y_- X_+ Y_- + Y_-^2 X_+ \big)
&&\=
  2 \big( X_+ Y_-^2 + 2 Y_- X_+ Y_- + Y_- Y_+ \big)
\\
\={}&
  2 \big( 3 X_+ Y_-^2 + 2 Y_+ Y_- + Y_- Y_+ \big)
&&\=
  2 \big( 3 X_+ Y_-^2 + 3 Y_+ Y_- \big)
\tx{,} 
\end{alignat*}
and similarly
\begin{alignat*}{2}
&
  2 \big( X_- Y_+^2 + Y_+ X_- Y_+ + Y_+^2 X_- \big)
&&\=
  2 \big( X_- Y_+^2 + 2 Y_+ X_- Y_+ + Y_+ X_- Y_- \big)
\\
\={}&
  2 \big( 3 X_- Y_+^2 + 2 Y_- Y_+ + Y_+ Y_- \big)
&&\=
  2 \big( 3 X_- Y_+^2 + 3 Y_+ Y_-\big)
\tx{.}
\end{alignat*}
Since the contributions of~$Y_+ Y_-$ for these terms cancel each other, we recover~$6 C'$ and finish the proof. 
\end{proof}

\subsection{Affine modular-invariant functions}%
\label{ssec:automorphic}

It will be convenient to pass back and forth between functions
on~$\Gp(\RR)$ and functions on the Jacobi upper half plane $\bbH' =
\Gp(\RR)/ \SO2(\RR)$.
For this we define the \emph{slash action} on functions on~$\bbH'$ parametrized
by $k \in \ZZ$ by
\begin{gather}
\label{eq:def:weight_k_slash_action_hp}
  \Big(
\phi \big\vert'_k\, \big( \begin{psmatrix} a & b \\ c & d \end{psmatrix}, w_1,w_2
\big) \Big)(\tau, z) \= (c \tau + d)^{-k}\,
\phi\big(\mfrac{a \tau + b}{c \tau + d}, \mfrac{z + w_1 \tau + w_2}
{c \tau + d}\big) \tx{,}
\end{gather}
extending the usual slash action on the upper half plane~$\bbH$. We say
that $\phi: \bbH' \to \CC$ is an \emph{affine modular-invariant function of
weight~$k$} if
\be
\phi\big|'_k (\ga,w) \= \phi \quad \text{for all} \quad (\ga,w) \in
\Gp(\ZZ) = \SAff2(\ZZ).
\ee
The first half of the correspondence is the \emph{lift} of affine
modular-invariant functions to functions on~$\G'(\RR)$ by
\begin{gather}%
\label{eq:def:modular_to_automorphic}
  \wtd\phi(g)
\defeq
\big( \phi \big|'_k\,g \big)(i,0)
\= e^{i k \theta} y^{\frac{k}{2}}\,
  \phi(\tau, z)
\end{gather}
for forms of weight~$k$, where for the second expression $\tau = x+iy$, $z =
u+iv$, and $g=(x,y,u,v,\theta)$ as in the Iwasawa decomposition
in~\eqref{eq:nak_decomposition_to_HSp} below. Note that the notation~$\wtd\phi$
suppresses the weight~$k$.
We generalise the standard raising as lowering operators and define the
operators~$\rmL_k$, $\rmR_k$, $\rmLH_k$, and~$\rmRH_k$ on affine modular invariant
functions via the lifts
\begin{gather}
\label{eq:def:maass_operators_from_lie_algebra}
  \wtd{\rmL_k\, \phi}
\defeq
  X_- \wtd\phi
\tx{,}\quad
  \wtd{\rmR_k\, \phi}
\defeq
  X_+ \wtd\phi
\tx{,}\quad
  \wtd{\rmLH_k\, \phi}
\defeq
  Y_- \wtd\phi
\tx{,}\quad
  \wtd{\rmRH_k\, \phi}
\defeq
  Y_+ \wtd\phi
\tx{.}
\end{gather}
From
\bas
  X_\pm\, \wtd\phi
&\=
  \pm \tfrac{i}{2} e^{i (k \pm 2) \theta}
  y^{\frac{k}{2}}
  \big(
  2 y (\partial_x \phi)
  +
  2 v (\partial_u \phi)
  \mp
  2 i y (\partial_y \phi)
  \mp
  2 i v (\partial_v \phi)
  \mp
  i k \phi
  -
  i k \phi
  \big)
,\\
  Y_\pm\, \wtd\phi
&\=
  \pm \tfrac{i}{2}
  e^{i (k \pm 1) \theta}
  y^{\frac{k+1}{2}}\,
  \big(
  (\partial_u \phi)
  \mp
  i (\partial_v \phi)
  \big)
\eas
we read off that the weight of the functions~$\rmL_k\, \phi$, $\rmR_k\, \phi$, $\rmLH_k\,\phi$, and~$\rmRH_k\,\phi$ in~\eqref{eq:def:maass_operators_from_lie_algebra} is~$k-2$, $k+2$, $k-1$, and~$k+1$ respectively. Explicit calculations show
\begin{gather}%
\label{eq:maass_operators}
\begin{alignedat}{2}
  \rmL_k
&\=
  - 2 i
  y^2\,
  \big(
  \partial_{\ov\tau}
  +
  v y^{-1} \partial_{\ov z}
  \big)
\tx{,}\qquad
&
  \rmR_k
&\=
  2 i\,
  \big(
  \partial_\tau
  +
  v y^{-1} \partial_z
  \big)
  +
  k y^{-1}
,
\\
  \rmLH_k
&\=
  - i y\, \partial_{\ov z}
\tx{,}\qquad
&
  \rmRH_k
&\=
  i \partial_z
\tx{.}
\end{alignedat}
\end{gather}
and yield the following lemma:
\par
\begin{lemma}%
  \label{la:DescendtoLaplace}
There are differential operators, $- \LapFol_k$ and $- \LapTot_k$  which we call the
\emph{foliated Laplacian} and \emph{total
Laplacian} of weight~$k$, respectively, with the property that
\begin{gather}
\label{eq:def:foliated_laplace_and_casimir_operator}
  \wtd{\LapFol_k\, \phi}
\defeq
  \CasFol\, \wtd\phi
\quad\tx{and}\quad
  \wtd{\LapTot_k\, \phi}
\defeq
  \CasTot\, \wtd\phi
\end{gather}
for any affine modular-invariant function~$\phi$ of weight~$k$. In $(x,y,u,v)$ coordinates 
\ba \label{eq:total_laplace_operator}
  \LapTot_k
\={}&
  k\, \rmRH_{k-1} \rmLH_k
  -
  \rmR_{k-2} \rmLH_{k-1} \rmLH_k
  +
  \rmL_{k+2} \rmRH_{k+1} \rmRH_k
\\
\={}&
  y
  \big( k\, \partial_{\ov{z}} + 2 i v\, \partial_z \partial_{\ov{z}} \big)
  (\partial_{\ov{z}} + \partial_z)
  +
  2 i y^2\,
  (\partial_{\ov{\tau}} \partial_z^2 + \partial_\tau \partial_{\ov{z}}^2)
\\
\={}&
  \tfrac{k}{2} y\, \partial_u (\partial_u + i \partial_v)
  +
  \tfrac{i}{2} y^2\, \partial_x (\partial_u^2 - \partial_v^2)
  +
  i y^2\, \partial_y \partial_u \partial_v
  +
  \tfrac{i}{2} y v\, \partial_u (\partial_u^2 + \partial_v^2)
  \tx{,}
\ea
and
\ba \label{eq:foliated_laplace_operator}
\LapFol_k &\=
  \rmR_{k-2}\, \rmL_k 
\\
&\=
    4 y^2 \partial_\tau \partial_{\ov\tau}
  + 4 yv \big( \partial_\tau \partial_{\ov{z}} + \partial_{\ov{\tau}} \partial_z \big)
  + 4 v^2 \partial_z \partial_{\ov{z}}
  - 2 i k \big( y \partial_{\ov\tau} + v \partial_{\ov{z}} \big)1
  \big)
\\
&\=
   y^2 (\partial_x^2 + \partial_y^2)
  + 2yv (\partial_x \partial_u + \partial_y \partial_v)
  + v^2 (\partial_u^2 + \partial_v^2)
\\
&\hphantom{{}\={}}
  - i k y (\partial_x + i \partial_y)
  - i k v (\partial_u + i \partial_v)
  \big)
\ea
Furthermore, in $(x,y,p,q)$ coordinates,
\begin{equation}
    \LapFol_k = y^2 (\del_x^2 + \del_y^2) - i k y (\del_x + i \del_y).
\end{equation}
\end{lemma}
Note that there is no dependence on $p, q$ in the definition of $\LapFol_k$.
\par
The converse of the correspondence~\eqref{eq:def:modular_to_automorphic} is
stated for fixed $\rmK$-type. Here a
vector~$v$ in a representation~$\rho$ of~$\rmK$ is said to be
\emph{of~$\rmK$\nbd{}type~$k \in \ZZ$}, if
\begin{gather}
\label{eq:def:k-type} 
  \rho\big(
  \begin{psmatrix} \cos\,\theta & \sin\,\theta \\ - \sin\theta & \cos\,\theta \end{psmatrix}
  \big)
  v
=
  e^{i k \theta}\, v
\tx{.}
\end{gather}
In particular, a function on~$\Gp(\RR)$ transforming in this way
under right shifts by~$\rmK$ is said to be of~$\rmK$\nbd{}type~$k$.
\par
\begin{lemma} 
\label{la:modular_to_automorphic}
Given an affine modular-invariant function~$\phi$, the
function~$\wtd\phi$ defined in~\eqref{eq:def:modular_to_automorphic} is a
$\Gp(\ZZ)$-left-invariant function.
\par
Conversely, if~$f$ is $\Gp(\ZZ)$-left-invariant and of\/~$\rmK$-type~$k$, then 
\bes
\phi(x+iy, u+iv) \=
f \big(\begin{psmatrix}  y^{1/2}& xy^{-1/2} \\
0 & y^{-1/2} \end{psmatrix}, uy^{-1/2}, vy^{-1/2} \big)\,
\ees
is an affine modular-invariant function of weight~$k$ with~$\wt{\phi} = f$.
\end{lemma}
\par
\begin{proof}
    The modular invariance of~$\phi$ implies that for~$\ga \in \Gp(\ZZ)$ and~$g
    \in \Gp(\RR)$
\begin{gather*}
  \wtd\phi(\ga g)
\=
  \big( \phi \big|'_k\, \ga g \big) (i,0)
\=
  \big( \phi \big|'_k\, g \big) (i,0)
\=
\wtd\phi(g).
\end{gather*}
That is, we can view~$\wtd\phi$ as a function on~$\Gp(\ZZ) \backslash \Gp(\RR)$.
\par
To determine the action of~$\rmK$ by right-shifts on~$\wtd\phi$
in~\eqref{eq:def:modular_to_automorphic} we consider ~$\begin{psmatrix} d & -c
\\ c & d \end{psmatrix} \in \rmK$ with $d = \cos\,\theta$, $c = -\sin\,\theta$
and~$g \in \Gp(\RR)$ and compute:
\begin{align*}
&
  \wtd\phi\big( g \begin{psmatrix} d & -c \\ c & d \end{psmatrix} \big)
\=
  \big( \phi \big|'_k\, g \begin{psmatrix} d & -c \\ c & d \end{psmatrix} \big)(i,0)
\\
\={}&
  (c i + d)^{-k}\,
  \big( \phi \big|'_k\, g \big)(i,0)
\={}
  (c i + d)^{-k}\,
  \wtd\phi(g)
\=
  e^{i k \theta}\,
  \wtd\phi(g)
\tx{.}
\end{align*}
In particular, $\wtd\phi$ is~$\rmK$-finite.
The converse is a direct computation, see also the Iwasawa decomposition
in~\eqref{eq:nak_decomposition_to_HSp}.
\end{proof}
\par
Under this correspondence the usual $\rmL^2$-scalar product on
$\Gp(\ZZ) \backslash \Gp(\RR)$ corresponds to the scalar product
\begin{gather}
    \begin{aligned}
\label{eq:def:inner_product_maass_forms}
  \big\langle \psi_1, \psi_2 \big\rangle
&\,\defeq\,
  \int_{\Gap \backslash \HSp}
  \psi_1(\tau, z)
  \ov{\psi_2(\tau,z)}\,
  \frac{\rmd x \rmd y \rmd u \rmd v}{y^{3-k}}
  \\&\=
  \int_{\Gap \backslash \HSp}
  \psi_1(\tau, z)
  \ov{\psi_2(\tau,z)}\,
  \frac{\rmd x \rmd y \rmd p \rmd q}{y^{2-k}}
\tx{.}
\end{aligned}
\end{gather}
We write the corresponding norm as~$\|\cdot \|_{\HSp,k}$ or~$\|\cdot \|_{\HSp}$, 
suppressing the $k$-dependence. It follows from the definition that the measure
on $\Gp(\ZZ) \backslash \Gp(\RR) / \rm K $ induced by this scalar produce is the
(push-forward to the $K$-quotient of the) Masur-Veech measure~$\nu_\MV$.
\par

\subsection{Invariant differential operators}%
\label{ssec:Balslev-Op}

We define the vertical Laplace operator in analogy with the formula~$\LapFol_k =
\rmR_{k-2} \rmL_k$ in~\eqref{eq:foliated_laplace_operator} for the foliated one
as
\begin{gather}
\label{eq:def:vertical_laplace_operator}
  \LapVert
\,\defeq\,
 \rmRH \rmLH \=   y \del_z \del_{\bar z}
\tx{.}
\end{gather}
Note that it does not depend on the weight. The vertical Laplace operator does
not play a distinguished role by itself, but it is the foundation to define a
one parameter family of
\emph{compound Laplace operators} perturbing  the foliated one.
For~$\epsilon>0$, we set
\begin{gather}
\label{eq:def:compound_laplace_operator}
 - \LapCmp{\epsilon}_k
\defeq
  - \LapFol_k - \epsilon \LapVert
\tx{.}
\end{gather}
\par
\begin{lemma}
The foliated Laplace operator, the vertical Laplace operator and consequently
the family of compound Laplace operators are equivariant with respect to
the action of the special affine group, i.e., 
\ba
\label{eq:compound_laplace_operator_covariance}
  \big( \LapFol_k \phi \big) \big|'_{k}\, g
&\=
\LapFol_k \big( \phi \big|'_{k}\, g \big), \quad
  \big( \LapVert \phi \big) \big|'_{k}\, g
=
  \LapVert \big( \phi \big|'_{k}\, g \big) \\
  \big( \LapCmp{\eps}_k \phi \big) \big|'_{k}\, g
&\=
  \LapCmp{\epsilon}_k \big( \phi \big|'_{k}\, g \big)
\tx{.}
\ea
for all~$g \in \rmGp(\RR)$.
\end{lemma}
\par
\begin{proof}
One directly computes the covariance properties
\begin{gather}
\label{eq:maass_operators_covariance}
\begin{alignedat}{2}
  \big( \rmL_k \phi \big) \big|'_{k-2}\, g
&=
  \rmL_k \big( \phi \big|'_{k}\, g \big)
\tx{,}\qquad
&
  \big( \rmR_k \phi \big) \big|'_{k+2}\, g
&=
  \rmR_k \big( \phi \big|'_{k}\, g \big)
\tx{,}\\
  \big( \rmLH_k \phi \big) \big|'_{k-1}\, g
&=
  \rmLH_k \big( \phi \big|'_{k}\, g \big)
\tx{,}\qquad
&
  \big( \rmRH_k \phi \big) \big|'_{k+1}\, g
&=
  \rmRH_k \big( \phi \big|'_{k}\, g \big)
\end{alignedat}
\end{gather}
for any~$\phi \defcol \HSp \ra \CC$, any~$k \in \ZZ$, and any~$g \in
\rmGp(\RR)$, see also \cite[Section~2.1]{bump-1997} for the first two
equalities. (This goes back to the
general setup considered by Helgason~\cite{helgason-1959}, or also
\cite{ConleyRaum}.) The claimed equivariance follows directly from this.
\end{proof}
\par
The following proposition tells us that it is the compound Laplacian that has
better chances to have a good spectral decomposition for $\RL^2(\HSp)$.
\par
\begin{proposition} \label{prop:compoundproperties}
  For every $k \in \N$ and $\epsilon > 0$, the compound Laplace operator
  $-\LapCmp{\epsilon}_k$ is a self-adjoint elliptic operator on $\RL^2(\HSp)$, and the
  foliated Laplacian $-\LapFol_k$ is hypoelliptic. Furthermore, the symmetric bilinear form
  on $\RL^2(\cH(0,0),\nu_\MV)$
  associated with the compound Laplacian $-\LapCmp{\eps}_k$ is 
  \begin{equation}
      \begin{aligned}
          Q_k^{(\eps)}(\phi,\psi) &= \int_{\Gamma' \backslash \HSp} 
          y^k\, \nabla_{x,y} \phi \cdot
          \overline{\nabla_{x,y} \psi}\,  \de x \de y \de p \de q \\ 
          & \quad + \int_{\Gamma' \backslash \HSp} iky^{k-1} \del_x \phi
          \overline{\psi}\, \de x \de y \de p \de q \\
          & \quad + \eps \int_{\Gamma' \backslash \HSp} y^{k-2} \nabla_{u,v}
          \phi \cdot \overline{\nabla_{u,v} \psi }\, \de x \de y \de u \de v.
      \end{aligned}
  \end{equation}
\end{proposition}
\par
\begin{remark}
Note that the terms in the bilinear form are not all given in terms of the same
coordinates. While unconventional, this greatly simplifies the proof that
$\LapCmp{\eps}$ has compact resolvent when restricted to cusp forms.
\end{remark}
\par
\begin{proof}
We first observe that $iky\del_x$ is self-adjoint as the product of
commuting self-adjoint operators, and that the second term in the bilinear
form is simply $\langle iky\del_x \phi,\psi\rangle$. We turn our attention
to the rest of the foliated Laplacian, and use indices in differential operator
to represent the variables they are acting on:
\bas
\langle (- y^2 \Delta_{x,y} - ky \del_y) \phi,\psi\rangle
&\= \int_{\Gamma' \backslash
        \HSp} y^2((- \Delta_{x,y} - ky^{-1} \del_y) \phi) \overline \psi \frac{\de x \de y \de p \de q}{y^{2-k}} \\
&\=
 \int_{\Gamma' \backslash
        \HSp} ((- \Delta_{x,y} - ky^{-1} \del_y) \phi) \overline \psi y^k \de x \de y \de p \de q \\
    &\=  
    \int_{\Gamma' \backslash \HSp} \nabla_{x,y} \phi \cdot
    \nabla_{x,y}(y^k \overline \psi)\, \de x \de y \de p \de q \\
    &\qquad - k
    \int_{\Gamma' \backslash \HSp} \del_y\phi \cdot
    \psi\, y^{k-1}\, \de x \de y \de p \de q \\
    &\=  
    \int_{\Gamma' \backslash \HSp} y^k\, \nabla_{x,y} \phi \cdot
    \overline{\nabla_{x,y} \psi}\, \de x \de y \de p \de q\,,
\eas
which we recognise as the first term in the bilinear form. Finally, for the
vertical Laplacian we compute in $(x,y,u,v)$ coordinates:
\bas
      \langle-\eps \LapVert \phi,\psi\rangle &\= \eps \int_{\Gamma' \backslash \HSp}
        (-\Delta_{u,v}\phi)\overline \psi y^{k-2} \de x \de y \de u \de v \\
        &\=  \eps \int_{\Gamma' \backslash \HSp} y^{k-2} \nabla_{u,v} \phi
        \cdot\overline{\nabla_{u,v} \psi} \de x \de y \de u \de v.
\eas
    Self-adjointness is now a consequence of the fact that the Laplacian
    is represented by a symmetric bilinear form.
\par
For ellipticity, we express the vertical Laplacian also in $(x,y,p,q)$
coordinates as
\begin{equation}
    - \eps \LapVert = - \frac{\eps}{4} y \big( \del_q^2 + y^{-2} (\del_p - x \del_q)^2 \big)
\end{equation}
Viewing derivatives as tangent vectors, consider the following local
coordinates 
\bes
        T^*\HSp = \big\{(x,\xi,y,\eta,p,\zeta,q,\omega) : y > 0, \, \langle
            \xi,\del_x  \rangle = \langle \eta,\del_y \rangle = \langle
        \zeta,\del_p \rangle = \langle \omega, \del_q \rangle =1 \big\}.
        \ees
for the cotangent bundle, where we use $\langle \cdot, \cdot \rangle$ to denote
the pairing between the tangent and cotangent bundle. We can now read off the
expression of the foliated and compound Laplacian in coordinates that their
principal symbols are given by
\bes
    \symb\big( -\LapFol_k \big) \= y^2(\xi^2 + \eta^2)
\ees
and
\bes
    \symb\big( -\LapVert \big) \= \frac{y}{4} \big( \omega^2 + y^{-2} (\zeta - x
    \omega)^2 \big).
\ees
Note that these principal symbols do not depend on $k$, and that
an operator $P$ is hypoelliptic if $\symb(P)$ does not change sign and elliptic if $\symb(P) = 0$ implies that
    $\xi = \eta = \zeta = \omega = 0$. Hypoellipticity of the foliated Laplacian
    is directly observed.
Since $\symb(-\LapCmp{\eps}_k) = \symb(-\LapFol_k) + \symb(- \eps
\LapVert)$ and $y>0$, it is easy to see that $\symb(-\LapCmp{\eps}_k) > 0$
whenever $(\xi,\eta,\omega) \ne (0,0,0)$. On the other hand, if $\omega = 0$,
then any $\zeta \ne 0$ ensures the same thing, so that the operator is elliptic.
\end{proof}

\section{The special affine group and its representation theory}
\label{sec:RepTh}

The first goal of this section, Theorem~\ref{thm:Gp_classification_unitary_dual}
is to recall an application of Mackey theory and to classify the genuine
representations of $\rmG'(\RR) = \SAff2(\RR)$
up to isomorphism. These are the representations~$\pip_{n m^2}$
defined in~\eqref{eq:def:genuine_irrep}. The second goal of this section
is to compute the restrictions of these representations as representations
of $\rmG(\RR) = \SL2(\RR)$.

\subsection{The goal: decomposing the \texpdf{$\rmL^2$}{L2}-space}
\label{ssec:square_integrable_left_invariant_functions}

The Haar measure on~$\Gp(\RR)$ gives rise to a right-invariant measure
on~$\Gp(\ZZ) \backslash \Gp(\RR)$. We are interested in the space of
square-integrable functions on this quotient, $\rmL^2\big(\Gp(\ZZ) \backslash
\Gp(\RR)\big)$  on this quotient. This is the same as understanding the
space $\rmL^2(\cH(0,0), \nu_\MV)$ of square-integrable functions on the space of
tori with two marked points, equipped with the Masur--Veech measure, as
$\cH(0,0)$ and $\Gp(\ZZ) \backslash \Gp(\RR)$ differ by a set of
measure zero.
\par
By~\cite[Théorème 1]{dixmier-1957} the group~$\Gp(\RR)$ is of type~I.
In particular by~\cite[Theorem 6.D.7]{bekka-de-la-harpe-2020} we have
a direct integral decomposition
\begin{gather}%
\label{eq:direct_integral_decomposition_L2Gp}
  \rmL^2\big( \Gp(\ZZ) \backslash \Gp(\RR) \big)
\congwd
  \int^\oplus_{\htGp} \pi \,\rmd \mu_{\Gp}(\pi)
\tx{,}
\end{gather}
where~$\htGp$ is the unitary dual of~$\Gp(\RR)$. Restricting to
$G(\RR)$-representations gives us another direct integral decomposition
\begin{gather}%
\label{eq:direct_integral_decomposition_L2G}
  \Res_{\G(\RR)}^{\Gp(\RR)}\, \rmL^2\big( \Gp(\ZZ) \backslash \Gp(\RR) \big)
\congwd
  \int^\oplus_{\htG} \pi \,\rmd \mu_{\G}(\pi)
\tx{.}
\end{gather}
There is an embedding 
\begin{gather*}
    \rmL^2\big( \G(\ZZ) \backslash \G(\RR) \big)
\hra 
  \rmL^2\big( \Gp(\ZZ) \backslash \Gp(\RR) \big)
\tx{.}
\end{gather*}
Its range is consists in functions invariant under the action of the translation
subgroup~$\RR^2$ of~$\Gp(\RR)$.
\begin{definition}
We call 
\be
  \rmL^2\big( \Gp(\ZZ) \backslash \Gp(\RR) \big)^\gen
\defeqwd
  \rmL^2\big( \G(\ZZ) \backslash \G(\RR) \big)^{\perp}
\ee
the \emph{genuine part} of $\rmL^2\big( \Gp(\ZZ) \backslash \Gp(\RR) \big)$,
so that
\be \label{eq:L2Gp_genuine_decomposition}
  \rmL^2\big( \Gp(\ZZ) \backslash \Gp(\RR) \big)
\=
  \rmL^2\big( \G(\ZZ) \backslash \G(\RR) \big)
  \oplus
  \rmL^2\big( \Gp(\ZZ) \backslash \Gp(\RR) \big)^\gen
\tx{.}
\ee
\end{definition}
Since our goal is an explicit determination of the right-hand side of the
decompositions~\eqref{eq:direct_integral_decomposition_L2Gp}
and~\eqref{eq:direct_integral_decomposition_L2G},  we may of course
restrict attention to the genuine subspace. Integration along the
torus fibers of the projection $\cH(0,0) \to \cH(0)$ defines an
averaging map $\av: \rmL^2(\cH(0,0)) \to \rmL^2(\cH(0))$. Disintegrating
the Haar measure of~$\Gp(\RR)$ along the torus fibers shows
\be \label{eq:genuinechar}
 \rmL^2\big( \Gp(\ZZ) \backslash \Gp(\RR) \big)^\gen \= \Ker(\av)\,.
\ee
\par
\medskip
\paragraph{\textbf{Standard subgroups of $\Gp(\RR)$ and coordinates}}
We fix notation for the standard subgroups of $\SL2(\RR)$ (left) and the
special affine group (right), noting that we abuse notation and write again $g =
(g,0,0)$ for the image of an element of $\G$ in $\Gp$.
\begin{alignat*}{4}
  \A(\RR)
&\defeqwd
&&&
  \Ap(\RR)
&\defeqwd
  \big\{ \begin{psmatrix} a & 0 \\ 0 & a^{-1} \end{psmatrix}
  \condcol
  a \in \RR^\times
  \big\}
\tx{,}\quad
&&
\\
  \N(\RR)
&\defeqwd
  \big\{
  \begin{psmatrix} 1 & b \\ 0 & 1 \end{psmatrix}
  \condcol
  b \in \RR
  \big\}
\tx{,}\quad
&&
&
  \Np(\RR)
&\defeqwd
  \big\{
  \big( \begin{psmatrix} 1 & b \\ 0 & 1 \end{psmatrix}, w_1,w_2 \big)
  \condcol
  b, w_1,w_2  \in \RR
  \big\}
\tx{,}\quad
&&
\\
&&&&
  \Hp(\RR)
&\defeqwd
  \big\{
  \big( \begin{psmatrix} 1 & 0 \\ 0 & 1 \end{psmatrix}, w_1,w_2 \big)
  \condcol
  w_1,w_2 \in \RR
  \big\}
\tx{.}
&&
\end{alignat*}
which gives rise to the Iwasawa decomposition
\be \label{eq:iwasawa_decomposition}
  \Gp(\RR) \= \Np(\RR)\, \Ap(\RR)\, \rmK
\tx{,}
\ee
as well as two further decompositions
\bes
  \Gp(\RR) \= \G(\RR)\, \Hp(\RR) = \Hp(\RR\,) \G(\RR)
\tx{.}
\ees
We denote the parabolic subgroups in the Iwasawa decomposition by
\begin{gather}
    \begin{aligned}
\label{eq:def:parabolic_subgroup}
  \rmP(\RR)
&\=&
  \rmN(\RR)\, \rmA(\RR)
  &\=&
  \rmA(\RR)\, \rmN(\RR)
\tx{,}\\
  \rmPp(\RR)
&\=&
  \rmNp(\RR)\, \rmAp(\RR)
&\=&
  \rmAp(\RR)\, \rmNp(\RR)
\tx{.}
\end{aligned}
\end{gather}
\par
Given~$(\tau,z) \in \HSp$ as in~\eqref{eq:Jacobicoords}, we let~$a = \sqrt{y}$,
$b = x$, $w_1 = v \slash y$, and~$w_2 = u$. Then for all~$\theta \in \RR$ the
Iwasawa decomposition corresponds to the identity
\ba  \label{eq:nak_decomposition_to_HSp}
  (\tau, z) &\=  
  \big( \begin{psmatrix} 1 & b \\ 0 & 1 \end{psmatrix}, w_1, w_2 \big)\,
  \begin{psmatrix} a & 0 \\ 0 & a^{-1} \end{psmatrix}\,
  \begin{psmatrix} \cos\,\theta & \sin\,\theta \\ -\sin\,\theta & \cos\,\theta \end{psmatrix}\,
  (i,0) \\
&\=
  \Big(\begin{psmatrix} a & b a^{-1} \\ 0 & a^{-1} \end{psmatrix}, w_1 a, w_2 a^{-1} \Big)\,
  (i,0)\,\tx{.}
\ea
Alternatively, if we let~$a = \sqrt{y}$, $b = x$, $w_1 = p$ and~$w_2 = q$,
then emphasizing the coordinates~$z = p \tau + q$ for all~$\theta \in \RR$ we have
\begin{gather}
\label{eq:hnak_decomposition_to_HSp}
  \big( \begin{psmatrix} 1 & 0 \\ 0 & 1 \end{psmatrix}, w_1, w_2 \big)\,
  \begin{psmatrix} 1 & b \\ 0 & 1\end{psmatrix}\,
  \begin{psmatrix} a & 0 \\ 0 & a^{-1} \end{psmatrix}\,
  \begin{psmatrix} \cos\,\theta & \sin\,\theta \\ -\sin\,\theta & \cos\,\theta \end{psmatrix}
  \,(i,0)
\=
  (\tau,z)
\tx{.}
\end{gather}
The relation between these sets of coordinates is given by~$p = v \slash y$ and~$q = u - vx \slash y$. 
\par
In the coordinates of~\eqref{eq:nak_decomposition_to_HSp}, the Haar measures
on the groups~$\Np(\RR)$, $\Ap(\RR)$, $\rmK$, and~$\Np(\RR) \Ap(\RR) \rmK$ are given
respectively by
\begin{gather}
\label{eq:nak_haar_measures}
  \rd b \,\rd w_1 \,\rd w_2
\tx{,}\quad
  \frac{\rd a}{a}
\tx{,}\quad
  \frac{\rd \theta}{2 \pi}
\tx{,}\quad\tx{and}\quad
  \frac{\rd\theta\, \rd b\, \rd w_1 \,\rd w_2 \, \rd a}
       {2 \pi\, a^3}
\tx{.}
\end{gather}

\par
\medskip
\paragraph{\textbf{Notation for $\rmL^2$-induction of representations}}
We only consider the case of a locally compact group~$\rmG = \rmH \rmL$ for two
subgroups~$\rmH$ and~$\rmL$ such that~$\rmG \slash \rmH$ is isomorphic to~$\rmL$
as a measure
space. Our notation is consistent with \cite[Sections~1.5 and~5.2]{wallach-1988}
and~\cite[Section~2.1]{berndt-schmidt-1998}, which in turn
follows~\cite{kirillov-1976}.
\par
Given a representation~$\sigma$ of~$\rmH$ on a Hilbert space~$V(\sigma)$,
its $\rmL^2$\nbd{}induction to~$\rmG$ is given by right shifts on
\begin{multline*}
  V\big( \Ind_\rmH^\rmG\,\sigma \big)
\defeqwd
  \big\{
  f \defcol \rmG \ra V(\sigma) \condcol
  f \tx{\ measurable},\ 
  f \tx{\ square integrable on\ }L,
\\
  f(hg) \= \sqrt{\mfrac{\delta_\rmH(h)}{\delta_G(h)}}\, \sigma(h) f(g)
  \tx{\ for all\ }h \in \rmH, g \in \rmG
  \big\}\,,
\tx{,}
\end{multline*}
where $\Delta_\rmG$ and~$\Delta_\rmH$ are the modular functions on~$\rmG$
and~$\rmH$.
\par

\subsection{Representation theory of the upper triangular subgroup
\texorpdfstring{$\rmP(\RR)$}{P(R)}}

Consider the representations of $\rmP(\RR)$ that factor through
the quotient by $\rmN(\RR)$
\begin{gather}
\label{eq:def:sl2_parabolic_characters}
  \chiP_{+,s}
\defcol
  \begin{psmatrix} a & b \\ 0 & a^{-1} \end{psmatrix}
\mto
  |a|^s
\quad\tx{and}\quad
  \chiP_{-,s}
\defcol
  \begin{psmatrix} a & b \\ 0 & a^{-1} \end{psmatrix}
\mto
\sgn(a) |a|^{s}
\tx{,}\quad
  s \in \CC
  \tx{,}
\end{gather}
and abbreviate
\begin{gather}
\label{eq:def:sl2_parabolic_sign_character}
  \sgnP
\defeqwd
  \chiP_{-,0}
\tx{.}
\end{gather}
Further, through the paper, we set
\begin{gather*}
  e(x)
\defeqwd
  e^{2 \pi i\, x}
\tx{.}
\end{gather*}
\par
\begin{proposition}%
\label{prop:sl_parabolic_representations}
The irreducible representations that are trivial on~$\rmN(\RR)$ are given by
$\chiP_{-,s}$ and $\chiP_{+s,}$. They are unitary if and only if~$s \in i \RR$.
\par
The irreducible unitary representations which are not trivial on~$\rmN(\RR)$ 
are given, up to unitary equivalence, by
\begin{gather}
\label{eq:prop:sl_parabolic_representations}
  \piP_{\pm}
\defeqwd
  \Ind_{\rmN(\RR)}^{\rmP(\RR)}\,
  \big( \begin{psmatrix} 1 & b \\ 0 & 1 \end{psmatrix} \mto e(\pm b) \big)
\quad\tx{and}\quad
  \sgnP\, \piP_{\pm}
\tx{.}
\end{gather}
\end{proposition}
\par
\begin{proof}
For the first statement we observe that those representations factor through
the quotient~$\rmA(\RR) \cong \RR^\times$ and are thus characters.
For the second statement, consider the map
\begin{gather*}
  \begin{psmatrix} a & b \\ 0 & a^{-1} \end{psmatrix}
\mto
\begin{psmatrix} a^2 & b \\ 0 & 1 \end{psmatrix}\tx{,}
\end{gather*}
which is surjective with central kernel~$\pm 1$ onto the connected component of
the identity~$\SAff{1}(\RR)^0$ of the one-dimensional affine
group~$\SAff{1}(\RR)$,  see~\cite[Remark~3.C.6]{bekka-de-la-harpe-2020}. We
can thus apply the classification given there and append the central
character~$\sgnP$ to obtain the desired statement.
\end{proof}
\par
\par
\begin{proposition}%
\label{prop:sl_parabolic_regular_representations}
The regular representation of\/ $\rmP(\RR)$ decomposes as 
\begin{gather}
\label{eq:prop:sl_parabolic_regular_representations}
  \rmL^2\big( \rmP(\RR) \big)
\congwd
  \aleph_0\, \big(
  \piP_+ \oplus \piP_-
  \,\oplus\,
  \sgnP \piP_+ \oplus \sgnP \piP_-
  \big)\,
\tx{,}
\end{gather}
where the right hand side denotes a countably infinite direct
sum of the representation in the brackets.
\end{proposition}
\par
\begin{proof}
Let~$\rmP(\RR)^0 \subset \rmP(\RR)$ be the connected component of the identity. We use induction by steps and first decompose the regular representation of~$\rmN(\RR)$:
\begin{gather*}
  \rmL^2\big( \rmP(\RR) \big)
\=
  \Ind_1^{\rmP(\RR)}\, \bbone
\congwd
  \Ind_{\rmP(\RR)^0}^{\rmP(\RR)}\,
  \Ind_{\rmN(\RR)}^{\rmP(\RR)^0}\,
  \Ind_1^{\rmN(\RR)}\,
  \bbone
\tx{.}
\end{gather*}
Now by Fourier analysis, we have
\begin{gather*}
  \rmL^2\big( \rmN(\RR) \big)
\=
  \Ind_1^{\rmN(\RR)}\,
  \bbone
\congwd
  \int^\oplus_\RR
  \big( \begin{psmatrix} 1 & b \\ 0 & 1 \end{psmatrix} \mto e(nb) \big)
  \,\rmd n
\tx{.}
\end{gather*}
Using Fubini and the definition of induced representations, we see that
\begin{gather*}
  \Ind_{\rmN(\RR)}^{\rmP(\RR)^0}\,
  \int^\oplus_\RR
  \big( \begin{psmatrix} 1 & b \\ 0 & 1 \end{psmatrix} \mto e(nb) \big)
  \,\rmd n
\congwd
  \int^\oplus_\RR
  \Ind_{\rmN(\RR)}^{\rmP(\RR)^0}\,
  \big( \begin{psmatrix} 1 & b \\ 0 & 1 \end{psmatrix} \mto e(nb) \big)
  \,\rmd n
\tx{.}
\end{gather*}

Proceeding as in the proof of Proposition~\ref{prop:sl_parabolic_representations}, we recognize the inductions on the right hand side. They are isomorphic to the restriction of~$\piP_{\sgn(n)}$ to~$\rmP(\RR)^0$ if~$n \ne 0$, while if~$n = 0$ it is the regular representation of~$\rmP(\RR)^0 \slash \rmN(\RR)$. Since~$\{0\} \subset \RR$ has measure zero, we can discard its contribution to the direct integral. We conclude that
\begin{align*}
&
  \int^\oplus_\RR
  \Ind_{\rmN(\RR)}^{\rmP(\RR)^0}\,
  \big( \begin{psmatrix} 1 & b \\ 0 & 1 \end{psmatrix} \mto e(nb) \big)
  \,\rmd n
\\
\congwd{}&
  \int^\oplus_{\RR^+}
  \Res_{\rmP(\RR)^0}^{\rmP(\RR)}\, \piP_+
  \,\rmd n
  \,\oplus\,
  \int^\oplus_{\RR^-}
  \Res_{\rmP(\RR)^0}^{\rmP(\RR)}\, \piP_-
  \,\rmd n
\tx{.}
\end{align*}
The direct integrals on the right hand side have constant integrand. The direct integral over~$\RR^\pm$ yields countably infinite multiplicity by the isomorphism of Hilbert spaces~$\rmL^2(\RR) \cong \rmL^2(\ZZ_{> 0})$, so that we arrive at
\begin{gather*}
  \int^\oplus_\RR
  \Ind_{\rmN(\RR)}^{\rmP(\RR)^0}\,
  \big( \begin{psmatrix} 1 & b \\ 0 & 1 \end{psmatrix} \mto e(nb) \big)
  \,\rmd n
\congwd
  \aleph_0\, \big(
  \Res_{\rmP(\RR)^0}^{\rmP(\RR)}\, \piP_+
  \,\oplus\,
  \Res_{\rmP(\RR)^0}^{\rmP(\RR)}\, \piP_-
  \big)
\tx{.}
\end{gather*}
Finally, induction to~$\rmP(\RR)$ introduces the sign character~$\sgnP$ confirming the proposition.
\end{proof}

\subsection{Representation theory of \texorpdfstring{$\SL2(\RR)$}{SL2(R)}: a brief summary}

The following results are standard and appear in many text books, e.g.\@
\cite{Knapp,wallach-1988}.
We set
\begin{gather*}
 \ISL_{+, s}
\defeqwd
  \Ind_{\rmP(\RR)}^{\G(\RR)}\,
  |a|^{s + 1}
\tx{,}\quad
  \ISL_{-, s}
\defeqwd
  \Ind_{\rmP(\RR)}^{\G(\RR)}\,
  \sgn(a) |a|^{s + 1}
\tx{.}
\end{gather*}
Note that the shift~$s+1$ in the exponents is chosen in such a way that
purely imaginary~$s$ correspond to unitary representation.
We have a duality between~$\ISL_{\epsilon,s}$ and~$\ISL_{\epsilon,-s}$ via intertwining
operators explained in~\cite[Section~5.3]{wallach-1988}. If~$k \in \ZZ_{> 0}$, then~$\ISL_{-,k-1}$
is reducible with two infinite-dimensional constituents~$\DSL_{\pm k}$, which
are discrete series if~$k > 1$ and limits of discrete series if~$k = 1$.
\par
\begin{theorem}[Bargmann]
\label{thm:sl2_langlands_classification_unitary}
The irreducible unitary representations of\/~$\SL{2}(\RR)$ are given up to unitary equivalence by
\begin{enumerate}[label=(\alph*)]
\item the principal series~$\ISL_{\epsilon, s} \cong \ISL_{\epsilon, -s}$ for~$\epsilon = +$ and~$s \in i \RR$ or~$\epsilon = -$ and~$s \in i \RR \setminus \{0\}$,
\item the complementary series~$\ISL_{+, s}$ for~$s \in \RR$, $0 < |s| < 1$,
\item the (limits of) discrete series representations~$\DSL_k$ for~$k \in \ZZ \setminus \{0\}$, and
\item the trivial representation
\end{enumerate}
\end{theorem}
\par
Section~5.6.4 of~\cite{wallach-1988} also provides a list of which of these representations are square-integrable or tempered, which allows us to deduce the Plancherel measure of regular representations of~$\SL{2}(\RR)$.
\par
\begin{theorem}
\label{thm:sl2_langlands_classification_tempered}
Among the representations in Theorem~\ref{thm:sl2_langlands_classification_unitary}, the ones contained in the regular representation~$\rmL^2(\SL{2}(\RR))$ are
\begin{enumerate}[label=(\alph*),series=sl2_langlands_classification_tempered]
\item the discrete series representations~$\DSL_k$ for~$k \in \ZZ \setminus \{0, \pm 1\}$
\tx{.}
\end{enumerate}
Beyond those, the ones that are weakly contained in~$\rmL^2(\SL{2}(\RR))$ are
\begin{enumerate}[label=(\alph*),resume=sl2_langlands_classification_tempered]
\item the principal series~$\ISL_{\epsilon, s} \cong \ISL_{\epsilon, -s}$ for~$\epsilon = +$ and~$s \in i \RR$ or~$\epsilon = -$ and~$s \in i \RR \setminus \{0\}$,
\item the limits of discrete series representations~$\DSL_k$ for~$k \in \{\pm 1\}$.
\end{enumerate}
\end{theorem}

\subsection{Representation theory of \texorpdfstring{$\SAff2(\RR)$}{SAff2(R)}: Mackey theory}

We summarize the general Mackey theory for representations of semidirect
products in our special case of $\rmGp(\RR) = \SAff2(\RR)$. As a prerequisite
we need to understand representations of $\rmN'(\RR)$. 
\par
\begin{proposition}%
\label{prop:stone_von_neumann}
The irreducible unitary representations with nontrivial central
character of\/~$\rmNp(\RR)$ are given up to unitary equivalence by
\begin{gather}
\label{eq:thm:stone_von_neumann}
  \piN_r
\defeqwd
  \Ind_{\Hp(\RR)}^{\Np(\RR)}\,
  \big( (w_1,w_2) \mto e(r w_2) \big)
\tx{,}\quad
  r \in \RR^\times
\tx{.}
\end{gather}
\par
The unitary irreducible representations of $\Np(\RR)$ with trivial central
character are the characters 
\begin{gather}
\label{eq:def:unitary_character_Np}
  \chiN_{n,m}
\defcol
  \big( \begin{psmatrix} 1 & b \\ 0  & 1 \end{psmatrix}, w_1, w_2 \big)
\mto 
  e(n b + m w_1)
\tx{,}\quad
  n, m \in \RR
\tx{.}
\end{gather}
\end{proposition}
\par
\begin{proof} This is an instance the Stone--von Neumann theorem, see
for example \cite[Exercise 32.5]{bump-2013} or \cite[Example~7.3.3]{zimmer-1984}.
\end{proof}
\par
We can now state the result for the special affine group.
\par
\begin{theorem}%
\label{thm:Gp_classification_unitary_dual}
The unitary dual of\/~$\Gp(\RR)$ is exhausted by the pullback of\/~$\SL{2}(\RR)$\nbd{}representations and by, for any fixed~$m \in \RR^\times$, the representations
\begin{gather}
\label{eq:def:genuine_irrep_as_induction}
  \pip_{n,m}
\defeqwd
  \Ind_{\Np(\RR)}^{\Gp(\RR)}\,\chiN_{n,m}
\tx{,}\quad
  n \in \RR
\tx{.}
\end{gather}
\par
The two unitary representations~$\pip_{n_1,m_1}$ and~$\pip_{n_2,m_2}$ are
isomorphic if and only if~$n_1 m_1^2 = n_2 m_2^2$.
\end{theorem}
\par
We fix representatives of the isomorphism classes as the representations
\begin{gather}
\label{eq:def:genuine_irrep}
  \pip_{s} \defeqwd  \pip_{s, 1}
\end{gather}
with a single index. The proof of Theorem~\ref{thm:Gp_classification_unitary_dual} requires the
next statement, which we will also need independently.
\par
\begin{proposition}%
\label{prop:totCasiEigenvalues}
The total Casimir operator acts on the representation $\pip_{n, m}$ with
eigenvalue~$- 4 \pi^3\, n m^2$.
\end{proposition}%
\par
\begin{proof}
We observe that the correspondence in~\eqref{eq:def:modular_to_automorphic} allows us to view~$\rmK$\nbd{}isotypical elements of~$V(\pisaff_{n,m})$ as functions on~$\HS'$, and then calculate with the total Laplace operator via~\eqref{eq:def:foliated_laplace_and_casimir_operator}. The~$\rmK$\nbd{}spherical element of~$V(\pisaff_{n,m})$ that is constant on~$\rmA'(\RR)$ corresponds to~$e( n x + m v \slash y ) $. We have
\begin{gather}
\label{eq:casimir_eigenvalue_fourier_term}
 - \LapTot_k\,
  e\big( n x + m \mfrac{v}{y} \big)
\=
  4 \pi^3 n m^2\,
  e\big( n x + m \mfrac{v}{y} \big)
\tx{,}
\end{gather}
by a straightforward calculation using the formula for the total Laplace
operator in Lemma~\ref{la:DescendtoLaplace}, once we observe that the given
expression is independent of~$u$.
\end{proof}
\par
\begin{proof}[Proof of Theorem~\ref{thm:Gp_classification_unitary_dual}]
The first statement is a reformulation of Theorem~2.4.2 of~\cite{berndt-schmidt-1998}. More precisely, the representations in part~i) of loc.\ cit.\  theorem are pullbacks from~$\SL{2}(\RR)$. The representations in part~ii) are the representations in our theorem. The character~$\psi$ in the notion of~\cite{berndt-schmidt-1998} is non-trivial and thus corresponds to~$x \mto e(m x)$ for an arbitrary but fixed~$m \in \RR^\times$. The character
\begin{gather*}
  \big( \begin{psmatrix} 1 & b \\ 0 & 1 \end{psmatrix}, w_1,w_2 \big)
\mto
  \psi(b n \slash m + w_1)
\=
  e(n b + m w_1)
\end{gather*}
is genuine for any~$n \in \RR$. Conversely every genuine character of~$\Np(\RR)$
is of this form by Proposition~\ref{prop:stone_von_neumann}.
See also e.g.\@ \cite[Example 7.3.4]{zimmer-1984}.
\par
By the first statement, if~$m_1 \ne m_2$ for each~$n_1$ there is exactly one~$n_2$
such that the two inductions are isomorphic. To determine this value we employ the eigenvalue under the total Casimir operator as given in Proposition~\ref{prop:totCasiEigenvalues}.
\end{proof}

\subsection{Restrictions of \texorpdfstring{$\SAff2(\RR)$}{SAff2(R)}-representations}

Our goal is to prove the following branching of~$\pisaff_{n,m}$ to~$\SL{2}(\RR)$.
\par
\begin{proposition}%
\label{prop:saff_representation_sl_branching}
For any~$m \in \RR^\times$ and~$n \in \RR$, the restrictions of the
$\SAff2(\RR)$-representations to $\SL{2}(\RR)$ decompose as direct integrals
\begin{gather*}
\begin{aligned}
  \Res_{\G(\RR)}^{\Gp(\RR)}\, \pisaff_{0,m}
&\=
  \int^\oplus_{\RR^+}
  2 \big( \ISL_{+,i t} \oplus \ISL_{-,i t} \big)
  \,\rmd t
\tx{,}
\\
  \Res_{\G(\RR)}^{\Gp(\RR)}\, \pisaff_{n,m}
&\=
  \bigoplus_{k = 2}^\infty
  \DSL_{\sgn(n) k}
  \,\oplus\,
  \int^\oplus_{\RR^+}
  \big( \ISL_{+,i t} \oplus \ISL_{-,i t} \big)
  \,\rmd t
\tx{,}\quad
  \tx{if\ } n \ne 0 
\tx{.}
\end{aligned}
\end{gather*}
\end{proposition}
\par
We prove this proposition at the end of the section. The proof features various intermediate inductions and restrctions, which we exhibit separately. It will depend on Lemmas~\ref{la:induction_oneN_to_parabolic},~\ref{la:induction_chiN_to_parabolic},~\ref{la:principle_series_restriction_N}, and~\ref{la:induction_piP_from_parabolic}, that we state and prove first.
\par
\begin{lemma}%
\label{la:induction_oneN_to_parabolic}
Given any~$m \in \RR^\times$, we have the decomposition into characters
\begin{gather}
\label{eq:la:induction_oneN_to_parabolic}
  \Ind^{\rmP(\RR)}_{\N(\RR)}\;
  \Res^{\Np(\RR)}_{\N(\RR)}\,
  \chiN_{0,m}
\congwd
  \int^\oplus_{\RR}\,
  \big( \chiP_{+,i t} \oplus \chiP_{-,i t} \big) \,\rd t
\=
  \int^\oplus_{\RR}\,
  \big( \bbone \oplus \sgnP \big) \chiP_{+,i t} \,\rd t
\tx{.}
\end{gather}
\end{lemma}
\par
\begin{proof}
We write~$\A(\RR)^0$ and~$\rmP(\RR)^0$ for the  connected components of the
identity of~$\A(\RR)$ and~$\rmP(\RR)$ respectively, that is, the subgroups
whose elements have positive diagonal entries. Transitivity of induction yields
\begin{gather*}
  \Ind^{\rmP(\RR)}_{\N(\RR)}\,
  \Res^{\Np(\RR)}_{\N(\RR)}\,
  \chiN_{0,m}
\congwd
  \Ind^{\rmP(\RR)}_{\rmP(\RR)^0}\,
  \Ind^{\rmP(\RR)^0}_{\N(\RR)}\,
  \Res^{\Np(\RR)}_{\N(\RR)}\,
  \chiN_{0,m}
\tx{.}
\end{gather*}
Functions in the representation space of the inner induction are left invariant under~$\N(\RR)$. Since~$\N(\RR) \subset \rmP(\RR)^0$ is normal, we can identify them with square-integrable functions on
\begin{gather*}
  \N(\RR) \big\backslash \rmP(\RR)^0
\congwd
  \A(\RR)^0
\congwd
  \RR^+
\tx{.}
\end{gather*}
That is, we have to determine the decomposition of~$\RL^2(\RR^+)$ as a representation of~$\RR^+$. We use the map~$\RR^+ \ra \RR$, $a \mto \log(a)$ and classical Fourier analysis to find
\begin{gather*}
  \Ind^{\rmP(\RR)}_{\rmP(\RR)^0}\,
  \Ind^{\rmP(\RR)^0}_{\N(\RR)}\,
  \Res^{\Np(\RR)}_{\N(\RR)}\,
  \chiN_{0,m}
\congwd
  \Ind^{\rmP(\RR)}_{\rmP(\RR)^0}\,
  \int^\oplus_{\RR}\,
  \chiP_{+,i t} \,\rd t
\congwd
  \int^\oplus_{\RR}\,
  \Ind^{\rmP(\RR)}_{\rmP(\RR)^0}\,
  \chiP_{+,i t} \,\rd t
\tx{,}
\end{gather*}
where for simplicity we identify~$\chiP_{+,i t}$ with its restrction to~$\rmP(\RR)^0$. The remaining induction is central, contributing one copy of~$\chiP_{+,i t}$ and one of~$\chiP_{-,i t} = \sgnP \chiP_{+,i t}$ as desired.
\end{proof}
\par
We now turn to the case $n \neq 0$.
\par
\begin{lemma}%
\label{la:induction_chiN_to_parabolic}
Given any~$m \in \RR^\times$ and $n \neq 0$, we have the decomposition into irreducible representations
\begin{gather*}
  \Ind^{\rmP(\RR)}_{\N(\RR)}\;
  \Res^{\Np(\RR)}_{\N(\RR)}\,
  \chiN_{n,m}
\congwd
  \piP_{\sgn(n)}
  \oplus
  \sgnP
  \piP_{\sgn(n)}
\tx{.}
\end{gather*}
\end{lemma}
\par
\begin{proof}
Inducing in steps to~$\rmP(\RR)^0$ as in the proof of Lemma~\ref{la:induction_oneN_to_parabolic} we can apply Mackey's Irreducibility Criterion (Corollary~1.F.5 of~\cite{bekka-de-la-harpe-2020}; compare also Remark~3.C.6 of op.\@ cit.) to obtain the restriction of~$\piP_{\sgn(n)}$ to~$\rmP(\RR)^0$. The central induction from~$\rmP(\RR)^0$ to~$\rmP(\RR)$ then yields two copies of it, one of which is twisted by the sign character.
\end{proof}
\par
The proof of~Lemma~\ref{la:induction_piP_from_parabolic}, the last auxiliar statement for the proof of Proposition~\ref{prop:saff_representation_sl_branching}, requires us
to determine the restrictions of principal and discrete series of~$\SL{2}(\RR)$
to~$\rmN(\RR)$. The statements of the next lemma are given, for instance, in
Proposition~3.3.2 and~3.3.3 of Kobayashi's notes~\cite{kobayashi-2005}.
\par
\begin{lemma}%
\label{la:principle_series_restriction_N}
For all~$t \in \RR$, we have the decomposition
\begin{gather*}
  \Res^{\G(\RR)}_{\rmN(\RR)}\,
  \ISL_{\pm,it}
\congwd
  \int^\oplus_{\RR}
  \big( \begin{psmatrix} 1 & b \\ 0 & 1 \end{psmatrix} \mto e(nb) \big)
  \,\rd n
\tx{.}
\end{gather*}
For all~$k \in \ZZ \setminus \{0\}$, we have the decomposition
\begin{gather*}
  \Res^{\G(\RR)}_{\rmN(\RR)}\,
  \DSL_k
\congwd
  \int^\oplus_{\RR^{\sgn(k)}}
  \big( \begin{psmatrix} 1 & b \\ 0 & 1 \end{psmatrix} \mto e(nb) \big)
  \,\rd n
\tx{.}
\end{gather*}
\end{lemma}
\par
\begin{lemma}%
\label{la:induction_piP_from_parabolic}
The inductions from the upper triangular subgroup~$\rmP(\RR)$ to $\G(\RR)$
decompose as follows:
\ba
\label{eq:la:induction_piP_from_parabolic}
  \Ind^{\G(\RR)}_{\rmP(\RR)}\,\piP_\pm
&\congwd
  \bigoplus_{\substack{k = 2\\k \tx{\ even}}}^\infty
  \DSL_{\pm k}
  \,\oplus\,
  \int^\oplus_{\RR^+}\,
  \ISL_{+,i t} \,\rd t
\tx{,} \\
\Ind^{\G(\RR)}_{\rmP(\RR)}\,\sgnP \piP_\pm
&\congwd
  \bigoplus_{\substack{k = 3\\k \tx{\ odd}}}^\infty
  \DSL_{\pm k}
  \,\oplus\,
  \int^\oplus_{\RR^+}\,
  \ISL_{-,i t} \,\rd t
\tx{.}
\ea
\end{lemma}
\par
\begin{proof}
The strategy is to spell out Mackey's version of Frobenius
reciprocity~\cite{mackey-1953} for the inclusion of groups~$\rmP(\RR)$ and~$\G(\RR)$
and then apply the same technique as in Section~7(b) of loc.\ cit. Namely,
we determine the isomorphism class of all restrictions on the right hand side
of equality~\eqref{eq:la:prf:induction_piP_from_parabolic:mackey-diag} below,
and consider this equation as representation of~$\rmP(\RR) \times 1
\subseteq \rmP(\RR) \times \rmG(\RR)$. Since only
finitely many, four in fact, isomorphism classes
of~$\rmP(\RR) \times 1$\nbd{}representations contribute, we thus derive an
isomorphism of isotypical components for~$\rmP(\RR) \times 1$ on the left and
right hand side as representations of the commutator~of~$\rmP(\RR) \times 1$,
thus of representations of the group~$1 \times \rmG(\RR)$, as desired.
\par
The Frobenius reciprocity formula involves all the representations weakly
contained in the regular representations, which have been listed in 
Proposition~\ref{prop:sl_parabolic_regular_representations} and
Theorem~\ref{thm:sl2_langlands_classification_tempered} for the two groups
under consideration. We deduce from Mackey's theorem that
\begin{gather}
\label{eq:la:prf:induction_piP_from_parabolic:mackey-diag}
\begin{aligned}
&
\begin{alignedat}{3}
  \big( \ovpiP_+ &\otimes \Ind_{\rmP(\RR)}\,\piP_+ \big)
  \;\;&&\oplus\;\;&
  \big( \ovpiP_- &\otimes \Ind_{\rmP(\RR)}\,\piP_- \big)
\\
  \oplus\quad
  \big( \ovsgnP \ovpiP_+ &\otimes \Ind_{\rmP(\RR)}\,\ovsgnP \piP_+ \big)
  \;\;&&\oplus\;\;&
  \big( \ovsgnP \ovpiP_- &\otimes \Ind_{\rmP(\RR)}\,\ovsgnP \piP_- \big)
\end{alignedat}
\\
\cong{}&
  \bigoplus_{k \in \ZZ \setminus \{0,\pm 1\}}
  \Res_{\rmP(\RR)}^{\rmG(\RR)}\big( \ovDSL_k \otimes \DSL_k \big) \\
  & \quad \;\;\oplus\;\;
  \int_{\RR} \Res_{\rmP(\RR)}^{\rmG(\RR)}\big( \ovISL_{+,i t} \otimes \ISL_{+,i t} \big) \,\rd t
  \;\;\oplus\;\;
  \int_{\RR} \Res_{\rmP(\RR)}^{\rmG(\RR)}\big( \ovISL_{-,i t} \otimes \ISL_{-,i t} \big) \,\rd t
\tx{.}
\end{aligned}
\end{gather}
where the overline denotes the contragredient representation. These are
easily computed to be $\ovsgnP = \sgnP$ and $\ovchiN_n = \chiN_{-n}$,
hence the induction satisfies~$\ovpiP_\pm \cong \piP_\mp$. By a similar argument,
we find that for~$t \in \RR$
\begin{gather*}
\ovISL_{\pm,i t} \,\cong \,\ISL_{\pm,\ov{i t}} \= \ISL_{\pm,-i t} \,\cong \,\ISL_{\pm,i t}\,
\quad \text{and} \quad {\ovDSL_k} \= \DSL_k
\end{gather*}
by inspection of~$\rmK$\nbd{}types.
\par
We next determine the multiplicities of~$\piP_\pm$ and~$\sgnP \piP_\pm$ in the
restrictions on the right hand side. Since central characters are preserved by
restriction, there are multiplicities~$m_{\rmD,k;\pm}$ and~$m_{\rmI,\pm,it;\pm}$,
which are possibly infinite, such that
\begin{alignat*}{3}
  \Res^{\G(\RR)}_{\rmP(\RR)}\, \DSL_k
&\congwd
&
  \big( \sgnP \big)^k\,
&
  \big( m_{\rmD,k;+} \piP_+ &&\,\oplus\, m_{\rmD,k;-} \piP_- \big)
\tx{,}
\\
  \Res^{\G(\RR)}_{\rmP(\RR)}\, \ISL_{+,i t}
&\, \cong \,
&&
  \big( m_{\rmI,+,i t;+} \piP_+ &&\,\oplus\, m_{\rmI,+,i t;-} \piP_- \big)
\tx{,}
\\
  \Res^{\G(\RR)}_{\rmP(\RR)}\, \ISL_{-,i t}
&\congwd
&
  \sgnP\,
&
  \big( m_{\rmI,-,i t;+} \piP_+ &&\,\oplus\, m_{\rmI,-,i t;-} \piP_- \big)
\tx{.} 
\end{alignat*}
To find these multiplicities we restrict both sides to~$\N(\RR)$ and
note that
\begin{gather}
\label{eq:la:prf:induction_piP_from_parabolic:restriction_multiplicities}
  \Res^{\G(\RR)}_{\rmN(\RR)} \piP_\pm
\congwd
  \Res^{\G(\RR)}_{\rmN(\RR)} \sgnP \piP_\pm
\congwd
  \int^\oplus_{\RR^\pm}
  \big( \begin{psmatrix} 1 & b \\ 0 & 1 \end{psmatrix} \mto e(n b) \big)
  \,\rd n
\tx{.}
\end{gather}
Comparing this with Lemma~\ref{la:principle_series_restriction_N} below
implies that
\begin{gather*}
  m_{\rmD,k,\sgn(k)}
\=
 1
\tx{,} \quad
  m_{\rmD,k,\sgn(-k)}
\=
  0
\quad\tx{and}\quad
  m_{\rmI,\pm,it;+}
\=
  m_{\rmI,\pm,it;-}
\=
  1
\tx{.}
\end{gather*}
Coming back to~\eqref{eq:la:prf:induction_piP_from_parabolic:mackey-diag}, viewed
as a representation of~$\rmP(\RR) \times 1$, and using the multiplicities
in~\eqref{eq:la:prf:induction_piP_from_parabolic:restriction_multiplicities},
we deduce our statement.
\end{proof}
\par
\begin{proof}[Proof of Proposition~\ref{prop:saff_representation_sl_branching}]
Comparing the decompositions
\begin{gather*}
  \Gp(\RR) \= \Np(\RR) \A(\RR) \rmK
\tx{,}\quad
  \G(\RR) \= \N(\RR) \A(\RR) \rmK
\tx{,}\quad\tx{and\ }
  \Np(\RR) \= \Hp(\RR) \N(\RR)
\end{gather*}
we claim that there is an isomorphism
\begin{gather}
\label{eq:prop:saff_representation_sl_branching:iwasawa}
\Res_{\rmG(\RR)}^{\Gp(\RR)}\, \pisaff_{n,m}
\= 
  \Res_{\G(\RR)}^{\Gp(\RR)}\;
  \Ind_{\Np(\RR)}^{\Gp(\RR)}\,
  \chiN_{n,m}
\congwd
  \Ind_{\N(\RR)}^{\G(\RR)}\;
  \Res_{\N(\RR)}^{\Np(\RR)}\,
  \chiN_{n,m}
\tx{.}
\end{gather}
In fact, the left hand side of~\eqref{eq:prop:saff_representation_sl_branching:iwasawa} consists of functions that transform like
\be \label{eq:transform_of_restr}
  f(\td{n}g) \= \chiN_{n,m}(\td{n}) f(g)
\quad
 \tx{for all\ } \td{n} \in \Np(\RR) = \N(\RR) \Hp(\RR)
 \tx{\ and\ } g \in \Gp(\RR)
 \tx{.}
\ee
In particular, such~$f$ is uniquely defined by its restriction to~$G(\RR)
\cong \Gp(\RR) \slash \Hp(\RR)$, and this restriction satisfies
again~\eqref{eq:transform_of_restr}, now for all $\td{n} \in \N(\RR)$ and
$g \in \G(\RR)$. Moreover, by the left covariance with respect to~$\Hp(\RR)$,
the function~$f$ is measurable if and only if its restriction to~$\G(\RR)$ is so.
Square integrability on~$\A(\RR) \rmK \subset \G(\RR)$ is preserved by
the restriction to~$\G(\RR)$.
\par
We can perform the induction on the right hand
of~\eqref{eq:prop:saff_representation_sl_branching:iwasawa} side in steps,
that is, 
\begin{gather*}
  \Ind_{\N(\RR)}^{\G(\RR)}\;
  \Res_{\N(\RR)}^{\Np(\RR)}\,
  \chiN_{n,m}
\,\cong \,
  \Ind_{\rmP(\RR)}^{\G(\RR)}\,
  \Ind_{\N(\RR)}^{\rmP(\RR)}\;
  \Res_{\N(\RR)}^{\Np(\RR)}\,
  \chiN_{n,m}
\tx{.}
\end{gather*}
In the case of~$n = 0$, Lemma~\ref{la:induction_oneN_to_parabolic} implies that
\begin{gather*}
  \Ind_{\N(\RR)}^{\G(\RR)}\;
  \Res_{\N(\RR)}^{\Np(\RR)}\,
  \chiN_{0,m}
\cong 
  \Ind_{\rmP(\RR)}^{\G(\RR)}\,
  \int^\oplus_{i\RR} \big( \chiP_{+,s} \oplus \chiP_{-,s} \big)
  \,\rd s
\tx{.}
\end{gather*}
To obtain the result, it suffices to note that induction and direct integral decomposition intertwine by Fubini's theorem, and to use that~$\ISL_{\pm,i t} \cong \ISL_{\pm,-i t}$.
\par
In the case~$n \neq 0$ the statement now follows directly by combining Lemma~\ref{la:induction_chiN_to_parabolic} and Lemma~\ref{la:induction_piP_from_parabolic}.
\end{proof}

\section{Fourier expansions and Poincar\'e series}%
\label{sec:fourier_expansions}

The first goal of this section is to examine and relate to each other two
Fourier expansions of affine modular forms. One of them along the torus fiber is
merely compatible with the action of~$\SL{2}(\RR)$, but exhibits better
compatibility with the action of the foliated Laplace operator and with the
corresponding notion of Eisenstein and Poincaré series. The other one arises
from the Heisenberg subgroup~$\Np(\RR) \subset \Gp(\RR)$ and is compatible with
the description of $\Gp(\RR)$\nbd{}representations that appears in
Theorem~\ref{thm:Gp_classification_unitary_dual}.
In particular we give in
Proposition~\ref{prop:fourier_term_heisenberg_to_whittaker_model} a criterion in
terms of Fourier coefficients that ensures that some
representation~$\pisaff_{nm^2}$
appears in the $\rmL^2$-representation~$\piLsq(\phi)$ generated by a modular form~$\phi$. 
\par
In Section~\ref{ssec:eisenstein_poincare_series} we construct affine group versions
of Eisenstein and Poincar\'e series. We compute their Fourier coefficients and show
in Proposition~\ref{prop:affine_eisenstein_poincare_series_representation} that they generate the representation $\pip_{n}$ for all values~$n \in \ZZ$. As a corollary we
obtain the eigenvalues of the total Casimir acting on~$\pip_n$.
\par
Parts of the Fourier expansion along the torus fiber for affine
modular-invariant functions of $\rmK$-type $0$ and, expressed differently,
the related construction of Eisenstein and Poincar\'e series has appeared in
unpublished work of Balslev~\cite{balslev-2011}.

\subsection{Fourier expansions on the torus fibers}%
\label{ssec:fourier_expansions_torus_fiber}

The subgroup of translations is an abelian subgroup isomorphic to~$\RR^2$
in~$\SAff2(\RR)$, whose dual yields a Fourier expansion of continuous
affine modular-invariant function~$\phi$ of weight~$k$:
\begin{gather}
\label{eq:affine_modular_form_fourier_expansion_torus}
  \phi(\tau,z)
\=
  \sum_{\substack{r,m \in \ZZ}}
  c^\rmT\big( \phi;\, m,r;\, \tau \big)\,
  e(m p + r q)\,, \qquad (z \= p \tau + q)
\tx{.}
\end{gather}
The Fourier coefficients are of course given by
\begin{gather}
\label{eq:coefficients_FT}
  c^{\rmT}(\phi;m,r;\tau) = \int_{\ZZ \backslash \RR} \int_{\ZZ \backslash
  \RR}  \phi(\tau,p\tau + q) e(-mp - rq) \, \de p \de q
\tx{.}
\end{gather}

The superscript~$\rmT$ indicates that this is the Fourier expansion along
the torus fibers of the projection $\cH(0,0) \to \cH(0)$.
\par
\begin{lemma}%
\label{la:affine_modular_form_fourier_expansion_torus_sl2_symmetry}
The Fourier coefficients have the equivariance property
\begin{gather*}
  c^\rmT\big( \phi \big|'_{k}\, \ga ;\, m,r;\, \tau \big)\,
\=
  c^\rmT\big( \phi ;\, \wt m, \wt r;\, \tau \big)\,
  \big|_k\, \ga
\tx{,}\quad
  (\wt m, \wt r) = (m,r) \rT \ga
\tx{.}
\end{gather*}
for a continous affine modular-invariant function~$\phi$ of weight~$k$
and~$\ga \in \SL{2}(\ZZ)$.
\end{lemma}
\par
\begin{proof}
We apply~$\ga$ to~$\phi$ and eventually compare coefficients, using the
uniqueness of the Fourier series. We get
\begin{align*}
&
  \sum_{\substack{r,m \in \ZZ}}
  c^\rmT\big( \phi;\, m,r;\, \tau \big)\,
  e(m p + r q)
\= \phi(\tau,z) \= 
  \big( \phi \big|'_k\,\ga \big) (\tau,z)
\\
\={}&
  \sum_{\substack{\wt m, \wt r \in \ZZ}}
  (c \tau + d)^{-k}
  c^\rmT\big( \phi;\, \wt m, \wt r;\, \ga \tau \big)\,
  e(\wt m \wt p + \wt r \wt q)
\tx{,}
\end{align*}
where~$(\wt p, \wt q) = (p,q) \ga^{-1}$ and $\ga = \begin{psmatrix} a&b\\c&d
\end{psmatrix}$, since
\bes
\mfrac{z}{c \tau + d}  \= \wt p (\ga \tau) + \wt q\quad
  \tx{with\ }
  (\wt p, \wt q) = (p,q) \ga^{-1}\,.
\ees
Now to compare coefficients, we need to determine the~$(\wt m,\wt r)$ for
which~$m p + r q = \wt m \wt p + \wt r \wt q$. This yields the equality
\begin{align*}
&
  (p,q) \rT(m,r) \= m p + r q \=
  \wt m \wt p + \wt r \wt q
\\
={}&
  (\wt p,\wt q) \rT(\wt m,\wt r)
\=
  (p,q) \ga^{-1} \rT(\wt m,\wt r)
\=
  (p,q) \rT\big( (\wt m,\wt r) \rT\ga^{-1} \big)
\tx{,}
\end{align*}
which yields desired relation between~$(m,r)$ and~$(\wt m, \wt r)$.
\end{proof}
\par\begin{proposition}%
\label{prop:affine_modular_form_vanishing_by_torus_fourier_coefficients}
Let~$\phi$ be an affine modular-invariant function of weight~$k$
such that the Fourier coefficients~$c^\rmT(\phi;\, m, 0; \tau)$ vanishes for
all~$m \in \ZZ$. Then~$\phi = 0$.
\end{proposition}
\par
\begin{proof}
By modular invariance under~$\SL{2}(\ZZ)$ and Lemma~\ref{la:affine_modular_form_fourier_expansion_torus_sl2_symmetry}, every~$c^\rmT$ with index~$(m,0) \ga$ for some~$\ga$ and some~$m$ vanishes. This exhausts all terms in the Fourier expansion~\eqref{eq:affine_modular_form_fourier_expansion_torus}, implying~$\phi = 0$.
\end{proof}
\par

\subsection{Fourier series along the Heisenberg group}%
\label{ssec:fourier_series_classical}

We now study a Fourier expansion of affine modular function that is compatible
with the induction of characters that appear in
Theorem~\ref{thm:Gp_classification_unitary_dual}. Specifically, for given
$m \in \ZZ \setminus \{0\}, n \in \ZZ$
we show that the nonvanishing of specific Fourier coefficients implies that
the~$\SAff{2}(\RR)$\nbd{}representation of the representation generated by an
affine modular-invariant function contains~$\pip_{nm^2}$.
\par
The dual of the abelian group
\bes
  \big\{
  \big(\begin{psmatrix} 1 & b \\ 0 & 1 \end{psmatrix}, 0, w_2 \big)
  \in
  \Gp(\RR)
\big\} \, \cong \, \RR^2
\quad \supset
Z'(\RR) \,:= \, \big\{ \big( \begin{psmatrix} 1 & 0 \\ 0 & 1 \end{psmatrix},
0, w_2 \big), \,w_2 \in \RR \big\}
\ees
yields the second Fourier expansion. On the constant term with respect
to the $w_2$-action, i.e.\ on functions that are $Z'(\RR)$-invariant the
factor group
\begin{gather*}
 \Gp(\RR) \big\slash \Zp(\RR)\, \cong \, 
  \big\{
  \big(\begin{psmatrix} 1 & b \\ 0 & 1 \end{psmatrix}, w_1, \RR \big)
\big\}  \, \cong \, \RR^2
\end{gather*}
acts. This yields another two-variable expansion, that we call
\emph{Fourier--Heisenberg expansion}. 
\begin{gather}
\label{eq:affine_modular_form_fourier_expansion_heisenberg}
  \phi(\tau,z)
\=
  \sum_{n,r \in \ZZ}
  c^\rmH\big( \phi;\, n,r;\, y, \mfrac{v}{y} \big)\,
  e(n x + r u) 
\tx{,}
\end{gather}
and we refine this further by writing for any~$n \in \ZZ$
\be
 c^\rmH\big( \phi;\, n,0;\, y, \mfrac{v}{y} \big)\,
\=
  \sum_{m \in \ZZ}
  c^{\rmH0}(\phi;\, n,m;\, y)\,
  e\big(m \mfrac{v}{y} \big) \,.
\ee
The Fourier--Heisenberg coefficients are again given by
\begin{gather}
\label{eq:coefficients_FH}
  c^{\rmH}\big( \phi;n,r;y,\mfrac v y \big) = \int_{\ZZ\backslash
  \RR}\int_{\ZZ\backslash \RR} \phi(x+iy,u+iv) e(- nx - ru) \, \de x \de u
\end{gather}
and
\begin{gather}
\label{eq:coefficients_FH0}
  c^{\rmH0}(\phi;n,m;y) = \frac 1 y \int_{y\ZZ\backslash \RR}
  c^{\rmH}\big( \phi;n,0;y,\mfrac v y \big) e\big(- m \mfrac v y\big) \, \de v
\tx{.}
\end{gather}
\par
\begin{lemma}%
\label{la:affine_modular_form_fourier_expansion_torus_and_heisenberg}
The $r=0$-part of the torus Fourier expansions coincides with
the $r=0$-part in the Fourier--Heisenberg expansion. That is, 
for any continuous affine modular-invariant function~$\phi$ 
\begin{gather*}
  c^\rmT\big( \phi;\, m,0;\, \tau \big)\,
=
  \sum_{n \in \ZZ}
  c^{\rmH0}\big( \phi;\, n,m;\, y \big)\,
  e(n x)
\tx{.}
\end{gather*}
In particular if for an affine modular-invariant function~$\phi$ all the
Fourier--Heisenberg
coefficients~$c^{\rmH0}(\phi;\, n, m; y)$ vanish for~$n,m \in \ZZ$,
then~$\phi = 0$.
\end{lemma}
\par
\begin{proof} Comparing the coefficient expressions \eqref{eq:coefficients_FT},
    \eqref{eq:coefficients_FH} and \eqref{eq:coefficients_FH0} we see
\ba
 c^\rmT\big( \phi;\, m,0;\, \tau \big) &\=
  \int_{\ZZ \backslash \RR}
  \int_{\ZZ \backslash \RR}
  \phi(\tau, q\tau + p)
  e(-m p) \,\rd p\, \rd q
\tx{,}
\\
  \sum_{n \in \ZZ} c^{\rmH0}\big( \phi;\, n,m;\, y \big) e(n x)
&\=
  \mfrac{1}{y}\,  
  \int_{y \ZZ \backslash \RR}
  \int_{\ZZ \backslash \RR}
  \phi(\tau, u + iv)
  e\big( - m \mfrac{v}{y} \big)
  \,\rd u\, \rd v
  \tx{.}
\ea
The claim then follows from the change of variables $u + i v = p \tau + q$, which
gives $\rmd u + i \rmd v = \tau \rmd p + \rmd q$, that is, $\rmd u = x \rmd p + \rmd q$ and~$\rmd v = y \rmd p$.
\end{proof}
\par
The following proposition gives us a criterion in terms of Fourier--Heisenberg expansions
to show that the representations $\pip_{n m^2}$ occur in $\rmL^2(\cH(0,0))$. It will
be used in the proof of Theorem~\ref{thm:genuine_decomposition}. Recall
that~$\piLsq(\phi)$ denotes the smallest~$\Gp(\RR)$-invariant subspace
of~$\rmL^2(\cH(0,0))$ that contains~$\phi$.
\par
\begin{proposition}%
\label{prop:fourier_term_heisenberg_to_whittaker_model}
Given a continuous square-integrable affine modular function~$\phi$, assume
that the Fourier coefficient~$c^{\rmH0}(\phi;\, n,m;\, y)$ in~\eqref{eq:affine_modular_form_fourier_expansion_heisenberg} does not vanish for
some~$n, m \in \ZZ$, $m \ne 0$.
Then averaging over the subgroup $\Np(\RR)$ defines a surjective homomorphism 
\begin{gather*}
  \piLsq(\phi)
\ra
  \pip_{n m^2}
\tx{,}\quad
  f
\lmto
\Big(g \mto \int_{\Np(\ZZ) \backslash \Np(\RR)} f(h g)\, \ovchiN_{n,m}(h)
\,\rmd h \Big)\,.
\end{gather*}
\end{proposition}
\par
\begin{proof}
Since~$\piLsq(\phi)$ consists of functions that are left invariant with respect
to~$\Gp(\ZZ)$, hence with respect to $\Np(\ZZ)$, and since~$\Np(\ZZ) \backslash
\Np(\RR)$ is compact, the integral is well defined. Given~$\td{h} \in \Np(\RR)$, we have
\begin{gather*}
  \int_{\Np(\ZZ) \backslash \Np(\RR)} f(h \td{h} g)\, \ovchiN_{n,m}(h) \,\rmd h
\=
  \ovchiN_{n,m}(\td{h})
  \int_{\Np(\ZZ) \backslash \Np(\RR)} f(h \td{h} g)\, \ovchiN_{n,m}(h \td{h}) \,\rmd h
\tx{.}
\end{gather*}
Since the integral is taken with respect to the right Haar measure, we can
replace~$h \td{h}$ in the integrand by~$h$. Thus the image of the map in Proposition~\ref{prop:fourier_term_heisenberg_to_whittaker_model} is contained in the representation space~$V({\pisaff_{n,m}})$. Since~${\pisaff_{n,m}}$ is irreducible by Theorem~\ref{thm:Gp_classification_unitary_dual}, the statement follows once we establish that the integral does not vanish for some~$f \in V({\piLsq(\phi)})$.

Recall that the function in~$V(\piLsq(\phi))$ corresponding to~$\phi$ is given by
\begin{gather*}
  f(h g) \=
  \big( \phi \big|_{k,0} h g \big)(i,0)
\tx{.}
\end{gather*}
We apply the Iwasawa decomposition to~$g$ and~$h$, that is, 
\begin{gather*}
  g
\=
  \big( \begin{psmatrix} 1 & b \\ 0 & 1 \end{psmatrix}, w_1,w_2 \big)
  \begin{psmatrix} a & \\ & a^{-1} \end{psmatrix}
  \begin{psmatrix} \cos\,\theta & \sin\,\theta \\ -\sin\,\theta & \cos\,\theta \end{psmatrix}
\tx{,}\quad
  h
\=
  \big( \begin{psmatrix} 1 & \td{b} \\ 0 & 1 \end{psmatrix}, \td{w}_1, \td{w}_2 \big)
\end{gather*}
and when inserting this into~$f$, we obtain
\begin{align*}
  f(h g)
\={} &
  e^{i k \theta} a^k\,
  \phi\big( a^2 i + b+\td{b}, (w_1 + \td{w}_1) a^2 i + (w_2 + \td{w}_2 + b \td{w}_1) \big)
\\
\={}&
  e^{i k \theta} a^k \Big(
  \sum_{n,m \in \ZZ}
  c^{\rmH0}(\phi;\, n,m;\, a^2)\,
  e\big( n (b+\td{b}) + m (w_1 + \td{w}_1) \big)
\\
&\quad
  +\,
  \sum_{\substack{n, r \in \ZZ \\ r \ne 0}}
  c^\rmH\big( \phi;\, n,r;\, a^2, w_1 + \td{w}_1 \big)\,
  e\big( n (b+\td{b}) + r (w_2 + \td{w}_2 + b \td{w}_1) \big) 
  \Big)
\tx{.}
\end{align*}
We next insert this into the expression defining the map in the proposition.
We replace~$n$ and~$m$ in the character~${\chiN_{n,m}}$ to distinguish them from
the indices of the Fourier--Heisenberg expansion. We obtain
\begin{align*}
&
  \int_{\Np(\ZZ) \backslash \Np(\RR)}
   f(h g) \ovchiN_{\td{n},\td{m}}(h) \,dh
\\
\={}&
  e^{i k \theta} a^k\,
  \int_{(\RR \slash \ZZ)^3}
  \Bigg(
  \sum_{n,m \in \ZZ}
  c^{\rmH0}(\phi;\, n,m;\, a^2)\,
  \frac{e\big( n (b+\td{b}) + m (w_1 + \td{w}_1) \big)}
  {e\big(\td{n} \td{b} + \td{m} \td{w}_1 \big)}
\\
&
  +\,
  \sum_{\substack{n, r \in \ZZ \\ r \ne 0}}
  c^\rmH\big( \phi;\, n, r;\, a^2, w_1 + \td{w}_1 \big)\,
  \frac{e\big( n (b+\td{b}) + r (w_2 + \td{w}_2 + b \td{w}_1) \big)}
  {e\big(\td{n} \td{b} + \td{m} \td{w}_1 \big)}
  \,\rmd \td{w}_2 \rmd \td{b} \rmd \td{w}_1
  \Bigg)
\tx{.}
\end{align*}
We can interchange the summation over~$n,m,r$ and the  integral over the compact
set~$\Zp(\RR) \slash \Zp(\ZZ)$, which is parametrized by~$\td{w}_2$. This allows us
to discard all contributions from the second part of the Fourier--Heisenberg series. We are then left with the integral
\ba \label{eq:Heis_to_Whit_explicit}
&
  e^{i k \theta} a^k \mspace{-3mu}
  \sum_{n,m \in \ZZ} \mspace{-5mu}
  c^{\rmH0}(\phi;\, n,m;\, a^2) \mspace{-1mu}
  \int_{(\RR \slash \ZZ)^2} \mspace{-28mu}
  e\big( n b + (n-\td{n}) \td{b} + m w_1 + (m-\td{m}) \td{w}_1 \big)
  \,\rmd \td{b} \rmd \td{w}_1
\\
\={}&
  e^{i k \theta} a^k\,
  c^{\rmH0}(\phi;\, \td{n},\td{m};\, a^2)\,
  e(n b + m w_1)
\tx{.}
\ea
By the assumptions of the proposition the righthand side does not
vanish for some choice of $\td{n}, \td{m}, a$.
\end{proof}

\subsection{Towards an \texorpdfstring{$\rmL^2$}{L2}-isometry}%
\label{ssec:ExtensionL2}

In this section we give an expression of the scalar product of
affine modular-invariant functions in terms of Fourier--Heisenberg coefficients.
It will be used for several orthogonality statements. It generalizes
the classical `unfolding' construction on the modular surface. To
generalize the content of this paper to general strata one of the
main challenges will be to find a replacement of  this lemma.
\par
\begin{lemma} \label{lemma:SPviaFourier}
The scalar product of two continuous affine modular-invariant
functions $\phi_i$ can be expressed in Fourier--Heisenberg coefficients as
\be
\langle \phi_1, \phi_2 \rangle = \sum_{n,m \in \ZZ} 
\int_{\RR^+} c^{\rmH0}\big(\phi_1;\, n, m;\, y \big)\,
\overline{c^{\rmH0}\big(\phi_2;\, n,m;\, y \big)}\, \frac{\rmd y}{y^{2-k}}\tx{.}
\ee
\end{lemma}
\par
\begin{proof} Starting with~\eqref{eq:def:inner_product_maass_forms} and
    abbreviating $X = \G(\ZZ)\backslash \G(\RR) / \rmK$ we find 
\bas
\langle \phi_1, \phi_2 \rangle 
&\= \int_{\G'(\ZZ) \backslash \Gp(\RR) / \rmK} \phi_1(\tau,p\tau+q)
\ov{\phi_2(\tau,p\tau +q)} \,\frac{\de x \de y \de p \de q}{y^{2-k}} \\
&\= \int_{X} \sum_{m \geq 0}
\sum_{\substack{c,d \in \ZZ \\ \gcd(c,d) = 1}} 
c^\rmT\big(\phi_1;\, m \cdot(d,-c);\, \tau \big) \, \cdot \,
\ov{c^\rmT\big(\phi_2;\, m \cdot (d,-c);\, \tau \big)} \,\frac{\de x \de y}{y^{2-k}} \\
&\= \int_{X} \sum_{m \geq 0}
\sum_{\gamma \in \Gamma_\infty^+ \backslash \G(\ZZ)}
c^\rmT\big(\phi_1;\, (m, 0);\, \tau \big)\big|_k \gamma\,
\cdot
\ov{c^\rmT\big(\phi_2;\, (m,0);\, \tau \big)\big|_k \gamma} \,\frac{\de x \de y}
{y^{2-k}} \\
&\= \sum_{m \geq 0}\, \int_{\Gamma_\infty^+ \backslash \rmG(\RR) / \rmK}
c^\rmT\big(\phi_1;\, m, 0; \, \tau \big) \,
\ov{c^\rmT\big(\phi_2;\, m, 0; \, \tau \big)} \,\frac{\de x \de y}{y^{2-k}}\,.
\eas
Here
we used the orthogonality of the exponential terms $e(\ell p+rq)$ on
$\rmL^2(\ZZ^2\backslash \RR^2)$, and writing $m = \gcd(r,\ell)$ we used
Lemma~\ref{la:affine_modular_form_fourier_expansion_torus_and_heisenberg}.
We then combined the $X$-integral and the summation over
$\Gamma_\infty^+ \backslash \G(\ZZ)$ to the integral over the strip
$\Gamma_\infty^+ \backslash \rmG(\RR) / \rmK$. The identity is then obtained by
rewriting the $c^\rmT$-coefficients
in terms of $c^{\rmH0}$-coefficients using
Lemma~\ref{la:affine_modular_form_fourier_expansion_torus_and_heisenberg}
and performing the $x$-integration. Using orthogonality of exponential terms,
only the diagonal terms remain, which gives the claimed formula.  
\end{proof}
\par
The following statement will help to prove injectivity in
Theorem~\ref{thm:genuine_decomposition}, compare with the disintegration of
the Haar measure along the torus fibres in~\eqref{eq:genuinechar}. It
complements the vanishing statements in
Proposition~\ref{prop:affine_modular_form_vanishing_by_torus_fourier_coefficients}
and Lemma~\ref{la:affine_modular_form_fourier_expansion_torus_and_heisenberg}.
\par
\begin{lemma}%
  \label{la:affine_modular_form_genuine_fourier_expansion_torus}
For a non-genuine affine-invariant modular function~$\phi$ of weight~$k$,
we have
\begin{gather*}
  c^\rmT(\phi; m, r; \tau) = 0
\quad
  \tx{for all\ }(m,r) \ne (0,0)
\tx{.}
\end{gather*}
An affine-invariant modular function~$\phi$ is genuine
if and only if $c^\rmT(\phi; 0, 0; \tau) = 0$.
\end{lemma}
\par
\begin{proof}
Non-genuine modular functions are pullbacks from~$\HS$ to~$\HSp$, i.e.\@ they are constant in~$z$. In particular, the only Fourier coefficient with respect to~$p$ and~$q$ that appears in~\eqref{eq:affine_modular_form_fourier_expansion_torus} is the one of index~$(0,0)$.
\par
A genuine modular function~$\phi$ of weight~$k$ by~\eqref{eq:L2Gp_genuine_decomposition} is orthogonal with respect to the inner product~\eqref{eq:def:inner_product_maass_forms} to all non-genuine ones of the same weight. The second statement thus follows from the first and
Lemma~\ref{lemma:SPviaFourier}.
%
\end{proof}

We get an immediate corollary, useful in proving discreteness of the spectrum
of the compound Laplacian on cusp forms.

\begin{proposition} \label{cor:genuineextravanishing}
    Let $f$ be a genuine affine-invariant modular form of weight $k$. For all $n
    \in \ZZ$, the Fourier--Heisenberg coefficients
    \begin{equation}
        c^{\rmH0}(f;\, n, 0; \,y) \= 0.
    \end{equation}
\end{proposition}

\begin{proof}
    Genuine affine-invariant modular forms have $c^{\rmT}(f;0,0;\tau)
    = 0$. It follows from Lemma
    \ref{la:affine_modular_form_fourier_expansion_torus_and_heisenberg} that
    \begin{equation}
        0 = |c^{\rmT}(f;0,0;\tau)|^2 = \sum_{n \in \ZZ} |c^{\rmH0}
        (f;n,0;y)|^2 
    \end{equation}
    which is only possible if every term in this sum vanishes.
\end{proof}

\subsection{Eisenstein and Poincar\'e series}%
\label{ssec:eisenstein_poincare_series}

We now define the series that provides
    generators
for the constituents of~$\rmL^2(\SAff{2}(\ZZ) \backslash \SAff{2}(\RR) )$.
Given~$n,m \in \ZZ \setminus \{0\}$ 
and moreover a function~$\beta \defcol \RR^+ \ra \CC$
with~$|\beta(y)| \ll y^{1-\frac{k}{2}+\epsilon}$ as~$y \ra 0$, we define
the \emph{affine Poincar\'e series}
\begin{gather}
\label{eq:def:affine_poincare_series}
  P_{k;n,m,\beta}(\tau, z)
\defeq
   2^{-\frac{1}{2}} \sum_{\ga \in \Ga_\infty^+ \backslash \SL{2}(\ZZ)}
  \beta(y)\,
  e\big( n x + m \mfrac{v}{y} \big)\big|'_{k}\, \ga
\tx{,}
\end{gather}
where $\Ga_\infty^+ = \langle \begin{psmatrix} 1&1 \\ 0 & 1 \\ \end{psmatrix}
\rangle$. 
The case $n=0$ in this description is what we call the \emph{
affine Eisenstein series}
\begin{gather}
\label{eq:def:affine_eisenstein_series}
  E_{k;m,\beta}(\tau, z)
  \defeq 2^{-\frac{1}{2}} \sum_{\ga \in \Ga^+_\infty \backslash \SL{2}(\ZZ)}
  \beta(y)\,
  e\big( m \mfrac{v}{y}  \big)
  \big|'_{k}\, \ga
\tx{.}
\end{gather}
\begin{lemma}%
\label{la:affine_eisenstein_poincare_series_convergent}
The righthand sides of~\eqref{eq:def:affine_eisenstein_series} and~\eqref{eq:def:affine_poincare_series} are absolutely and locally uniformly convergent.
\end{lemma}
\par
\begin{proof}
The summations in~\eqref{eq:def:affine_eisenstein_series}
and~\eqref{eq:def:affine_poincare_series} are well-defined thanks to the
periodicity of of~$e(\cdot)$.
\par
We identify a coset~$\gamma \in \Ga_\infty^+ \backslash \SL{2}(\ZZ)$ with the
entries~$c, d \in \ZZ$, $\gcd(c,d) = 1$ in the bottom row of the matrix. To
show convergence, we need an estimate for the right hand side of
\bes
  \Bigg|
  \sum_{\ga \in  \Ga^+_\infty \backslash \SL{2}(\ZZ)}
  \beta(y)\,
  e\big( n x + m \mfrac{v}{y} \big)
  \Big|'_k\, \ga
  \Bigg|\, \le \,
  \sum_{\substack{c,d \in \ZZ \\ \gcd(c,d) = 1}}
  |c \tau + d|^{-k}\,
  \big| \beta\big( \Im(\ga\tau) \big) \big|
\ees
both for $n=0$ and $n \neq 0$. Using the bound~$\beta(y) \ll
y^{1 - \frac{k}{2} + \epsilon}$ as~$y \ra 0$, we obtain the estimate
\begin{multline*}
  \sum_{\substack{c,d \in \ZZ \\ \gcd(c,d) = 1}}
  |c \tau + d|^{-k}\,
  \big| \beta\big( \Im(\ga\tau) \big) \big|
=
  y^{-\frac{k}{2}}
  \sum_{\substack{c,d \in \ZZ \\ \gcd(c,d) = 1}}
  \Im(\ga \tau)^{\frac{k}{2}}
  \big| \beta\big( \Im(\ga\tau) \big) \big|
\\
\ll
  y^{-\frac{k}{2}}
  \sum_{\substack{c,d \in \ZZ \\ \gcd(c,d) = 1}}
  \Im(\ga \tau)^{1+\epsilon}
\tx{,}
\end{multline*}
which converges absolutely and locally uniformly as required. 
\end{proof}
\par
We determine the Fourier--Heisenberg expansions of affine Eisenstein and
Poincar\'e series. This allows us to examine the~$\SAff{2}(\RR)$ representations
that those series generate. It is also an important ingredient in the proof of
the~$\SAff{2}(\RR)$ decomposition in Theorem~\ref{thm:genuine_decomposition}.
\begin{lemma}%
\label{la:affine_eisenstein_poincare_series_fourier_expansion_along_heisenberg}
The Fourier--Heisenberg coefficients of the affine Eisenstein and Poincar\'e
series at $r=0$ are
\begin{align*}
  c^{\rmH0}\big( E_{k;m,\beta}(\tau, z); \td{n}, \td{m}; y \big)
\=
  \begin{cases}
    2^{-\frac{1}{2}}\,
    \beta(y) \tx{,}\quad & \tx{if\ } \td{n} = 0 \tx{\ and\ } \td{m} = \pm m
  \tx{;}
  \\
    0 \tx{,}\quad & \tx{otherwise.}
  \end{cases}
\\
  c^{\rmH0}\big( P_{k;n,m,\beta}(\tau, z); \td{n}, \td{m}; y \big)
\=
  \begin{cases}
    2^{-\frac{1}{2}}\,
    \beta(y) \tx{,}\quad & \tx{if\ } \td{n} = n \tx{\ and\ } \td{m} = \pm m
  \tx{;}
  \\
    0 \tx{,}\quad & \tx{otherwise.}
  \end{cases}
\end{align*}
In particular, if ~$y^{\frac{k-1}{2}} \beta(y)$
in~\eqref{eq:def:affine_poincare_series}
and~\eqref{eq:def:affine_eisenstein_series} is square-integrable, then
\begin{gather*}
  \big\| E_{k, \beta, m} \big\|_{\HSp}^2
\=
  \big\| P_{k, \beta, m, n} \big\|_{\HSp}^2
\=
  \big\| y^{\frac{k-1}{2}} \beta \big\|^2
\tx{,}
\end{gather*}
where on the right hand side the $\rmL^2$-norm is with respect to the Haar
measure on~$\RR^+$.
\end{lemma}
\par
Note that these Fourier coefficients are independent of~$k$.
\par
\begin{proof}
By definition the coefficients $c^{\rmH0}$ are the~$0$\thdash{} Fourier
coefficient with respect to~$u$.
We observe that~$e(n\Re(\tau))$ will not contribute any dependency on~$u$ and
neither will~$(c \tau + d)^{-k}$. We thus have to examine only the contribution
of~$e( m v \slash y)$. We have
\begin{gather}
\label{eq:vy_transformation}
  \mfrac{v}{y} \Big|'_{0}\, \begin{psmatrix} a & b \\ c & d \end{psmatrix}
=
  \frac{z (c \ov\tau + d) - \ov{z} (c \tau + d)}{|c \tau + d|^2}\,
  \frac{|c \tau + d|^2}{y}
=
  \mfrac{v}{y} (c x + d)
  -
  cu
\tx{.}
\end{gather}
In particular, the only contributions to the~$0$\thdash{} Fourier coefficient with respect to~$u$ arise from~$c = 0$, which is the term of the identity matrix.
\par
The second statement follows from the first and
Lemma~\ref{lemma:SPviaFourier}.
\end{proof}
\par
The Fourier--Heisenberg expansion also allows us to determine the~$\rmL^2$\nbd{}norm of Eisenstein and Poincar\'{e} series. This stands in stark contrast to the case of Maa\ss{} forms for~$\SL{2}(\ZZ)$, where Maa\ss{} Eisenstein series are not square-integrable in general and Maa\ss{} Poincar\'{e} series have comparably inaccessible formulae for their Fourier expansion. 
\par

\subsection{Representations generated by Eisenstein and Poincar\'{e} series}%
\label{ssec:eisenstein_poincare_series_representations}

The main concern of this section is to determine the isomorphism class of
the representations generated by lifts of the Eisenstein and Poincar\'{e}
series, which we will achieve in Proposition~\ref{prop:affine_eisenstein_poincare_series_representation}. To prepare its proof, we first identify the pullback of Eisenstein and Poincar\'{e} series to~$\Gp(\RR)$ as the images of partially defined maps from~$\pip_{n m^2} \cong \pip_{n,m}$ (allowing~$n = 0$ to include the case of Eisenstein series) to~$\rmL^2(\Gp(\ZZ) \backslash \Gp(\RR))$. Second, we show that these maps are isometries on their range and equivariant with respect to~$\SAff{2}(\RR)$. In the proof of Proposition~\ref{prop:affine_eisenstein_poincare_series_representation} this allows us to extend them to all of~$\pip_{n,m}$.
\par
Recall that~$\pip_{n,m}$ is an induction from~$\Np(\RR)$ to~$\Gp(\RR)$. The Iwasawa decomposition in~\eqref{eq:iwasawa_decomposition} shows that $f \in V(\pip_{n,m})$ of $\rmK$\nbd{}type~$k$ is uniquely defined by its values on~$\Ap(\RR) \slash (\Ap(\RR) \cap \rmK)$. To make the connection to the function~$\beta$ in~\eqref{eq:def:affine_eisenstein_series} and~\eqref{eq:def:affine_poincare_series}, we identify this quotient with~$\RR^+$ via the section
\begin{gather*}
  \RR^+
\ra
  \Ap(\RR) \big\slash \big( \Ap(\RR) \cap \rmK \big)
\tx{,}\quad
  a
\lmto
  \pm
  \begin{psmatrix}
   a & \\
   & a^{-1}
  \end{psmatrix}
\tx{.}
\end{gather*}
The functions~$\alpha(a)$ in this section correspond to~$\beta(y) = \alpha(\sqrt{y})$ in~\eqref{eq:def:affine_eisenstein_series} and~\eqref{eq:def:affine_poincare_series}.
\par
Given a function~$\alpha \defcol \RR^+ \ra \CC$ and~$k \in \ZZ$, we use the
pullback construction in~\eqref{eq:def:modular_to_automorphic}, which
implicitly depends on~$k$, to define the function
\begin{gather}
\label{eq:la:induction_elements_on_positive_ray}
  \wht\alpha_k(g)
\defeq 
\Bigl(  \bigl( \alpha(\sqrt{y})\, e\big( n x + m \mfrac{v}{y} \big) \bigr)
\Big|'_k g  \Bigr) (i,0)
\tx{,}
\end{gather}
where we suppress~$n$ and~$m$ from our notation.
\par
\begin{lemma}%
\label{la:induction_elements_on_positive_ray}
If $\alpha \in \rmL^2(\RR^+, a^{2k-3} \rmd a)$, then $\wht\alpha_k \in V(\pip_{n,m})$.
More precisely, given the~$\rmK$\nbd{}type decomposition
\begin{gather*}
  V\big( \pip_{n,m} \big)
\=
  \bigoplus_{k \in \ZZ}
  V\big( \pip_{n,m} \big)_k
\tx{,}
\end{gather*}
we have
\begin{gather}
\label{eq:la:induction_elements_on_positive_ray:decomposition}
  V\big( \pip_{n,m} \big)_k
\=
  \linspan\big\{ 
  \wht\alpha_k
  \condcol
  \alpha \in \rmL^2\big(\RR^+, a^{2k-3} \rmd a \big)
  \big\}
\tx{.}
\end{gather}
\end{lemma}
\par
\begin{proof}
The coordinates for the decomposition~$\Gp(\RR) = \Np(\RR)\, \Ap(\RR)\, \rmK$
given in~\eqref{eq:nak_decomposition_to_HSp} yield
\begin{gather*}
  \wht\alpha_k\Big(
  \big( \begin{psmatrix} 1 & b \\ & 1 \end{psmatrix}, w_1, w_2 \big)\,
  \begin{psmatrix} a & \\ & a^{-1} \end{psmatrix}\,
  \begin{psmatrix} \cos\,\theta & \sin\,\theta \\ -\sin\,\theta & \cos\,\theta \end{psmatrix}
  \Big)
=
  e^{i k \theta} a^k\,
  \alpha(a)
  e\big( n b + m w_1 \big)
\tx{.}
\end{gather*}
From this we directly read off that~$\wht{\alpha}_k$ satisfies the
transformation properties required for elements of~$V(\pip_{n,m})$, namely
\begin{gather*}
  \wht\alpha_k(h g)
\=
  \chiN_{n,m}(h)\,
  \wht\alpha_k(g)
\quad\tx{for all\ } h \in \Np(\RR)
\tx{.}
\end{gather*}
To verify that~$\wht\alpha_k \in V(\pip_{n,m})$ it remains to verify that it is square integrable on~$\Ap(\RR)\, \rmK$, which follows from
\begin{gather*}
  \int_{\Ap(\RR)\, \rmK}
  \Big|
  \wht\alpha_k\Big(
  \begin{psmatrix} a & \\ & a^{-1} \end{psmatrix}\,
  \begin{psmatrix}
    \cos\,\theta & \sin\,\theta \\
    -\sin\,\theta & \cos\,\theta
  \end{psmatrix}
  \Big)
  \Big|^2
  \,\frac{\rmd \theta\, \rmd a}{2 \pi\, a^3}
=
  \int_{\RR^+}
  a^{2k-2} |\alpha(a)|^2
  \,\frac{\rmd a}{a}
\tx{.}
\end{gather*}
Notice that the correspondence between~$\alpha$ and~$\wht\alpha_k$ is
one-to-one.
\par
In order to confirm the given~$\rmK$\nbd{}type decomposition, note that by the
Peter--Weyl theorem for the compact group~$\rmK$, any~$f \in V(\pip_{n,m})$ can
be decomposed as a square-summable series~$\sum_k f_k$, where~$f_k$ is square-integrable and of~$\rmK$\nbd{}type~$k$. By the previous argument and for fixed~$k$, we find $f_k = \wht\alpha_k$ for some~$\alpha \in
\rmL^2(\RR^+,a^{2k-3} \rmd a)$ as desired.
\end{proof}
\par
Consider~$f \in V(\pip_{n,m})$. We associate to~$f$
a~$\SAff2(\ZZ)$-left-invariant function on~$\SAff2(\RR)$ via
\begin{gather}
\label{eq:def:automorphic_eisenstein}
  E(f, \cdot) \defeq \Bigl( g \mapsto
  \sum_{\ga \in \Ga^+_\infty \backslash \SL{2}(\ZZ)}
  f(\ga g)\, \Bigr)\,, \qquad g \in \SAff{2}(\RR)\,,
\end{gather}
provided absolute convergence of the series.
\par
\begin{lemma}%
\label{la:automorphic_eisenstein_isometry}
Suppose $f \in V(\pip_{n,m})$ has a finite decomposition~$\sum \wht\alpha_k$
for functions~$\alpha_k \in \rmL^2(\RR^+, a^{2k-3} \rmd a)$, $k \in \ZZ$ according to~\eqref{eq:la:induction_elements_on_positive_ray:decomposition}. Assume that
\bes
  \Big|
   \wht\alpha_k \big(
  \begin{psmatrix} 1 & b \\ & 1 \end{psmatrix}, w_1, w_2 \big)\,
  \begin{psmatrix} a & \\ & a^{-1} \end{psmatrix}\,
  \begin{psmatrix} \cos\,\theta & \sin\,\theta \\ -\sin\,\theta & \cos\,\theta \end{psmatrix}
  \big)
  \Big|
  \ll a^{2 - k + \epsilon}
\quad\tx{as}\quad
  a \ra 0
\tx{.}
\ees
Then~$E(f, \cdot)$ converges absolutely and locally uniformly and
$\|f\|^2 = 2 \|E(f,\,\cdot\,)\|^2$.
\end{lemma}
\par
\begin{proof}
Since~$\rmK$ is compact, every~$\wht\alpha_k$ in the decomposition~$f =
\sum \wht\alpha_k$  satisfies the same growth condition as~$f$. Therefore,
it suffices to demonstrate convergence of~$E(\wht\alpha_k, g)$
with~$\wht\alpha_k$ as in~\eqref{eq:la:induction_elements_on_positive_ray}. Further, since we average over~$\Ga^+_\infty \backslash \SL{2}(\ZZ)$ from the left and~$\rmK$\nbd{}types are defined via right-shifts, the~$\rmK$\nbd{}type of each summand in the definition~\eqref{eq:def:automorphic_eisenstein} of~$E(\wht\alpha_k,g)$ is~$k$ as well. It thus suffices to consider~$g \in \Np(\RR) \Ap(\RR)$. Using~\eqref{eq:def:modular_to_automorphic}, this allows us to perform the proof for functions on~$\HSp$ instead of~$\Gp(\RR)$.
\par
Recall the expression~$e^{i k \theta} a^k \alpha_k(a)\, e(n b + m w_1)$ from the proof of Lemma~\ref{la:induction_elements_on_positive_ray}. Under the map in~\eqref{eq:def:modular_to_automorphic}  using the coordinates of~\eqref{eq:nak_decomposition_to_HSp}, we have to show the absolute and locally uniform convergence of
\bas
  \sum_{\ga \in \Ga^+_\infty \backslash \SL{2}(\ZZ)}
\phantom{\=} & e^{-i k \theta}
  y^{-\frac{k}{2}}\,
  e^{i k \theta}
  \sqrt{y}^k \alpha_k(\sqrt{y}) e\big( n x + m \mfrac{v}{y} \big)
  \big|'_k\, \ga \\
\=&
  \sum_{\ga \in \Ga^+_\infty \backslash \SL{2}(\ZZ)}
  \alpha_k(\sqrt{y}) e\big( n x + m \mfrac{v}{y} \big)
  \big|'_k\, \ga
\tx{.}
\eas
Since by assumptions~$\beta_k(y) \defeq \alpha_k(\sqrt{y}) < \sqrt{y}^{2 - k + \epsilon} = y^{1 - k\slash2 + \epsilon\slash2}$, we can apply Lemma~\ref{la:affine_eisenstein_poincare_series_convergent} to finish the proof of convergence.
\par
To show the isometry property, we notice that we have the~$\rmK$\nbd{}type decomposition
\begin{gather*}
  E(f,\,\cdot\,)
=
  \sum_{k \in \ZZ} E(\wht\alpha_k,\,\cdot\,)
\tx{,}
\end{gather*}
and that the summands on the right hand side are mutually orthogonal.
\begin{align*}
&
  \| E(f, \,\cdot\,) \|^2
=
  \sum_{k \in \ZZ}
  \| E(\wht\alpha_k, \,\cdot\,) \|^2
\\
={}&
  \sum_{k \in \ZZ}
   \big\|y^{\frac{k-1}{2}} \beta_k(y) \big\|^2
=
  \sum_{k \in \ZZ}
   \big\|a^{k-1} \alpha_k(a) \big\|^2
=
  \| f \|^2
\tx{.}
\end{align*}
We employed the identification in~\eqref{eq:la:automorphic_eisenstein_isometry:pullback} and
then Lemma~\ref{la:affine_eisenstein_poincare_series_fourier_expansion_along_heisenberg}
to express the resulting norms in terms of the norms of first~$\beta_k$ and then~$\alpha_k$.
\end{proof}
\par
\begin{lemma}%
\label{la:automorphic_eisenstein_interwining}
Given a function~$f$ subject to the conditions of Lemma~\ref{la:automorphic_eisenstein_isometry}, the function~$f_h(g) \defeq f(g h)$ for~$h \in \SAff{2}(\RR)$ also satisfies the conditions of Lemma~\ref{la:automorphic_eisenstein_isometry}, and
$E(f, g h) = E(f_h, g)$.
\end{lemma}
\par
\begin{proof}
The first statement is clear when applying the~$\Np(\RR) \Ap(\RR) \rmK$\nbd{}decomposition to~$g = n a k$ and~$k h k^{-1} = \td{n} \td{a} \td{k}$ and multiplying out the result as~$g h = n a \td{n} \td{a}\, \td{k} k$. Now the second claim follows from the observation that~$E(f_h,g)$ is defined by a sum over left-shifts, while~$h$ acts via right-shifts.
\end{proof}
\par
We can now state the first main result of this section.
\par
\begin{proposition}%
\label{prop:affine_eisenstein_poincare_series_representation}
The Eisenstein series and Poincar\'e series for appropriate parameters generate
all the genuine unitary irreducible representations of\/~$\Gp(\RR)$. More precisely, for $k \in \ZZ$, and $\beta: \RR^+ \ra \CC$ with $|\beta(y)| \ll y^{1-\frac k 2 + \varepsilon}$ as as~$y \ra 0$
\begin{gather*}
  \piLsq\big( \wtd{E_{k;m,\beta}} \big)
\cong
  \pip_0
\quad\tx{and}\quad
  \piLsq\big( \wtd{P_{k;n,m,\beta}} \big)
\cong
  \pip_{n m^2}
\tx{.}
\end{gather*}
\end{proposition}
\par
\begin{proof}
In the proof of Lemma~\ref{la:automorphic_eisenstein_isometry}
we have shown that if we let $\alpha(a) = \beta(a^2)$ then 
\begin{gather}
\label{eq:la:automorphic_eisenstein_isometry:pullback}
  E(\wht\alpha_k,\,\dot\,)
=
  \wtd{E_{k;m,\beta}}
\tx{,}\quad
  \tx{if\ } n = 0
\tx{,}\quad\tx{and}\quad
  E(\wht\alpha_k,\,\cdot\,)
=
  \wtd{P_{k;n,m,\beta}}
\tx{,}\quad
  \tx{otherwise.}
\end{gather}
\par
The automorphic Eisenstein series $E(\,\cdot\,,\,\cdot\,)$
in~\eqref{eq:def:automorphic_eisenstein} yields by
Lemma~\ref{la:automorphic_eisenstein_isometry}
and Lemma~\ref{la:automorphic_eisenstein_interwining} a
 map from a dense subspace of~$V(\pip_{n,m})$
to~$\rmL^2(\SAff{2}(\ZZ) \backslash \SAff{2}(\RR))$ which is $\SAff{2}(\RR)$\nbd{}equivariant.
As it is an isometry by Lemma~\ref{la:automorphic_eisenstein_isometry}, this map does not vanish and is a homomorphism of Hilbert space representations from~$\pip_{n,m}$ to~$\rmL^2(\SAff{2}(\ZZ) \backslash \SAff{2}(\RR))$. Since~$\pip_{n,m}  \cong \pip_{nm^2}$ is irreducible by the classification in Theorem~\ref{thm:Gp_classification_unitary_dual}, the image of the map~$E(\,\cdot\,,\,\cdot\,)$
in~\eqref{eq:la:automorphic_eisenstein_isometry:pullback} is equal to the image of
\begin{gather*}
  E \defcol
  \pip_{n,m}
\ra
  \rmL^2\big( \SAff{2}(\ZZ) \backslash \SAff{2}(\RR) \big)\,.
\end{gather*}
as claimed in the proposition.
\end{proof}

\section{Decomposition of the \texorpdfstring{$\rmL^2$}{L2}-space}
\label{sec:decomp}

\subsection{Decomposition as \texorpdfstring{$\Gp(\RR)$}{G'(R)}-representations}

We restate and refine Theorem~\ref{intro:decompSAff} from the introduction.
\par
\begin{theorem}%
\label{thm:genuine_decomposition}
The genuine part of
the $\rmL^2$-space of the stratum~$\cH(0,0)$ admits the abstract decomposition
\ba \label{eq:genuine_decomposition}
  \rmL^2\big( \Gp(\ZZ) \backslash \Gp(\RR) \big)^\gen
&\;\cong\;
  \bigoplus_{m = 1}^\infty
  \bigoplus_{n \in \ZZ}
  \pisaff_{n,m}
\tx{.}
\ea
into irreducible~$\Gp(\RR)$\nbd{}representations. More precisely,
it can be decomposed as
\begin{gather}
\label{eq:genuine_decomposition_equality}
  \rmL^2\big( \Gp(\ZZ) \backslash \Gp(\RR) \big)^\gen
\=
  \bigoplus_{m = 1}^\infty
  \Big(
  \piLsq\big( \wtd{E_{k;m,\beta}} \big)
  \oplus
  \bigoplus_{n \in \ZZ \setminus \{0\}}
  \piLsq\big( \wtd{P_{k;n,m,\beta}} \big)
  \Big)
\tx{,}
\end{gather}
and the map~$E$ from~\eqref{eq:def:automorphic_eisenstein} defines an
isometry of\/~$\Gp(\RR)$\nbd{}representations 
\bes
\pisaff_{0,m} \,\cong\,  \piLsq\big( \wtd{E_{k;m,\beta}} \big)
\quad \text{and} \quad
\pisaff_{n,m} \,\cong\,  \piLsq\big( \wtd{P_{k;n,m,\beta}} \big)
\ees
for any~$k \in \ZZ$ and~$0 \neq \beta \in \rmL^2(\RR^+, y^{k-2} \rmd y)$,
if the left hand side is provided with the $\rmL^2$-norms
from~\eqref{la:induction_elements_on_positive_ray}. Furthermore, for $\beta_i \in \rmL^2(\RR^+, y^{k_i-2}\rmd y), i=1,2$
\begin{alignat*}{2}
  \piLsq\big( E_{k_1;m_1,\beta_1} \big)
&\cong
  \piLsq\big( E_{k_2;m_2,\beta_2} \big)
\quad
&&
  \tx{for all~$m_1$ and~$m_2$, and}
\\
  \piLsq\big( P_{k_1;n_1,m_1,\beta_1} \big)
&\cong
  \piLsq\big( P_{k_2;n_2,m_2,\beta_2} \big)
\quad
&&
  \tx{if and only if~$n_1 m_1^2 = n_2 m_2^2$}
\tx{.}
\end{alignat*}
\end{theorem}
\par
\begin{proof}
The isomorphisms stated in the second part are a consequence of
Theorem~\ref{thm:Gp_classification_unitary_dual}. Note that the weights~$k_1$
and~$k_2$ do not appear in these statements and hence do not distinguish
representations.
\par
By Proposition~\ref{prop:affine_eisenstein_poincare_series_representation}, Eisenstein and Poincar\'{e} series yield isomorphisms
\begin{alignat*}{2}
  \pisaff_{0,m}
&\ra
  \piLsq\big( \wtd{E_{k;m,\beta}} \big)
&&\subseteq
  \rmL^2\big( \Gp(\ZZ) \backslash \Gp(\RR) \big)^\gen
\quad\tx{and}
\\
 \pisaff_{n,m}
&\ra
  \piLsq\big( \wtd{P_{k;n,m,\beta}} \big)
&&\subseteq
  \rmL^2\big( \Gp(\ZZ) \backslash \Gp(\RR) \big)^\gen
\quad
  \tx{for\ }n \ne 0
\tx{.}
\end{alignat*}
Taking the direct sum, we obtain a map
\begin{gather}
\label{eq:thm:genuine_decomposition:eisenstein_poincare}
  \bigoplus_{m = 1}^\infty
  \bigoplus_{n \in \ZZ}
  \pisaff_{n,m}
=
  \bigoplus_{m = 1}^\infty
  \pisaff_{0,m}
  \,\oplus\,
  \bigoplus_{m = 1}^\infty
  \bigoplus_{n \in \ZZ \setminus \{0\}}
  \pisaff_{n,m}
\ra
  \rmL^2\big( \Gp(\ZZ) \backslash \Gp(\RR) \big)^\gen
\tx{.}
\end{gather}
The Fourier--Heisenberg expansions group provide us by
Proposition~\ref{prop:fourier_term_heisenberg_to_whittaker_model} with a map
\begin{gather}
\label{eq:thm:genuine_decomposition:fourier_expansion}
\rmL^2\big( \Gp(\ZZ) \backslash \Gp(\RR) \big)^\gen
\cap \rmC\big( \Gp(\ZZ) \backslash \Gp(\RR) \big)
\ra
  \bigoplus_{m = 1}^\infty
  \bigoplus_{n \in \ZZ}
  \pisaff_{n,m}
\tx{.}
\end{gather}
By Lemma~\ref{la:affine_eisenstein_poincare_series_convergent} Eisenstein and
Poincar\'{e} series associated with continuous functions~$\beta$ that satsify
the growth condition~$\beta(y) \ll y^{1-\frac{k}{2}+\epsilon}$ as~$y \ra 0$ are
continuous.
\par
Let~$V\subset \bigoplus_{m\ge 0} \bigoplus_{n \in \ZZ} \pisaff_{n,m}$ be the
dense subspace consisting in finite sums of
continuous functions that satisfy the assumptions given in
Lemma~\ref{la:automorphic_eisenstein_isometry}. By
Lemma~\ref{la:affine_eisenstein_poincare_series_fourier_expansion_along_heisenberg}
the composition of the restriction
of ~\eqref{eq:thm:genuine_decomposition:eisenstein_poincare} to~$V$
and~\eqref{eq:thm:genuine_decomposition:fourier_expansion} is the multiplication by~$2^{-1 \slash 2}$ map.
Since~$V$ is a dense subspace, this shows that there is an injection
\begin{gather}
\label{eq:thm:genuine_decomposition:injection}
  \bigoplus_{m = 1}^\infty
  \bigoplus_{n \in \ZZ}
  \pisaff_{n,m}
\=
  \bigoplus_{m = 1}^\infty
  \bigoplus_{n \in \ZZ}
  \pisaff_{nm^2}
\hra
  \rmL^2\big( \Gp(\ZZ) \backslash \Gp(\RR) \big)^\gen
\tx{.}
\end{gather}
\par
We next investigate the kernel
of~\eqref{eq:thm:genuine_decomposition:fourier_expansion}. 
Lemma~\ref{la:affine_modular_form_fourier_expansion_torus_and_heisenberg} shows
that it consists of functions whose~$\rmK$\nbd{}isotypical components,
say~$\phi_k$, satisfy~$c^\rmT(\phi_k;\, m, 0;\, \tau ) = 0$ for all
positive~$m$. By
Lemma~\ref{la:affine_modular_form_fourier_expansion_torus_sl2_symmetry}
with~$\ga$ equal the negative identity, this implies that~$c^\rmT(\phi_k;\, m,
0;\, \tau ) = 0$ for all~$m \ne 0$. Since~$\phi_k$ is genuine, we also
have $c^\rmT(\phi_k;\, 0, 0;\, \tau ) = 0$ by
Lemma~\ref{la:affine_modular_form_genuine_fourier_expansion_torus}. This allows
us to apply
Proposition~\ref{prop:affine_modular_form_vanishing_by_torus_fourier_coefficients}
to deduce that the kernel
of~\eqref{eq:thm:genuine_decomposition:fourier_expansion} is trivial.
Hence~\eqref{eq:thm:genuine_decomposition:injection} is an isomorphism,
finishing the proof. 
\end{proof}
\par

\subsection{Cusp forms}

Consider an affine modular-invariant function~$f$ of weight~$k$.
We call~$f$ a \emph{cusp form} if the~$r=n=0$ Fourier--Heisenberg coefficients vanish, i.e.
$c^\rmH(f;\, 0, 0;v,v/y) = 0$. By definition of the Fourier--Heisenberg
expansion, this is equivalent to
\be
c^{\rmH0}(f;\, 0, m;\, y) \= 0  \quad \text{for all $m \in \ZZ$.}
\ee 
We denote by $\rmL^2( \cH(0,0) )^{\gen}_{\cusp}$
the closure of the space of the lifts of genuine cusp forms. (Recall
Lemma~\ref{la:affine_modular_form_genuine_fourier_expansion_torus}
for a characterization of these in terms of Fourier coefficients.)
\par
\begin{proposition} \label{prop:cuspforms}
In terms of the decomposition~\eqref{eq:genuine_decomposition}
the space of cusp forms in the genuine part coincides with
$\oplus_{m = 1}^\infty \pisaff_{0,m}$.
\end{proposition}
\par
\begin{proof}
Proposition~\ref{prop:totCasiEigenvalues} implies
that $\oplus_{m = 1}^\infty \pisaff_{0,m}$
is precisely the subspace where the total Casimir acts with eigenvalue
zero and by~\eqref{eq:casimir_eigenvalue_fourier_term}
this is equivalent to the vanishing of
the Fourier coefficients involved in the definition of a cusp form.
\end{proof}

\subsection{Spectral decomposition of the foliated Laplacian}
\label{sec:diffeqforbeta}

This section prepares for the explicit description of
$\pisaff_{n,m}$ Theorem~\ref{thm:genuine_decomposition_branching} below.
The idea is that the spectral data of the foliated Laplacian~$-\LapFol_k$
suffices to distinguish almost all unitary representations of~$\SL{2}(\RR)$,
in particular, those that appear in the continuous and discrete part
of~$\pisaff_{n,m}$ per Proposition~\ref{prop:saff_representation_sl_branching}.
We thus compute here the solutions of the differential equations
for functions of the form~\eqref{eq:la:induction_elements_on_positive_ray}
that are generalized eigenfunctions of~$-\LapFol_k$. Since we already know the
abstract decomposition of these representations thanks to
Proposition~\ref{prop:saff_representation_sl_branching} we only solve
the generalized eigenvalue equation for the Casimir eigenvalues of the representations
appearing there. Recall that the Casimir eigenvalue of the
discrete series~$\DSL_k$ is equal to $\tfrac{|k|}{2} (\tfrac{|k|}{2}-1)$
and for the complementary series $\ISL_{+,s}$ it is equal to $s^2 - 1/4$.
\par
Recall e.g.\@ from \cite[Section 13.14]{nist-dlmf-1-1-11} that
the \emph{Whittaker differential equation}
\be
  \frac{\rd^2 f}{\rd y^2}
  +
  \Big(
  \mfrac{-1}{4}
  +
  \mfrac{\kappa}{y}
  +
  \mfrac{\frac{1}{4} - \mu^2}{y^2}
  \Big)
  f
\= 0
\ee
has two solutions, traditionally called the \emph{Whittaker
functions~$\rmM_{\kappa,\mu}$ and\/~$\rmW_{\kappa,\mu}$} except if~$2 \mu \in \ZZ_{<0}$.
\par
We have computed in~\eqref{eq:casimir_eigenvalue_fourier_term}
the action of $\LapTot$ on the summands of the Poincar\'e series aiming
for the computation of $\LapTot$-eigenvalues. Using
Lemma~\ref{eq:def:foliated_laplace_and_casimir_operator} we similarly find
(with an auxiliary factor $y^{-\frac{k}{2}}$ that simplifies the equation):
\par
\begin{lemma}%
\label{la:foliated_vertical_laplace_operator_on_fourier_term}
A smooth function~$\beta(y) \defcol \RR^\times \ra \CC$ is mapped under
the foliated Laplace operator to
\bas 
&
 - \LapFol_k\,
  \Big(y^{-\frac{k}{2}} \beta(y)\,
  e\big( n x + m \mfrac{v}{y} \big)\Big)
\\
\= {}&
  y^{-\frac{k}{2}}\,
  \Big(
  -y^2 \beta''(y)
  + 4 \pi^2 n^2\, y^2\beta(y)
  - 2 \pi k n\, y \beta(y)
\Big)\,
e\big( n x + m \mfrac{v}{y} \big)
\tx{.}
\eas
\end{lemma}
\par
We apply the preceding lemma to search for the eigenfunctions and generalized eigenfunctions we expect
according to Proposition~\ref{prop:saff_representation_sl_branching}.
\par
\begin{lemma} \label{la:eigenfunctions}
    For fixed~$k \in \ZZ \setminus \{0,\pm1\}$ and $n \neq 0$ consider the differential equation
\begin{equation} \label{eq:eigenfct_diffeq}
- \LapFol_k\, \beta(y)\, e\big(nx + m \mfrac{v}{y} \big)
\=  \lambda \cdot \beta(y)\,
e\big(nx + m \mfrac{v}{y} \big)\tx{.}
\end{equation}
For $\lambda = \tfrac{|k|}{2} (1 - \tfrac{|k|}{2})$ it has a basis of solutions
consisting of
\begin{gather}
  e^{- 2\pi |n| y}
\quad\tx{and}\quad
  y^{- \frac{k}{2}}\,
  \rmM_{\frac{|k|}{2}, \frac{|k|-1}{2}}(4 \pi |n| y)
\end{gather}
if~$k, n > 0$, and
\begin{gather}
  y^{-k}\,
  e^{- 2\pi |n| y}
\quad\tx{and}\quad
  y^{- \frac{k}{2}}\,
  \rmM_{\frac{|k|}{2}, \frac{|k|-1}{2}}(4 \pi |n| y)
\tx{,}
\end{gather}
if~$k, n < 0$.
\par
For fixed~$k \in \ZZ$, $n\neq 0$ and $\lambda = t^2 + 1/4$ the differential equation
\eqref{eq:eigenfct_diffeq} has a basis of solutions
consisting of
\begin{gather}
\label{eq:differential_equation_foliated_laplace_operator_final_nne0}
  y^{- \frac{k}{2}}\,
  \rmW_{\frac{\sgn(n) k}{2}, it} \big( 4 \pi |n| \, y\big)
\quad\tx{and}\quad
  y^{- \frac{k}{2}}\,
  \rmM_{\frac{\sgn(n) k}{2}, it} \big( 4 \pi |n| \, y\big)\tx{.}
\end{gather}
Finally, for $n=0$ and $\lambda = t^2 + 1/4$, the differential equation~\eqref{eq:eigenfct_diffeq}
has a basis of solutions
\begin{gather}
\label{eq:differential_equation_foliated_laplace_operator_final_n0}
  y^{\frac{1-k}{2} + it}
\quad\tx{and}\quad
  y^{\frac{1-k}{2} - it}
\tx{.}
\end{gather}
\end{lemma}
\par
\begin{proof}
For $n=0$ the function~$\td\beta(y) = y^{\frac{k}{2}}\, \beta(y)$ is a solution of the differential equation
$\td\beta''(y) = \lambda\, y^{-2} \td\beta(y)$, whose solutions are directly
seen to yield~\eqref{eq:differential_equation_foliated_laplace_operator_final_n0}.
\par
For $n \neq 0$ the solutions in~\eqref{eq:differential_equation_foliated_laplace_operator_final_nne0}
follow from the observation that~$y^{\frac{k}{2}}\, \beta( y \slash 4 \pi |n| )$
satisfies the Whittaker differential equation with parameters~$\sgn(n) k \slash 2$
and~$it$.
\par
In the special case that~$|k| > 1$ is an integer with the same sign as~$n$, the
exponential solutions are equal to the~$\rmW$-Whittaker solutions
in~\eqref{eq:differential_equation_foliated_laplace_operator_final_nne0}
by~\cite[Equation 13.14.9]{nist-dlmf-1-1-11}.
\end{proof}
\par
We will need the following asyptotics estimates in the next section to
verify integrability. We define
\be
 \Ga^\rmW(t)
\defeq
  \frac{\Ga(2 i t)}{\Ga\big( \frac{1 - \sgn(n) k}{2} + i t \big)}\,
\tx{,}
\ee
\par
\begin{lemma} \label{la:WhittakerGrowth}
The Whittaker functions $\rmW_{\kappa,\mu}(y)$ decay exponentially as
$y \to \infty$. Moreover the asympoticts of $y \to 0$ is
\begin{gather}
\label{eq:la:prf:genuine_decomposition_branching_continous_whittaker:asymptotic}
\begin{aligned}
&
  (4 \pi |n|\, y)^{-\frac{k}{2}}
  \rmW_{\frac{\sgn(n) k}{2}, i t} \big( 4 \pi |n| \, y\big)\,
\\
\={}&
  \Ga^\rmW(t)\,
  \big( 4 \pi |n| y \big)^{\frac{1-k}{2} - i t}
  +
  \Ga^\rmW(-t)\,
  \big( 4 \pi |n| y \big)^{\frac{1-k}{2} + i t}
  +
  \cO\big( y^{\frac{3-k}{2}} \big)
\tx{.}
\end{aligned}
\end{gather}
\end{lemma}
\par
\begin{proof} This follows from \cite[Equations 13.14.21 and 13.14.16]{nist-dlmf-1-1-11}.
\end{proof}

\subsection{Decomposition as \texorpdfstring{$\SL2(\RR)$}{SL2(R)}-representations}
\label{sec:decompSL2R}

We now state and prove Theorem~\ref{intro:gen_decomp_branching} in the
complete version, including the case $n=0$. Recall that we gave in
Lemma~\ref{la:induction_elements_on_positive_ray} an explicit $\rmL^2$-structure
on the representations $\pip_{n,m}$.
\par
\begin{theorem}%
\label{thm:genuine_decomposition_branching}
For $k \in \ZZ \setminus \{0,\pm 1\}$ and $n \in \ZZ$ with $\sgn(nk)=1$
the representation~$\DSL_{\sgn(n) k}$
in Proposition~\ref{prop:saff_representation_sl_branching}
is generated by the Poincar\'e series
for $\beta = e^{-2 \pi |n| y}$ if\/~$k>1$ and $\beta = y^{-k} e^{-2 \pi |n| y}$
if\/~$k<-1$.
\par
Associating for fixed $n \in \ZZ \setminus \{0\}$ with $\psi \in \rmL^2(\RR^+,dt)$
the lifts of the Poincar\'e series $P_{k;n,m,\beta^\rmW_{k,n,\psi}}$ of the Whittaker
transform
\begin{gather*}
  \beta^\rmW_{k,n,\psi}(y)
  \defeq
  \frac{1}{4\pi\, |n|^{\frac{3}{2}}}
  \int_{t \in \RR^+}
 \frac{\psi(t)}{(\Ga^\rmW(t) \Ga^\rmW(-t))^{\frac{1}{2}}}\,
  y^{-\frac{k}{2}}
  \rmW_{\frac{\sgn(n) k}{2}, i t} \big( 4 \pi |n| \, y\big)\,
  \,\rmd t
\end{gather*}
gives isometric embeddings
\bes
\rmP^\rmW_+:  \bigoplus_{k \in 2 \ZZ}   \rmL^2\big( \RR^+, \rmd t \big)
\to \pisaff_{n,m}, \qquad 
\rmP^\rmW_-:  \bigoplus_{k \in 1+ 2 \ZZ}   \rmL^2\big( \RR^+, \rmd t \big)
\to \pisaff_{n,m} 
\ees
whose images are $\int^\oplus_{\RR^+} \ISL_{+,it} \,\rmd t$ and
$\int^\oplus_{\RR^+} \ISL_{-,it} \,\rmd t$ respectively, in the
decomposition of Proposition~\ref{prop:saff_representation_sl_branching}.
\par
Associating with $\psi \in \rmL^2(\RR^+,dt)$ the lifts of
the Eisenstein series $E_{k;m,\beta^\rmW_{k,n,\psi}}$ of the `y-power 
transform'
\be
 \beta^{\rmc\pm}_{k,\psi}(y)
\defeq
  \int_{t \in \RR^+}
  \psi(t)\,
  \big( y^{\frac{1-k}{2} + i t} \pm y^{\frac{1-k}{2} - i t} \big)
  \,\rmd t
\ee
gives isometric embeddings
\bes
\rmE^{\rmc\pm}_+:  \bigoplus_{k \in 2 \ZZ}   \rmL^2\big( \RR^+, \rmd t \big)
\to \pisaff_{0,m}, \qquad 
\rmE^{\rmc\pm}_-:  \bigoplus_{k \in 1+ 2 \ZZ}   \rmL^2\big( \RR^+, \rmd t \big)
\to \pisaff_{0,m} 
\ees
whose images are one of the two summands $\int^\oplus_{\RR^+} \ISL_{+,it} \,\rmd t$
and $\int^\oplus_{\RR^+} \ISL_{-,it} \,\rmd t$ respectively, in the
decomposition of Proposition~\ref{prop:saff_representation_sl_branching}.
Moreover, the images of\/~$\rmE^{\rmc+}_{\pm}$ and~$\rmE^{\rmc-}_{\pm}$ are
orthogonal.
\end{theorem}
\par
\begin{proof}
We start with the discrete series and observe
that~$y^{\frac{k-1}{2}} \beta(y) \in \rmL^2(\RR^+,\rmd y/y)$. Therefore,
the Poincar\'{e} series in the first statement is defined as
an~$\rmL^2$\nbd{}limit of the series in Lemma~\ref{la:affine_eisenstein_poincare_series_convergent}. Recall that the representation generated by~$P_{k;n,m,\beta}$
is isomorphic by Theorem~\ref{thm:genuine_decomposition} to the one
generated by~$\wht\beta$, defined as the lift of $\beta(y)\,
e\big( n x + m {v}/{y} \big)$. This function is smooth and square-integrable
on~$\SAff{2}(\RR)$. In particular, we can compute the action $\CasFol$
pointwise. Furthermore, the definition
of~$\CasFol$ via~$\LapFol$ in Lemma~\ref{la:DescendtoLaplace}
allows us to compute it on~$\HSp$. Now the first statement of
Lemma~\ref{la:eigenfunctions} confirms the existence of the eigenspaces.
(Note that the M-Whittaker function given as the second solution in that lemma
has exponential growth as $y \to \infty$ and will not give an $\rmL^2$-function.)
\par
Our argument for the principal series follows the argument for the modular
surface (e.g.~\cite[Section~4.2.5]{Bergeron}) with one major difference
in Lemma~\ref{la:WhittakerIso}. Because we cannot evaluate exactly the inner
products of truncated $\rmW$\nbd{}Whittaker functions that we define in the
proof, we need to estimate some of their contribution via asymptotic remainder
terms.
\par
Suppose $n \neq 0$. Thanks to Lemma~\ref{la:WhittakerGrowth}
partial integration with respect to~$t$ in the defining equation
for~$\beta^\rmW_{k,n,\psi}$ shows that~$y^{(k-1) \slash 2}\, \beta^\rmW_{k,n,\psi}(y)$
is square-integrable with respect to the Haar measure on~$\RR^+$. Using
Lemma~\ref{la:WhittakerIso} below and
Lemma~\ref{la:automorphic_eisenstein_isometry} we conclude that
assigning with~$\psi$ the Poincar\'e series is an isometry as claimed.
To finish the proof in this case we apply Weyl's criterion for essential spectrum
membership for any $t_0 \in \RR$. Let $D = - \LapFol_k - (t_0^2 +1/4)$ and $\psi_n(t)$
a sequence of bump functions limiting to $t_0$ with $\|\psi_n\| =1$.
Then $\rmG$-invariance of the Laplacian, abolute convergence and isometry
of the $E$-operator (by Lemma~\ref{la:affine_eisenstein_poincare_series_convergent}
and Lemma~\ref{la:affine_eisenstein_poincare_series_fourier_expansion_along_heisenberg})
and the eigenvalue property of the Whittaker function from Lemma~\ref{la:eigenfunctions}
imply
\bas
& \phantom{\=} \big\| D(P_{k,n,m,\beta^W_{k,n,\psi_n}})) \big\|
\= \big\| E\big( D(\beta^W_{k,n,\psi_n} \exp(..)),\cdot \big) \big\|
\= \big\| D(\beta^W_{k,n,\psi_n} \exp(..)) \big\| \\
& \=  \Bigg| \frac{1}{4\pi\, |n|^{\frac{3}{2}}}
  \int_{t \in \RR^+}
 \frac{\psi_n(t)(t^2-t_0^2)}{(\Ga^\rmW(t) \Ga^\rmW(-t))^{\frac{1}{2}}}\,
  y^{-\frac{k}{2}}
  \rmW_{\frac{\sgn(n) k}{2}, i t} \big( 4 \pi |n| \, y\big)\,
  \,\rmd t \Bigg| \quad\to 0\,\tx{,}
\eas
since on the support of $\psi_n$ the factor $(t^2-t_0^2)$ becomes small.
This shows that $t_0^2  + 1/4$ is an approximative eigenvalue and since
the discrete spectrum is associate with positive eigenvalues, our knowledge
about the total spectral decomposition from
Proposition~\ref{prop:saff_representation_sl_branching} shows
that we have covered everything.
\par
The case $n=0$ is similar. By partial integration, we see that~$y^{(k-1) \slash 2}\,
\beta^{\rmc\pm}_{k,\psi}(y)$ is square-integrable. Now Lemma~\ref{la:ypowerIso}
below and Lemma~\ref{la:automorphic_eisenstein_isometry} show the isometry
claim. To show the orthogonality one proceeds as in Lemma~\ref{la:ypowerIso},
but now the first and last term cancel and we get zero as $\epsilon \to 0$.
The spectral conclusion is similar as above using Weyl's criterion,
Lemma~\ref{la:eigenfunctions} and
Proposition~\ref{prop:saff_representation_sl_branching}, arguing separately
for each of the two orthogonal summands.
\end{proof}
\par
\begin{lemma} \label{la:WhittakerIso}
The Whittaker transform $\psi \mapsto  \beta^\rmW_{k,n,\psi}$
is an isometry $\rmL^2\big( \RR^+, \rmd t \big) \to \rmL^2(\RR^+,y^{k-2}\rmd y)$.
\end{lemma}
\par
\begin{proof} We first verify the claim for smooth compactly supported
functions. To this end, we introduce for $\epsilon >0$ the
truncated $\rmW$\nbd{}Whittaker functions
\bas
&
  \rmW^\epsilon_{\frac{\sgn(n) k}{2}, it} \big( 4 \pi |n| \, y\big)
\\
\defeq{}&
  \rmW_{\frac{\sgn(n) k}{2}, i t} \big( 4 \pi |n| \, y \big)
  \,-\,
  \bbone_{(0,\epsilon)}(y)\,
  \Big(
  \Ga^\rmW(t)\,
  \big( 4 \pi |n| y \big)^{\frac{1}{2} - i t}
  +
  \Ga^\rmW(-t)\,
  \big( 4 \pi |n| y \big)^{\frac{1}{2} + i t}
  \Big)
\eas
and denote by $\beta^{\rmW\epsilon}_{k,n,\psi}(y)$ the Whittaker transform
with respect to the truncated Whittaker functions.
Using partial integration with respect to~$t$ we see that
\begin{gather*}
  \big\| 
  y^{\frac{k-1}{2}} \big(
  \beta^\rmW_{k,n,\psi}(y)
  -
  \beta^{\rmW\epsilon}_{k,n,\psi}(y)
  \big)
  \big\|^2
\ll
  \int_0^\epsilon
  \mfrac{1}{y \log(y)^2}
  \,\rmd y
\ll
  \log(\epsilon)^{-1}
\tx{.}
\end{gather*}
Combining this estimate with the Cauchy-Schwartz inequality, we conclude that
\begin{gather*}
  \big\| 
  y^{\frac{k-1}{2}}\,
  \beta^\rmW_{k,n,\psi}(y)
  \big\|^2
\=
  \big\| 
  y^{\frac{k-1}{2}}\,
  \beta^{\rmW\epsilon}_{k,n,\psi}(y)
  \big\|^2
  +
  \cO\big( \log(\epsilon)^{-\frac{1}{2}} \big)
\tx{.}
\end{gather*}

We next expand the defining integral for the~$\rmL^2$\nbd{}norm and interchange
the integration with respect to~$y$, $t_1$, and~$t_2$, which is justified
because all integrands are nonnegative:
\begin{align*}
&
  \big\| 
  y^{\frac{k-1}{2}}\,
  \beta^{\rmW\epsilon}_{k,n,\psi}(y)
  \big\|^2
\\
\={}&
  \big( 4 \pi |n|\, \big)^{-k}\,
  \int_{t_1,t_2 \in \RR^+}
  \psi(t_1) \ov{\psi(t_2)}\,
  \big( \Ga^\rmW(t_1) \Ga^\rmW(- t_1)\, \Ga^\rmW(t_2) \Ga^\rmW(- t_2) \big)^{-\frac{1}{2}}
\\
&\mspace{90mu}
  \int_{\RR^+}
  \rmW^\epsilon_{\frac{\sgn(n) k}{2}, i t_1} \big( 4 \pi |n| \, y\big)\,
  \rmW^\epsilon_{\frac{\sgn(n) k}{2}, -i t_2} \big( 4 \pi |n| \, y\big)\,
  \,y^{-2} \rmd y
  \,\rmd t_1 \rmd t_2
\tx{.}
\end{align*}
Using the  asymptotic expansion of the Whittaker function in~\eqref{eq:la:prf:genuine_decomposition_branching_continous_whittaker:asymptotic}, we can determine the leading asymptotic with respect to~$\epsilon$ of the inner integral. For~$\epsilon_1 > \epsilon_2 > 0$, we have
\begin{align*}
&
  \int_{\RR^+}
  \rmW^{\epsilon_2}_{\frac{\sgn(n) k}{2}, i t_1} \big( 4 \pi |n| \, y\big)\,
  \rmW^{\epsilon_2}_{\frac{\sgn(n) k}{2}, -i t_2} \big( 4 \pi |n| \, y\big)\,
  \,y^{-2} \rmd y
\\
&\quad
  \,-\,
  \int_{\RR^+}
  \rmW^{\epsilon_1}_{\frac{\sgn(n) k}{2}, i t_1} \big( 4 \pi |n| \, y\big)\,
  \rmW^{\epsilon_1}_{\frac{\sgn(n) k}{2}, -i t_2} \big( 4 \pi |n| \, y\big)\,
  \,y^{-2} \rmd y
\allowdisplaybreaks
\\
\={}&
  \int_{\epsilon_2}^{\epsilon_1}
  \Big(
  \Ga^\rmW(t_1)\,
  \big( 4 \pi |n| y \big)^{\frac{1}{2} - i t_1}
  +
  \Ga^\rmW(-t_1)\,
  \big( 4 \pi |n| y \big)^{\frac{1}{2} + i t_1}
  \Big)\,
\\
&\qquad\quad
  \Big(
  \Ga^\rmW(-t_2)\,
  \big( 4 \pi |n| y \big)^{\frac{1}{2} + i t_2}
  +
  \Ga^\rmW(t_2)\,
  \big( 4 \pi |n| y \big)^{\frac{1}{2} - i t_2}
  \Big)
  \,y^{-2} \rmd y
  \;+\;
  \cO(\epsilon_1 - \epsilon_2)
\allowdisplaybreaks
\\
\={}&
  \Ga^\rmW( t_1) \Ga^\rmW(-t_2)\,
  \big( 4 \pi |n| \big)^{1 - i t_1 + i t_2}\,
  \frac{\epsilon_1^{- i t_1 + i t_2} - \epsilon_2^{- i t_1 + i t_2}}{- i t_1 + i t_2}
\\
&\quad
  +\,
  \Ga^\rmW(-t_1) \Ga^\rmW(-t_2)\,
  \big( 4 \pi |n| \big)^{1 + i t_1 + i t_2}\,
  \frac{\epsilon_1^{+ i t_1 + i t_2} - \epsilon_2^{  i t_1 + i s_2}}{  i t_1 + i t_2}
\\
&\quad
  +\,
  \Ga^\rmW( t_1) \Ga^\rmW( t_2)\,
  \big( 4 \pi |n| \big)^{1 - i t_1 - i t_2}\,
  \frac{\epsilon_1^{- i t_1 - i t_2} - \epsilon_2^{- i t_1 - i t_2}}{- i t_1 - i t_2}
\\
&\quad
  +\,
  \Ga^\rmW(-t_1) \Ga^\rmW( t_2)\,
  \big( 4 \pi |n| \big)^{1 + i t_1 - i t_2}\,
  \frac{\epsilon_1^{+ i t_1 - i t_2} - \epsilon_2^{  i t_1 - i t_2}}{  i t_1 - i t_2}
  \;+\;
  \cO(\epsilon_1 - \epsilon_2)
\tx{.} 
\end{align*}

To evaluate the integral with respect to~$t_2$, we can perform the same steps
as in \cite[Proposition~4.15]{Bergeron} towards the end of his proof.
Since~$t_1, t_2 \in \RR^+$, the second and third term in the inner integral are
regular with respect to~$t_1$ and~$t_2$ and thus yield
contributions of order~$\log(\epsilon)^{-1}$ after partial integration with respect to
either of them. It remains to consider the first and fourth term, which yield
\begin{gather*}
  \big\| 
  y^{\frac{k-1}{2}}\,
  \beta^{\rmW\epsilon}_{k,n,\psi}(y)
  \big\|^2
\=
  4 \pi
  \big( 4 \pi |n|\, \big)^{1-k}
  \int_{t_1 \in i \RR^+}
  \psi(t_1) \ov{\psi(t_1)}\,
  \,\rmd t_1
  \,+\,
  \cO\big( \log(\epsilon)^{-1} \big)
\tx{.}
\end{gather*}
This establishes the claimed isometry, when letting~$\epsilon$ tend to~$0$.
It also guarantees that the assigment from~$\psi$ to~$\beta$ extends to a map
on the $\rmL^2$-spaces as claimed. 
\end{proof}
\par
With similar arguments we show:
\par
\begin{lemma} \label{la:ypowerIso}
  The y-power transform $\psi \mapsto  \beta^\rmW_{k,n,\psi}$
is an isometry $\rmL^2\big( \RR^+, \rmd t \big) \to \rmL^2(\RR^+,y^{k-2}\rmd y)$.
\end{lemma}
\par
\begin{proof}
The main difference to the previous lemma is that we need to truncate both
towards~$0$ and~$\infty$. We thus define
\begin{gather*}
  \beta^{\rmc\pm\epsilon}_{k,\psi}(y)
\defeq
  \int_{t \in \RR^+}
  \psi(t)\,
  \bbone_{(\epsilon, 1 \slash \epsilon)}\,  
  \big( y^{\frac{1-k}{2} + i t} \pm y^{\frac{1-k}{2} - i t} \big)
  \,\rmd t
\tx{.} 
\end{gather*}
Similar calculations as above yield
\begin{align*}
&
  \big\|
  y^{\frac{k-1}{2}}\,
  \beta^{\rmc\pm\epsilon}_{k,\psi}(y)
  \big\|^2
\\
\={}&
  \int_{t_1,t_2 \in \RR^+} \!\!\!\!\!\!\!\!\!\!
  \psi(t_1) \ov{\psi(t_2)}\,
  \int_\epsilon^{1 \slash \epsilon} \!\!\!\!
  \big( y^{+ i t_1} \pm y^{- i t_1} \big)
  \big( y^{- i t_2} \pm y^{+ i t_2} \big)
  \,y^{-1} \rmd y
  \,\rmd t_1 \rmd t_2
\tx{.}
\end{align*}
The inner integral equals
\begin{align*}
&
  \int_\epsilon^{1 \slash \epsilon}
  \big( y^{+ i t_1} \pm y^{- i t_1} \big)
  \big( y^{- i t_2} \pm y^{+ i t_2} \big)
  \,y^{-1} \rmd y
\\
={}&
  \int_\epsilon^{1 \slash \epsilon}\big(
  y^{+ i t_1 - i t_2}
  \pm
  y^{- i t_1 - i t_2}
  \pm
  y^{+ i t_1 + i t_2}
  +
  y^{- i t_1 + i t_2}
  \big)
  \,y^{-1} \rmd y
\\
={}&
  \frac{\epsilon^{i t_2 - i t_1} - \epsilon^{i t_1 - i t_2}}{i t_1 - i t_2}
  \pm
  \frac{\epsilon^{i t_1 + i t_2} - \epsilon^{- i t_1 - i t_2}}{- i t_1 - i t_2}
\\
\hphantom{={}}&
\quad
  \pm
  \frac{\epsilon^{- i t_1 - i t_2} - \epsilon^{i t_1 + i t_2}}{i t_1 + i t_2}
  +
  \frac{\epsilon^{i t_1 - i t_2} - \epsilon^{i t_2 - i t_1}}{- i t_1 + i t_2}
\tx{.}
\end{align*}
The second and third term, {which agree}, contribute~$\cO(\log(\epsilon)^{-1})$ to the final expression. From the first and fourth term, {which are also equal}, we obtain
\begin{align*}
  \big\|
  y^{\frac{k-1}{2}}\,
  \beta^{\rmc\pm\epsilon}_{k,\psi}(y)
  \big\|^2
=
  4 \pi\,
  \int_{t_1 \in \RR^+}
  \psi(t_1) \ov{\psi(t_1)}\,
  \,\rmd t_1
  \;+\;
  \cO\big( \log(\epsilon)^{-1} \big)
\end{align*}
and the claim follows taking the limit $\epsilon \to 0$.
\end{proof}

\subsection{The compound operator}

The goal of this subsection is to understand the spectral decomposition of
$-\LapCmp{\epsilon}_k$  and prove Theorem~\ref{intro:compound}. This decomposition
is closer in nature to that of the
Laplacian on the modular surface. The following theorem is claimed without
proof for $k = 0$ and $\epsilon = 4$ in \cite{balslev-2011}.
\par
\begin{theorem} \label{thm:compound_discrete_spectrum}
  For every $k \in \N$ and $\epsilon > 0$, the $\rmK$-type $k$ cusp forms
  are an invariant
    subspace of the compound Laplacian $-\LapCmp{\epsilon}_k$ on which it has
    discrete spectrum.
\end{theorem}
\par
\begin{proof}
That cusp forms are an invariant subspace can be seen by computing the
compound Laplacian term by term in the Fourier expansion. We first consider
the principal self-adjoint part of the compound Laplacian, that is the
operator $L = -\LapCmp{\eps}_k - i k y \del_x$ and prove that
$L$ has compact resolvent. We can read from
Proposition~\ref{prop:compoundproperties} that the quadratic form associated
with~$L$ is
\bas
        Q_L(\phi) &\= \int_{\Gamma' \backslash \HSp} y^k |\nabla_{x,y} \phi|^2 \de
        x \de y \de p \de q +
        \eps \int_{\Gamma'\backslash \HSp} y^{k-2} |\nabla_{u,v} \phi|^2 \de x
        \de y \de u \de v \\
        &\=  \sum_{n,m \in \ZZ \setminus \{0\}} \left( \int_{\RR^+}
    |c^{\rmH0}(y\del_x\phi;n,m;y)|^2 + |c^{\rmH0}(y \del_y \phi;n,m;y)|^2
\frac{\de y}{y^{2-k}} \right. \\
&\quad \left. \,+\, \eps \int_{\RR^+} |c^{\rmH0}(\del_u \phi;n,m;y)|^2 +
|c^{\rmH0}(\del_v \phi;n,m;y)|^2 \frac{\de y}{y^{2-k}}\right),
\eas
where the second equality follows from the Parseval identity for the
Fourier--Heisenberg series,
Lemma~\ref{lemma:SPviaFourier}.
To prove discreteness of the spectrum, we need to prove that
the set
\begin{equation}
    A = \left\{\phi \in \RL^2(\cH(0,0))^\gen_\cusp : Q_L(\phi) \le
    1 \right \}
\end{equation}
is compact in the $\RL^2$ topology. We adapt the proof strategy in
\cite[Lemma~8.7]{LaxPhillips}. Note that for any $a,b > 0$, in the
compact region of $\Gamma' \backslash \HSp$ where $a < y < b$, the quadratic
form defines a norm equivalent to the standard Sobolev norm on $\mathrm
W^{1,2}(\Gamma'\backslash \HSp)$;
since genuine affine-invariant modular form have mean zero, the
Rellich--Kondrachov theorem \cite[Theorem~6.3]{AdamsFournier} tells us that
the set of genuine cusp forms in~$A$ supported on $a < y < b$ is compact in
the $\RL^2$ topology (note that the theorem is usually stated in euclidean space,
but is a purely local statement and so holds in the bulk of $\Gamma'
\backslash \HSp$). Using
Lemma~\ref{lemma:SPviaFourier}, 
to prove compactness it suffices to prove that cusp forms in~$A$ satisfy
uniformly
 \begin{equation}
     \label{eq:compactcuspatorigin}
     0 \= \lim_{a \to 0} \int_0^a \sum_{m,n \in \ZZ \setminus \{0\}}
     \Big|c^{\rmH0}(\phi;n,m;y)\Big|^2 \frac{\de y}{y^{2-k}}
 \end{equation}
 and
 \begin{equation}
     \label{eq:compactcuspatinfinity}
     0 \= \lim_{b \to \infty} \int_b^\infty \sum_{m,n \in \ZZ \setminus \{0\}}
     \Big|c^{\rmH0}(\phi;n,m;y)\Big|^2 \frac{\de y}{y^{2-k}}.
 \end{equation}
By definition of cusp forms, they have no $n = 0$ terms in their
Fourier--Heisenberg series. As such, we can
study~\eqref{eq:compactcuspatinfinity} as in the case of the modular surface,
using that
 $$
    |c^{\rmH0}(\phi;n,m;y)|^2 \le 4 \pi n^2 |c^{\rmH0}(\phi;n,m;y)|^2 \=
    y^{-2} |c^{\rmH0}(y \del_x \phi;n,m;y)|^2$$ 
    we deduce
\bas
        \int_b^\infty \sum_{m,n \ne 0} \Big|c^{\rmH0}(\phi;n,m;y)\Big|^2 \frac{\de
        y}{y^{2-k}} &\le b^{-2} \int_b^\infty \sum_{m,n \ne 0 } \Big|c^{\rmH0}(y \del_x
        \phi;n,m;y)\Big|^2 \frac{\de y}{y^{2-k}}\\ &\le \frac{Q_L(\phi)}{b^2}.
\eas
For \eqref{eq:compactcuspatorigin}, we use this time the derivative in the
$v$ direction. It follows from Corollary \ref{cor:genuineextravanishing}
that there are no $m = 0$ terms in the Fourier expansion of a
genuine cusp form we can use 
$$y^{-2} |c^{\rmH0}(\phi;n,m;y)|^2 \le 4 \pi m^2 y^{-2} |c^{\rmH0}(\phi;n,m;y)|^2
\=
|c^{\rmH0}(\del_v \phi;n,m;y)|^2$$
to deduce
\bas
        \int_0^a \sum_{m,n \ne 0} \Big|c^{\rmH0}(\phi;n,m;y)\Big|^2 \frac{\de
        y}{y^{2-k}} &\,\le\, a^2  \int_0^a y^{k-2} \sum_{m,n \ne 0
        }\Big|c^{\rmH0}(\del_v
        \phi;n,m;y)\Big|^2 \de y \\& \,\le\, \frac{a^2}{\eps} Q_L(\phi).
\eas
All in all, this implies that $A$ is compact, so that the $L$ restricted to
cusp forms has compact resolvent and therefore discrete spectrum with finite
multiplicity. 
\par
Recall now that we defined $-\LapCmp{\eps}_k = L + i k y \del_x$, and we now
aim to show that $i k y \del_x$ is relatively compact with respect to~$L$.
From
\cite[Theorem~IV.5.35]{Kato} we know that a relatively compact perturbation
does not change the essential spectrum, and the essential spectrum of $L$ is
empty.
To show that $i k y \del_x$ is relatively compact with respect to $L$, it is
sufficient to show that for some $\lambda \in \CC$, $i k y \del_x (L - \lambda)^{-1}$ is
compact. We just proved that if~$z$ is not an eigenvalue of $L$ then $(L -
\lambda)^{-1}$ is compact, so that by the functional calculus $(L - \lambda)^{-1/2}$
also is. It also follows from its definition that~$L$ has the same principal symbol
as $-\LapCmp{\eps}_k$ and as such is also a second order elliptic operator,
therefore $(L - \lambda)^{-1/2}$ is a pseudodifferential operator of order $-1$
\cite[Theorem 2]{SeeleyComplexPowers}, so that $i k y \del_x (L-\lambda)^{-1/2}$ is a
pseudodifferential operator of order~$0$, and therefore bounded by the
Calder\'on--Vaillancourt theorem \cite[Theorem~18.1.11]{HormanderIII}.
Through this discussion, we obtain that
\bes
    i k y \del_x(L - \lambda)^{-1} = (i k y \del_x (L - \lambda)^{-1/2}) (L
    - \lambda)^{-1/2}
\ees
is the composition of a bounded and compact operator, and as such compact.
Therefore $-\LapCmp{\eps}_k$ also has compact resolvent when restricted to
cusp forms, and as such discrete spectrum.
\end{proof}
\par
\begin{proposition}
    \label{prop:src}
    For every $\lambda \in \CC \setminus \RR$, 
    $(\LapCmp{\eps}_k + \lambda)^{-1}
    \to (\LapFol_k + \lambda)^{-1}$ 
    in the strong operator topology as $\eps \to
    0$.
\end{proposition}

\begin{remark}
    Because cusp forms are an invariant subspace for both $\LapCmp{\eps}_k$ and
    $\LapFol_k$ it also means that the restriction of the resolvents to cusp
    forms or to their orthogonal complement also converge appropriately in the
    strong operator topology.
\end{remark}

\begin{proof}
    Recall that the strong operator topology is induced by pointwise
    convergence. Let $f \in \RL^2(\cH(0,0))$, we need to show  
    $(\LapCmp{\eps}_k + \lambda)^{-1}f
    \to (\LapFol_k + \lambda)^{-1}f$.
Since 
$$
\|(-\LapCmp{\eps}_k + \lambda)^{-1}\| \le
\dist(\lambda,\spec(-\LapCmp{\eps}_k))^{-1} \le |\Im(\lambda)|^{-1},
$$
the family $(-\LapCmp{\eps}_k  + \lambda)^{-1}$ is uniformly bounded; therefore
it is sufficient to verify this pointwise convergence of $(-\LapCmp{\eps}_k +
\lambda)^{-1} \to (-\LapFol_k +\lambda)^{-1}$ on a dense subset of
$\rmL^2(\cH(0,0))$, and in particular to verify it on functions 
of Schwartz class in $y$, which we now assume $f$ to be.
From the second resolvent identity, we have that
\bes
    (-\LapCmp{\eps}_k + \lambda)^{-1} f - ( -\LapFol_k + \lambda)^{-1} f \= 
    \eps
    (-\LapCmp{\eps}_k + \lambda)^{-1} \LapVert (-\LapFol_k + \lambda)^{-1} f.
\ees
Since $-\LapFol_k + \lambda$ is hypoelliptic, its inverse is a pseudodifferential
operator \cite{HormanderHypoelliptic}, in particular the Schwartz class of functions is stable under 
$\LapVert (-\LapFol_k + \lambda)^{-1}$. Furthermore, the Schwartz class
embeds boundedly in $\rmL^2(\cH(0,0))$, and we have previously indicated that
the family $(- \LapCmp{\eps}_k +\lambda)^{-1}$ is uniformly bounded as operators
on $\rmL^2$. Consequently, for any Schwartz function $f$,
$$
\|
    \eps
    (-\LapCmp{\eps}_k + \lambda)^{-1} \LapVert (-\LapFol_k + \lambda)^{-1} f
    \|_{\rmL^2} = O(\eps),
$$
and we deduce the strong convergence of the resolvents.
\end{proof}
As a corollary, we get that the family $\{-\LapCmp{\eps}_k\}$ is spectrally
inclusive, meaning that the spectrum of $\{-\LapFol_k\}$ is comprised of limit
points from the spectra of $\{-\LapCmp{\eps}_k\}$.
\begin{corollary}
    \label{cor:spectral_inclusion}
    For every $\lambda \in \spec(-\LapFol_k)$, there is a family
    $\{\lambda_\eps \in \spec(-\LapCmp{\eps}_k)\}$ such that $\lambda_\eps \to
        \lambda$. Furthermore, for every bounded continuous function $f : \RR \to
        \RR$, $f(-\LapCmp{\eps}_k) \to f(-\LapFol_k)$ in the strong operator topology.
\end{corollary}
\par
\begin{proof}
The second statement is a consequence of strong convergence of the resolvent and
\cite[Theorem 9.17]{Weidmann}. Suppose that $\lambda \in \spec(-\LapFol_k)$ is
not a limit point of the spectra of $-\LapCmp{\eps}_k$, and take $f$ to be a
function supported near $\lambda$ and so that $\supp(f) \cap
\spec(-\LapCmp{\eps}_k)=\varnothing$ for all sufficiently small $\eps$. Then, it
would be impossible for $f(-\LapCmp{\eps}_k)$ to converge to $f(-\LapFol_k)$, in
the strong topology, proving the corollary.
\end{proof}
\par
\begin{remark}
Let us make a few remarks about Proposition~\ref{prop:src} and Corollary~ \ref{cor:spectral_inclusion}. First, since cusp forms are an invariant subspace of $\LapFol_k$ and $\LapCmp{\eps}_k$, the statements apply
\emph{mutatis mutandis} to the restriction of those operators to cusp forms
or their orthogonal complement. Second, the resolvents do not converge in the
operator norm topology, to see this it suffices to compare their action on a
sequence $f(\tau) e(\lfloor \eps^{-1}\rfloor q )$ for a fixed $f$. Finally, the convergence of
$f(-\LapCmp{\eps}_k) \to f(-\LapFol_k)$ in the strong operator topology is a
bit weaker than convergence of the spectral projections but for most
intents and purposes can be used the same way. Note that by monotonicity
of the involved operators ($-\LapFol_k \le - \LapCmp{\eps}_k$ for all $\eps >
0$), the condition on the continuity of $f$ can be relaxed to right-continuity.
\end{remark}

\section{Siegel--Veech transforms}%
\label{sec:siegel_veech_transforms}

In this section we briefly recall basic properties of the Siegel--Veech
transforms for any configuration on any stratum. We then specialize to our
case of the stratum $\cH(0,0)$ and prove the main results,
Theorem~\ref{intro:SVupperbound} and Theorem~\ref{intro:SVperp}, in particular
showing that Siegel--Veech transforms exhaust the complement of cusp forms.
The main technical step is the computation of Fourier coefficients of
Siegel--Veech transforms in
Proposition~\ref{prop:siegel_veech_fourier_coefficient_span}.
We complement this in Section~\ref{sec:L2svrelM} by computing adjoints and kernels
of Siegel--Veech transforms.
\par

\subsection{Basic properties}%

A flat surface $(X,\omega) \in \cH(\alpha)$ determines a \emph{singular flat metric}, with cone points of angle $2\pi(\alpha_i + 1)$ where $\omega$ has a zero of order $\alpha_i$. A \emph{saddle connection} $\gamma$ is a geodesic in the flat metric connecting two zeros, with none in its interior. We denote the set of saddle connections by $\SC(\omega)$. To each saddle connection $\gamma$ we associate the \emph{holonomy vector} $\hol({\gamma}) = \int_{\gamma} \omega \in \CC$.
\par
A \emph{configuration} $\cC$ is a choice of subset $\cC(\omega) \subset \SC(\omega)$
such that if we set
$$\Lambda^{\mathcal C}_{\omega} \= \{\hol(\gamma): \gamma \in \cC(\omega)\}\tx{,}$$
the assignment  $$\omega\, \mapsto  \, \Lambda^{\mathcal C}_{\omega}$$ is
$\SL2(\RR)$-equivariant. Examples of configurations include the set of saddle
connections joining  two specified zeros, two zeros of specified orders, saddle
connections that sit at the boundary of cylinders in a fixed homotopy class, etc.
Given any configuration~$\cC$ and a function~$f:\RR^2 \to \CC$, we define
the \emph{Siegel--Veech transform with respect to~$\cC$} as
$$\SV_{\cC}(f): \cH(\alpha) \rightarrow \CC, \qquad
(X,\omega) \mapsto \sum_{v \in \Lambda^{\cC}_{\omega}} f(v).$$
By definition $\SV_{\cC}(g \cdot f) = g \cdot \SV_{\cC}(f)$
for any $g \in \SL2(\RR)$.
\par
If the function~$f$ is of $\rmK$-type~$k$, then the Siegel--Veech transform is
the lift (in the sense of~\eqref{eq:def:modular_to_automorphic}) of an affine
modular-invariant function of weight~$k$ on~$\HS'$. We indicate that we
work with this function by writing a pair of variable $\SV_\cC(f)(\tau,z)$
with $(\tau,z) \in \HS'$ as the argument of the  Siegel--Veech transform.
\par
We now specialize to the stratum~$\cH(0,0)$ we are mainly interested in. In this
case there are two obvious configurations. The first consists of \emph{absolute
periods}, the configurations of saddle connections joining (say) the first zero
to itself. If the second zero is not a rational point with respect to the period
lattice based at the first zero, then $\Lambda_\omega^\abs = \Lambda_\omega^\prim$
consists of the primitive lattice vectors in the period lattice underlying
$(X,\omega)$. Since we
consider Siegel--Veech transforms as $\rmL^2$-functions, we may ignore the measure
zero complementary set. By definition $\SVabs(f)$ factors through the
projection to~$\cH(0)$, i.e., contributes to the non-genuine part of
$L^2(\cH(0,0))$. 
It in fact orthogonal to cusp forms and covers their
orthogonal complement,
i.e., the space of Eisenstein transforms in any weight~$k$, since
\begin{equation}\label{eq:eisenstein}
E_k(\tau | \psi)  \=  \sum_{\substack{(c,d) \in \ZZ^2 \\ \gcd(c,d) = 1}}
\frac{(c\tau+d)^k} {|c\tau+d|^{k}} \psi\Bigl(\frac{\Im(\tau)}
{|c\tau +d|^{2}} \Bigr) \= \SVabs(f)(\Lambda_\tau)
\end{equation}
for $f(\lambda)  =  \frac{\lambda^k}{|\lambda|} \psi(1/|\lambda|^2)$ and
$\Lambda(\tau) = \langle \frac{\tau}{\sqrt{y}}, \frac{1}{\sqrt{y}} \rangle$.
Here we consider $\Lambda \subset \RR^2 \cong \CC$ when taking powers of
elements in~$\Lambda$. 
\par
The second case consists of \emph{relative periods}, the configurations
of saddle connections $\Lambda_\omega^\rel$ joining (say) the first zero to the
second zero. We denote the corresponding Siegel--Veech transform by $\SVrel(f)$.
Having decompositions of $\rmL^2$-spaces in mind we may ignore the flat
surfaces where the relative period is a real multiple of an absolute period,
since this is a measure zero set. Consequently, the definition
of~$\Lambda_\omega^\rel$ does not involve any primitivity condition on
lattice vectors. However, these two cases do not exhaust all configurations:
\par
\begin{lemma}\label{lemma:CM} For any $M \in \NN$ assigning with 
$(\Lambda, z) \in \cH(0,0)$
the set $\cC_M = z + \tfrac1M \Lambda$ of translates of
the relative period by a $1/M$-th lattice vector is a configuration.
\end{lemma}
\par
\begin{proof} Both independence of the choice of the relative period and
$\rmG(\RR)$-equivariance are obvious.
\end{proof}
\par
We let $\SVrelM := \SV_{\cC_M}$ be the corresponding Siegel-Veech
transformation. That is,
\begin{equation}
\label{eq:def:svrelm}
  \SVrelM(f)(\Lambda, z) = \sum_{w \in \cC_M} f(w)
\tx{.}
\end{equation}
\par
\begin{remark} \label{rem:relSVaspushpull}
In this homogeneous space setting, the relative Siegel--Veech transform
can also be stated as follows. Let $\rmS'(\RR) \subset \Gp(\RR)$ be the
stabilizer of $(1,0)$ with respect to the right action and $\rmS'(\ZZ) =
\rmS'(\RR) \cap \Gp(\ZZ)$. Then $\rmS'(\RR) \backslash \Gp(\RR) \cong
\RR^2 \setminus \{0\}$ as $\Gp(\RR)$ spaces, and we may set for any $f: \RR^2 \to \RR$
$$\wht{f}_M(g,w_1,w_2) \= f\big((\tfrac1M,0) \cdot (g,w_1,w_2)\big)
\=  f\big((\tfrac1M,0)g + (w_1,w_2)\big)\,.$$ 
This function is left-invariant under $\rmS'(\RR)$, which contains the lower
triangular subgoup $\rmL(\RR) \subset \rmG(\RR)$. Since the $\Gp(\ZZ)$-orbit
of~$(\tfrac1M,0)$ is obviously~$\tfrac1M\ZZ^2$, we conclude that 
\ba
\SVrelM(f)(g,w_1,w_2) &\=
\sum_{m,n \in \ZZ} f\Big((w_1,w_2)+ \mfrac1M \big(m(a,b) + n(c,d)\big) \Big) \\
&\= \sum_{\gamma \in  \rmS'(\ZZ)\backslash \Gp(\ZZ)}\, \, 
\wht{f}_M(\gamma \cdot (g,w_1,w_2)) 
\tx{,}
\qquad g = \begin{psmatrix} a & b \\ c & d \end{psmatrix}
\tx{,}
\ea
\end{remark}
\par
\begin{proposition}\label{prop:config} For any configuration~$\cC$ and any compactly supported
function~$f$, the Siegel--Veech transform $\SV_\cC(f) \in\rmL^2(\cH(\alpha))$.
\par
For any configuration~$\cC$ there is a constant, the \emph{Siegel--Veech
constant} $c_\cC$, depending on the stratum~$\cH(\alpha)$ such that
$$\int_{\cH(\alpha)} \SV_\cC(f) \,\rmd\mu \= c_\cC \int_{\RR^2} f(x,y) \,\rmd x \rmd y $$
for any compactly supported~$f$. In particular~$\SV_\cC$ is a bounded
linear operator.
\par
For the stratum~$\cH(0,0)$ the equivariance
\be \label{eq:SVequiv}
\SVrel(g'\cdot f) \= g' \cdot \SVrel(f) 
\ee
holds for any $g' \in G'(\RR)$, where $g'$ acts on~$\RR^2$ affine-linearly.
\end{proposition}
\par
\begin{proof} The first statement is the main result of \cite{ACM}, but can be proven for $\cH(0,0)$ directly. The second result is the main
result of Veech in \cite{VeechSiegel}. The third follows from direct
computation or from Remark~\ref{rem:relSVaspushpull}, observing that
the Siegel--Veech transform is a sum over left cosets and that $g' \in \Gp(\RR)$
acts on function on the left by acting on the variable on the right.
\end{proof}
\par

\subsection{Casimir elements acting on the Euclidean plane}
\label{ssec:casimir_operators_euclidean_plan}

The continuity and $\Gp(\RR)$-equivariance~\eqref{eq:SVequiv} imply
that for any $X \in \frakgp$ the Lie derivative of the action functions~$f$
on~$\RR^2$
\be
Xf \defeqwd \lim_{t \to 0} \frac1t (e^{tX}f - f)
\ee
has the property that $X \SV_\cC(f) = \SV_\cC(Xf)$. We compute this
action explicity for the Casimir operators. We work on $\RR^2$ with
coordinates~$(w_1,w_2)$ and use the differential operators
$D_{w_i} = w_i \tfrac \partial {\partial w_i}$.
\par
\begin{lemma}%
\label{la:casimir_on_plane}
The differential operators
\bes
8\CasFolEucl \= D^2_{w_1} + D_{w_1}D_{w_2} +D_{w_2}D_{w_1} + D^2_{w_2}
+ 2D_{w_1} + 2D_{w_2} \quad\tx{and}\quad
\CasTotEucl \= 0
\ees
on~$\RR^2$ have the property that
\be
\SV_\cC(\CasFolEucl f) \= \CasFol \SV_\cC f
\quad\tx{and}\quad
\SV_\cC(\CasTotEucl f) \= \CasTot \SV_\cC f
\ee
for any smooth compactly supported~$f \defcol \RR^2 \ra \CC$
and any configuration~$\cC$.
\end{lemma}
\par
\begin{proof} Direct computations give the Lie derivative action
of the standard generators (see Section~\ref{sec:diff_op}) on such functions,
namely
\bas
P f &\= \frac \partial {\partial w_1} f, \qquad
Q f \= \frac \partial {\partial w_2} f, \quad
&H f &\= \Bigl( w_1 \frac \partial {\partial w_1} + w_2 \frac \partial
{\partial w_2} \Bigr) f, \\
(F+G) f&\=  \Bigl(w_2  \frac \partial {\partial w_1} + w_1
\frac \partial {\partial w_2} \Bigr) f,
\quad
&(F-G) f &\=  \Bigl(-w_2  \frac \partial {\partial w_1} + w_1
\frac \partial {\partial w_2} \Bigr) f\,.
\eas
The claim follows by combining using~\eqref{eq:frakgp_basis}
the expressions~\eqref{eq:casimir_element_sl2}
and~\eqref{eq:casimir_element_saff2} for~$C$ and~$C'$
in these generators.
\end{proof}

\subsection{The representation generated by Siegel--Veech transforms}%

Theorem~\ref{intro:SVupperbound} is a consequence of the following
proposition together with Proposition~\ref{prop:cuspforms}.
\par
\begin{proposition}%
\label{prop:siegel_veech_saff_representation}
For the relative period Siegel--Veech transforms of a mean-zero compactly
supported function~$f \neq 0$, there is a multiplicity~$m \in \ZZ_{\ge 0}
\cup \{\aleph_0\}$ (depending on~$M$) such
that the representation it generates is
\begin{gather*}
  \piLsq\big( \SVrelM(f) \big)
\congwd
  m\, \pip_0 \quad \in \rmL^2\big( \cH(0,0) \big)
\end{gather*}
where $\pip_0$ is the representation from~\eqref{eq:def:genuine_irrep} with
index $n=0$.
\end{proposition}
\par
\begin{proof} By Lemma~\ref{la:casimir_on_plane} the Casimir element
of~$\Gp(\RR)$ acts trivially on the representation~$\piLsq\big( \SVrelM(f)
\big)$. The classification of representations of~$\Gp(\RR)$ shows that we
have a direct sum decomposition
\begin{gather*}
  \piLsq\big( \SVrelM(f) \big)
\congwd
  m\, \pip_0
  \,\oplus\,
  \pi^\G
\end{gather*}
for a nonnegative integer~$m$ and the pullback~$\pi^\G$ of a~$\G(\RR)$
representation. Consider the averaging map $\av: \rmL^2(\Gp(\ZZ) \backslash
\Gp(\RR)) \to \rmL^2 (\rmG(\ZZ) \backslash \rmG(\RR))$, given by the 
integral along the torus~$\Hp(\ZZ) \backslash \Hp(\RR)$. When applied to
the right hand side, the averaging yields~$\pi^\G$. When applied to
a Siegel--Veech transform we combine the summation over the period
lattice with the integral over a fundamental parallelogram to obtain the $\RR^2$-integral
of~$f$, which is zero by hypothesis. Hence~$\pi^\G$ is zero.
\end{proof}

\subsection{Fourier--Heisenberg coefficients}%
\label{ssec:siegel_veech_fourier_coefficients}

We determine some Fourier--Heisenberg coefficients of Siegel--Veech transforms as
preparation for Theorem~\ref{intro:SVperp}.
Suppose $f = f_0(r)\exp(i k \theta)$ is of $\rmK$-type~$k$. Then using
Lemma~\ref{la:modular_to_automorphic} we may view the Siegel--Veech
transform as a function
on~$\bbH'$, writing abusively $\SVrelM(f)(\tau,z)$ to indicate this,
which is affine modular-invariant of weight~$k$ and whose
lift~\eqref{eq:def:modular_to_automorphic}
is the honest Siegel--Veech transform~$\SVrelM(f)$ on~$\cH(0,0)$.
\par
The first statement will be used to conclude that they are orthogonal
to cusp forms.
\par
\begin{proposition}%
  \label{prop:siegel_veech_vanishing_fourier_coefficient}
  Let~$f \defcol \RR^2 \ra \CC$ be a $\rmK$\nbd{}isotypical Schwartz function of
  $\rmK$-type~$k$. Then, for any $M \in \NN$, the $c^{\rmH0}$-Fourier coefficients  of the $M$-relative
Siegel-Veech transforms vanish, i.e.\@
\bes
 c^{\rmH0}\big(\SVrelM(f);\, n,m;\, y \big)\= 0
\ees
for any~$m \in \ZZ$ and any $n \in \ZZ \setminus \{0\}$,
where $\tau = x+iy$ as usual.
\end{proposition}
\par
\begin{proof} 
    From Lemma
    \ref{la:affine_modular_form_fourier_expansion_torus_and_heisenberg}, we want
    to show that the constant term in the Fourier expansion with respect to~$u$
    (which gives the sum of the $c^{\rmH0}$-terms)
    is independent of~$x$, so that only the constant term remains. We may view
    the Siegel--Veech transform of a function as an affine
modular-invariant function of weight~$k$ on~$\HS'$. Explicitly
\be \label{eq:SVexplicit}
\SVrelM(f)(\tau, z) \=
  \sum_{a, b \in \ZZ}
  f\Bigg(\frac{1}{\sqrt{y}}\Big(u + iv + \frac{a(x + iy) + b}{M}\Big) \Bigg)
\ee
so that, unfolding in $b$,
\be \label{eq:SVdu}
\int_{0}^1 \SVrelM(f)(\tau, z)  \de u \= \int_\RR \sum_{a \in \ZZ}\, \sum_{b =
1}^M
f\Bigg(\frac{1}{\sqrt{y}}\Big(u + iv + \frac{a(x + iy) + b}{M} \Big)\Bigg) \de u\,.
\ee
Now it is clear that for any fixed~$v$ and~$y$ translating~$x$ does not change this integral.
\end{proof}
\par
The second proposition will later help us to show that enough Fourier--Heisenberg
coefficients can be controlled by Siegel--Veech transforms, and that they
thus span the space of Eisenstein series. Before stating the proposition we
require a few definitions. One of the many ways
to define \emph{Bessel functions} for integer index~$k$ is
via the \emph{Hansen--Bessel integral formula} \cite[Formula 8.411.1]{gradryzh}:
    \begin{equation}
        J_k(z) \defeqwd \frac{1}{2\pi} \int_{-\pi}^\pi \exp(- i k \theta + i z \sin
        \theta) \,\rd \theta\,.
    \end{equation}
\par
\begin{definition}
 The \emph{Hankel transform of order $k \in \ZZ$} is the integral operator
 defined on functions $f_0 : \RR^+ \to \CC$ as 
 \begin{equation}
     \big( \mathcal H_k f_0 \big) (s) \defeqwd \int_0^\infty f_0(r) J_k(sr) \, r \,\de r, \qquad s \ge
     0.
\end{equation}
\end{definition}
\par
While we would denote $\cH_k$ any realisation of the Hankel transform
from one function space to another, we make the observation that~$\cH_k$ is
an isometric involution on $\RL^2(\RR^+, r \rd r)$, in the sense that it is
norm preserving and that $\cH_k^{-1} = \cH_k$. This can be deduced immediately
from the orthogonality relation enjoyed by Bessel functions
\cite[Formula~6.512.8]{gradryzh}:
 \begin{equation}
     \int_0^\infty  J_k(sr) J_k(tr) \, r \,\rd r \= s^{-1} \delta(s - t)\,,
 \end{equation}
 where $\delta$ is the Dirac delta distribution.
\par
For $j \in \ZZ\setminus \{0\}$, define the isometry $T_j :\RL^2(\RR^+,r \rd r)
\to \RL^2(\RR^+, y^{-3} \rd y)$ and its inverse $S_j$ by
\be
T_jh(y) \=  yh\Big(\frac{2\pi j}{\sqrt{y}}\Big)
\quad \text{and} \quad 
S_jh(r) \= \frac{r^2}{(2\pi j)^2}\, h\Big(\frac{(2\pi j)^2}{r^2}\Big)\,.
\ee
\par
\begin{proposition}%
  \label{prop:siegel_veech_fourier_coefficient_span}
For every $m \in \ZZ \setminus \{0\}$, $k \in \ZZ$, $M \in \NN$ and $f_0 \in
\RL^2(\RR^+,r \rd r)$, the function $f = f_0(r) \exp(i k \theta)$ of
$\rmK$-type~$k$ has Fourier coefficients
\bes
  c^{\rmH0}\big(\SVrelM(f);\, 0,m M;\, y \big)
\=
  (mM)^2\, \big( T_{M} \cH_k f_0 \big) \big(\frac{y}{m^2} \big)
  \in \RL^2(\RR^+,y^{-3} \rd y) \,,
\ees
and every other Fourier coefficient vanishes.
\par
Conversely, given $h \in
\RL^2(\RR^+,y^{-3} \rd y)$, the $M$-relative Siegel--Veech transform of the
function $\wt f = M^{-2}\, (\cH_k S_{M}h) \exp(i k
\theta)$ of\/~$\rmK$-type~$k$ has $h$ as its Fourier coefficients, that is
\bes
c^{\rmH0}\big(\SVrelM(\wt f);\, 0,m M;\, y \big)\= m^2 h(m^{-2} y) \,.
\ees
\end{proposition}
\par
\begin{proof}
Let $\wt{m} \in \ZZ \setminus \{0\}$. We compute, starting
with~\eqref{eq:SVdu} that the coefficient $c_{\wt{m}} =
c^{\rmH0}\big(\SVrelM(f);\, 0,\wt{m};\, y \big)$ equals
\bas
c_{\wt{m}} &\= \int_0^y \int_0^1  \int_\RR
\sum_{a \in \ZZ}\,\sum_{b=1}^M
f\bigg(\frac{u + i v}{y^{1/2}} + \frac{a(x+iy) + b}{y^{1/2}M} \bigg)
\de u\, \de x \, e(-\wt m\frac{v}{y})\de v  \\
  \text{($x$-invariance)}  &\= \int_0^y  \int_\RR \sum_{a \in \ZZ}\,
  \sum_{b =1}^M
f\bigg( \frac{i(Mv +ay)}{y^{1/2}M} + \frac {uM +b}{y^{1/2}M}
\bigg) \de u\,
  \exp(-2\pi i \wt{m}\frac{v}{y})\de v \\
  \text{(unfolding in $a$)}  &\= \sum_{a,b=1}^M \int_\RR  \int_\RR  
  f\bigg(y^{-1/2}\Big(u + i(v + \frac{ay}{M}) + \frac b M \Big) \bigg)\,
  e(-\wt{m}\frac{v}{y}) \de u\,\de v  \\
  \quad &\= M \sum_{a = 1}^M e\big(-\frac{a\wt{m}}{M} \big) \int_\RR  \int_\RR 
  f\big(\wt{u} + i\wt{v} \big)\,  e\big(-\wt{m}\frac{\wt{v}}{\sqrt{y}}\big) 
y \de \wt{u}\, \de \wt{v} \,,
\eas
where we set $\wt{u} = \frac{u + b/M}{\sqrt y}$ and $\wt{v} = \frac{v + ay/M}
{\sqrt{y}}$.
At this point, we see that the integrals are independent of $a$ and the sum
vanishes whenever $\wt{m} \not \in M\ZZ$, otherwise the sum is equal to $M$. We
therefore continue, assuming that $\wt{m} = m M \in M\ZZ$ and changing to polar coordinates
to obtain
\bas 
c^{\rmH0}\big(\SVrelM(f);\,0,m M;\,y\big) &\= y M^2 \int_0^\infty \int_{-\pi}^\pi
f_0(r) \exp(i k \theta
 - 2 \pi i mM \frac{r \sin \theta}{\sqrt{y}}) \de \theta \, r \de  r \\
\text{(Hansen--Bessel formula)}       &\=  y M^2 \int_0^\infty f_0(r)
J_k(2\pi mM y^{-1/2} r) r \rd r \\
&\=  (mM)^2\, \big( T_{M} \cH_k f_0 \big) (m^{-2}y) \in \RL^2(\RR^+,y^{-3}\rd y). \\
\eas
For the converse statement, given $h \in \RL^2(\RR^+,s^{-3} \rd s)$ apply the
previous reasoning to $f_0 = M^{-2} \cH_k S_M h \in \RL^2(\RR^+,r \rd r)$ and $f =
f_0 \exp(i k \theta)$ to obtain the desired identity in the end.
\end{proof}

\subsection{Orthogonality to cusp forms}

Recall from the introduction that we want to prove that the closure
$\CSVrelinf = \clinspan(\cup_{M=1}^\infty \CSVrelM)$ of the union of the spaces 
\begin{equation}
    \CSVrelM \= \clinspan \Big\{\SVrelM(f) : f \in \RC_{c,0}^\infty(\RR^2)\Big\}
\end{equation}
fills the orthogonal complement of cusp forms.
\par
\begin{proof}[Proof of Theorem~\ref{intro:SVperp}]
To show orthogonality it suffices to show orthogonality to all Siegel--Veech
transforms of fixed $\rmK$-type~$k$. We may thus decompose the cusp form
also in~$\rmK$-types and it suffices to show orthogonality of the component
of type~$k$. We may thus work on $\Gamma'\backslash \bbH'$ by the
correspondence in Lemma~\ref{la:modular_to_automorphic}. There we use the
expression for the scalar product in Lemma~\ref{lemma:SPviaFourier}.
Each of these summands under the integral vanishes, either by
Proposition~\ref{prop:siegel_veech_vanishing_fourier_coefficient} or by
definition of a cusp form.
\par
Let now $\varphi \perp \RL^2(\cH(0,0))^{\gen}_{\cusp} \oplus \CSVrelinf$, we
need to show that $\varphi = 0$. Without loss of generality, we assume by
density that~$\varphi$ is smooth and has compact support in the~$y$ variable. 
Then again Lemma~\ref{lemma:SPviaFourier} shows by
Proposition~\ref{prop:siegel_veech_vanishing_fourier_coefficient}
that for every~$M$ and every $f \in \RC^\infty_{c,0}(\RR^2)$ of $\rmK$-type~$k$
\bes
0 \= \int_{\RR^+} \sum_{\ell \geq 1} c^{\rmH0}\big(\SVrelM(f);\, 0, \ell;\, y \big)\,
c^{\rmH0}\big(\varphi;\, 0,\ell;\, y \big) \frac{\de y}{y^{2-k}}\,
\ees
By Proposition~\ref{prop:siegel_veech_fourier_coefficient_span} this implies
that for every $M \in \NN$ and any $h: (0,\infty) \to \CC$ smooth and compactly supported
\be \label{eq:orthogonality_with_SVrelM}
0 \= \int_{\RR_+}  \sum_{\ell \geq 1} 
\ell^2 h(\ell^{-2} y) \,
c^{\rmH0}\big(\varphi;\, 0,\ell M;\, y \big)\, \frac{\de y}{y^{2-k}}\,.
\ee
In particular, since $\varphi$ and $h$ are compactly supported in the
$y$ variable, this is a finite sum and we do not have to worry about
convergence, let $L$ be the largest index in the sum.
\par
Towards a contradiction, suppose there were some $M \in \NN$ and some $h \in
\RL^2(\RR^+,s^{-3} \rd s)$ so that
\bes
0 \newd \int_{\RR^+} h(y) c^{\rmH0}(\varphi;0,M;y) \,\frac{\de y}{y^{2-k}}\,,
\ees
without loss of generality assume that it is equal to $1$. But then, it follows
from \eqref{eq:orthogonality_with_SVrelM} that
\begin{equation}
    \label{eq:contradicted_orthogonality_with_SVrelM}
    -1 \=  \int_{\RR^+} \sum_{\ell = 2}^L
\ell^2 h(\ell^{-2} y) \,
c^{\rmH0}\big(\varphi;\, 0,\ell M;\, y \big)\, \frac{\de y}{y^{2-k}}\,.
\end{equation}
However, using again \eqref{eq:orthogonality_with_SVrelM} with $2M$ replacing $M$, and $h$
replaced with
$\wt h(y) = 4h(y/4)$ we have
that
\bas
0 &\=  \int_{\RR^+} \sum_{\ell = 1}^{\lfloor L/2 \rfloor}
\ell^2 \wt{h}(\ell^{-2} y) \,
c^{\rmH0}\big(\varphi;\, 0,2 \ell M;\, y \big)\, \frac{\de y}{y^{2-k}}\, \\ &\= 
\int_{\RR^+} \sum_{\substack{2 \le \ell \le L \\ 2\mid \ell}}
\ell^2 h(\ell^{-2} y) \,
c^{\rmH0}\big(\varphi;\, 0, \ell M;\, y \big)\, \frac{\de y}{y^{2-k}}\,.
\eas
so that \eqref{eq:contradicted_orthogonality_with_SVrelM} can be rewritten as
\bes
    -1 \= \int_{\RR^+} \sum_{\substack{2 \le \ell \le L \\ 2 \nmid \ell}} 
\ell^2 h (\ell^{-2} y) \,
c^{\rmH0}\big(\varphi;\, 0, \ell M;\, y \big)\, \frac{\de y}{y^{2-k}}\,
\ees
and the other, finitely many, arithmetic progressions can all be sieved out in
the same way so that the righthand side in
\eqref{eq:contradicted_orthogonality_with_SVrelM} is necessarily $0$, a
contradiction. 
\par
By density, we therefore have that necessarily $c^{\rmH0}(\varphi;0,M;y) = 0$ for
all $M \in \NN$, making $\varphi$ a cusp form; yet we also supposed that
$\varphi$ was orthogonal to cusp forms so that $\varphi = 0$.
\end{proof}

\subsection{Kernels, adjoints, and norms of Siegel--Veech transforms}
\label{sec:L2svrelM}

Understanding the Siegel-Veech transform as a linear operator between $\rmL^2$-spaces
comprises determining its range (as we did in the previous section), its
adjoint and its kernel. We address the last two items here. The type of answers
differs even between the cases~$\cH(0)$ and~$\cH(0,0)$, leaving a coherent picture
for general strata as an interesting future problem.
\par
\medskip
\paragraph{\textbf{Adjoints}} Formal adjoints to the Siegel-Veech transform
can be computed using a standard integration trick based on Fubini's theorem.
This is classical for~$\cH(0)$, see e.g.\@ \cite[p.~242]{LangSL}, and can be
adapted to~$\cH(0,0)$ as follows.
\par
\begin{proposition}
The formal adjoint of the relative Siegel-Veech transform is given by
assigning with $h \in \rmL^2(\Gp(\ZZ) \backslash \Gp(\RR))$ the function
\be
\SV^*_\rel(h)(g') \= \int_{\rmS'(\ZZ) \backslash\rmS'(\RR)}\, {h}(s g') \,\rmd\nu(s)\,.
\ee
on $\rmS'(\RR) \backslash \Gp(\RR) \cong \RR^2 \setminus \{0\}$.
\end{proposition}
\begin{proof} We abbreviate $\Gamma' = \Gp(\ZZ)$ and disintegrate the Haar measure
of~$\Gp(\RR)$ as $ \rmd \mu = \rmd\nu(s)\, \rmd\ov{\mu}(g')$
into the Haar measure on~$\rmS'(\RR)$ and the measure~$\ov{\mu}$ on
${\rmS'(\RR) \backslash \Gp(\RR)}$. Now 
\ba
\big\langle \SV_\rel(f),h \big\rangle_{\Gamma' \backslash \Gp(\RR)} &\=
\int_{\Gamma' \backslash \Gp(\RR)}
\sum_{\gamma \in  \rmS'(\ZZ)\backslash \Gamma'}\, \, 
\wht{f}(\gamma g') \, \ov{h}(g') \rmd \mu(g')  \\
(\text{$\Gamma'$-invariance of~$h$}) \quad
&\= \int_{\rmS'(\ZZ) \backslash \Gp(\RR)}\wht{f}( g') \, \ov{h}(g') \,\rmd \mu(g') \\
(\text{$\rmS'(\RR)$-invariance of~$\wht{f}$})
&\= \int_{\rmS'(\RR) \backslash \Gp(\RR)} \wht{f}( g')
\int_{\rmS'(\ZZ) \backslash\rmS'(\RR)}\, \ov{h}(s g') \rmd\nu(s) \,\rmd \ov{\mu}(g')
\\ &\= \big\langle f , \SV^*_\rel(h) \big\rangle_{\rmL^2(\RR^2)}
\ea
verifies the claim.
\end{proof}
\par
\medskip
\paragraph{\textbf{Kernels and norms}}  On $\mathcal H(0) = \rmG(\mathbb
Z^2)\backslash \rmG(\mathbb R^2)$, the Siegel--Veech transform is not an
$\rmL^2$-isometry, since it has obviously a non-trivial kernel consisting of odd
functions. However
the functional equation for Eisenstein series provides more:
\par
\begin{proposition}\label{prop:L2ker}
The kernel of the absolute Siegel--Veech transform on~$\cH(0)$ strictly contains
the odd functions.
\end{proposition}
\par
\begin{proof}
Working formally, putting $k=0$, $\psi(u) = u^s$ in \eqref{eq:eisenstein}, we obtain the classical Eisenstein series $$E(\tau, s) \= \SVabs(h_s)(\Lambda_{\tau}),$$ where $h_s(x) = \|x\|^{-2s}$. Following, for example, Bergeron~\cite[Section~4.1]{Bergeron}, we put $$E^*(\tau, s) \= \pi^{-s} \Gamma(s) \zeta(2s) E(\tau, s)\,.$$ Then the functional equation states
\begin{equation}\label{eq:functional} E^*(\tau, s) \= E^*(\tau, 1-s)\,. \end{equation}
Formally, then, putting $h^*_s(x) = \pi^{-s} \Gamma(s) \zeta(2s)\, \|x\|^{-2s},$ this implies
$$\SVabs\left( h^*_s - h^*_{1-s}\right) {\=} 0\,.$$
To resolve the obvious integrability and convergence issues, we perform a standard trick. We define $-2s = -1+it$, so that $s= \frac{1}{2} - i\frac{t}{2}$ and $1-s = \frac{1}{2} + i \frac{t}{2}$. For a smooth function~$\eta$ of compact support on $\mathbb R^+$, we write~$x$
in polar coordinates as $(r, \theta)$ and obtain the desired kernel elements as
$$f_{\eta}(x) \= \int_{\mathbb R^+} \eta(t) (h_s(x) - h_{1-s}(x)) \,\rmd t \= \int_{\mathbb R^+} \eta(t) (r^{-1+it} - r^{-1-it}) \,\rmd t\,.$$
(This construction can be generalized to functions of other $\rmK$-types by defining $$f_{k, \eta}(r, \theta) \= e^{ik\theta} \int_{\mathbb R^+} \eta(t) (r^{-1+it} - r^{-1-it}) \,\rmd t$$
for nonzero integers~$k$.)
\end{proof}
\par
This is in contrast to the stratum~$\cH(0,0)$:
\par
\begin{proposition}\label{prop:L2svrelM} The $M$-relative Siegel--Veech transform
is $M$ times an isometry on the space of mean zero functions, i.e.\@
$$\big\| \SVrelM(f)\big\|_2 \= M \|f\|_2$$
for $f \in C_c(\mathbb R^2)$ of mean zero. More precisely, for
$f \in C_c(\mathbb R^2)$, we have
$$\int_{\mathcal H(0,0)} \SVrelM(f)\,\rmd\mu \= M^2 \int_{\mathbb R^2} f(x) \,\rmd x$$ and
 \begin{align*}\int_{\mathcal H(0,0)} \SVrelM(f)^2 \,\rmd\mu &\= M^4 \left(\int_{\mathbb R^2} f(x) \,\rmd x\right)^2 + M^2 \int_{\mathbb R^2} f(x)^2 \,\rmd x.\end{align*}
\end{proposition}
\par
\begin{proof} By the equivariance~\eqref{eq:SVequiv}, the map $$f \mapsto
    \int_{\mathcal H(0,0)} \SVrelM(f) \,\rmd\mu$$ is a $\Gp(\mathbb
    R^2)$-invariant functional on $C_c(\mathbb R^2)$, and so we must have $$
    \int_{\mathcal H(0,0)} \SVrelM(f) \,\rmd\mu \= c_M \int_{\mathbb R^2}f(x)
    \,\rmd x,$$ since the only $\Gp(\mathbb R^2)$-invariant measure on $\mathbb
    R^2$ is Lebesgue measure. To find the constant $c_M$, note that if we take
    $f = \chi_{B(0, R)}$ to be the indicator function of the ball of radius $R$,
    with $R$ sufficiently large,
the Siegel--Veech transform~$\SVrelM(f)$ will be approximately constant, with
value $M^2\pi R^2$, so $c_M = M^2$. For the $\rmL^2$ computation, we consider
the configuration $\mathcal C_M^2 \subset \mathbb R^2 \times \mathbb R^2$, and
for $h \in C_c(\mathbb R^2 \times \mathbb R^2)$, we define (by abuse of
notation) $\SVrelM(h)$ as the sum over $\mathcal C_M^2$. By the same proof as
for $M=1$ (see~\cite{Athreya} for further details), the Siegel--Veech transform
$\SVrelM(h) \in \rmL^1(\mathcal H(0,0))$. Consequently, by the equivariance
\eqref{eq:SVequiv} the map $$h \mapsto \int_{\mathcal H(0,0)} \SVrelM(h)
\,\rmd\mu$$ is a $\Gp(\mathbb R^2)$-invariant functional on $C_c(\mathbb R^2
\times \mathbb R^2)$, and so we must have $$ \int_{\mathcal H(0,0)} \SVrelM(h)
\,\rmd\mu \= a_M \int_{\mathbb R^2 \times \mathbb R^2}h(x,y) \,\rmd x \rmd y +
b_M \int_{\mathbb R^2} h(x, x) \,\rmd x,$$ since the only $\Gp(\mathbb
R^2)$-invariant measures on $\mathbb R^2 \times \mathbb R^2$ are Lebesgue
measure and the measure supported on the diagonal $\Delta$.  A similar argument
to above shows that with $h(x, y) = \chi_{B(0, R)}(x) \chi_{B(0, R)}(y),$ for $R
\gg1$, that $a_M = M^4$, and $b_M = M^2$.
\end{proof}

\printbibliography
\end{document}